\documentclass[a4paper,twoside]{article}
\usepackage{a4}
\usepackage{amssymb}
\usepackage{amsmath}
\usepackage{eepic}
\usepackage{upref}
\usepackage[active]{srcltx}
\allowdisplaybreaks[2] 
\usepackage[dvips,colorlinks,citecolor=blue,linkcolor=blue]{hyperref}
\newcount\minutes \newcount\hours
\hours=\time
\divide\hours 60
\minutes=\hours
\multiply\minutes -60
\advance\minutes \time
\newcommand{\klockan}{\the\hours:{\ifnum\minutes<10 0\fi}\the\minutes}
\newcommand{\tid}{\today\ \klockan}
\newcommand{\prtid}{\smash{\raise 10mm \hbox{\LaTeX ed \tid}}}
\renewcommand{\prtid}{}
\makeatletter
\def\sectionmark#1{} 
\def\subsectionmark#1{}
\newcommand{\sectnr}{\ifnum \c@secnumdepth >\z@
                 \thesection.\hskip 1em\relax \fi}
\def\@evenhead{\footnotesize\rm\thepage\hfil\leftmark\hfil\llap{\prtid}}
\def\@oddhead{\footnotesize\rm\rlap{\prtid}\hfil\rightmark\hfil\thepage}
\def\tableofcontents{\section*{Contents} 
 \@starttoc{toc}}
\makeatother
\makeatletter
\def\@biblabel#1{#1.}
\makeatother
\makeatletter
\let\Thebibliography=\thebibliography
\renewcommand{\thebibliography}[1]{\def\@mkboth##1##2{}\Thebibliography{#1}
\addcontentsline{toc}{section}{References}
\frenchspacing 
\setlength{\@topsep}{0pt}
\setlength{\itemsep}{0pt}%
\setlength{\parskip}{0pt plus 2pt}%
}
\makeatother
\makeatletter
\def\mdots@{\mathinner.\nonscript\!.%
 \ifx\next,.\else\ifx\next;.\else\ifx\next..\else
 \nonscript\!\mathinner.\fi\fi\fi}
\let\ldots\mdots@
\let\cdots\mdots@
\let\dotso\mdots@
\let\dotsb\mdots@
\let\dotsm\mdots@
\let\dotsc\mdots@
\def\vdots{\vbox{\baselineskip2.8\p@ \lineskiplimit\z@
    \kern6\p@\hbox{.}\hbox{.}\hbox{.}\kern3\p@}}
\def\ddots{\mathinner{\mkern1mu\raise8.6\p@\vbox{\kern7\p@\hbox{.}}%
    \raise5.8\p@\hbox{.}\raise3\p@\hbox{.}\mkern1mu}}
\makeatother
\makeatletter
\let\Enumerate=\enumerate
\renewcommand{\enumerate}{\Enumerate%
\setlength{\@topsep}{0pt}
\setlength{\itemsep}{0pt}%
\setlength{\parskip}{0pt plus 1pt}%
\renewcommand{\theenumi}{\textup{(\alph{enumi})}}%
\renewcommand{\labelenumi}{\theenumi}%
}
\let\endEnumerate=\endenumerate
\renewcommand{\endenumerate}{\endEnumerate\unskip}
\makeatother
\makeatletter
\def\@seccntformat#1{\csname the#1\endcsname.\quad}
\makeatother
\makeatletter
\long\def\@makecaption#1#2{%
  \vskip\abovecaptionskip
  \sbox\@tempboxa{ #1. #2}%
  \ifdim \wd\@tempboxa >\hsize
    #1. #2\par
  \else
    \global \@minipagefalse
    \hb@xt@\hsize{\hfil\box\@tempboxa\hfil}%
  \fi
  \vskip\belowcaptionskip}
\makeatother
\newcommand{\authortitle}[2]{\author{#1}\title{#2}\markboth{#1}{#2}}
\newcommand{\art}[6]{{\sc #1, \rm #2, \it #3\/ \bf #4 \rm (#5), \mbox{#6}.}}
\newcommand{\auth}[2]{{#1, #2.}}
\newcommand{\artprep}[3]{{\sc #1, \rm #2, #3.}}
\newcommand{\artin}[3]{{\sc #1, \rm #2, in #3.}}
\newcommand{\book}[3]{{\sc #1, \it #2, \rm #3.}}
\newcommand{\AND}{{\rm and }}
\RequirePackage{amsthm}
\newtheoremstyle{descriptive}%
  {\topsep}   
  {\topsep}   
  {\rmfamily} 
  {}          
  {\bfseries} 
  {.}         
  { }         
  {}          
\newtheoremstyle{propositional}%
  {\topsep}   
  {\topsep}   
  {\itshape}  
  {}          
  {\bfseries} 
  {.}         
  { }         
  {}          
\newtheoremstyle{remarkstyle}%
  {\topsep}   
  {\topsep}   
  {\rmfamily}  
  {}          
  {\itshape} 
  {.}         
  { }         
  {}          
\theoremstyle{propositional}
\newtheorem{thm}{Theorem}[section]
\newtheorem{prop}[thm]{Proposition}
\newtheorem{lem}[thm]{Lemma}
\newtheorem{cor}[thm]{Corollary}
\theoremstyle{descriptive}
\newtheorem{deff}[thm]{Definition}
\newtheorem{example}[thm]{Example}
\newtheorem{remark}[thm]{Remark}
\newtheorem{openprob}[thm]{Open problem}
\makeatletter
\renewenvironment{proof}[1][\proofname]{\par
  \pushQED{\qed}%
  \normalfont 
  \trivlist
  \item[\hskip\labelsep
        \itshape
    #1\@addpunct{.}]\ignorespaces
}{%
  \popQED\endtrivlist\@endpefalse
}
\makeatother
\newcommand{\setm}{\setminus}
\renewcommand{\subsetneq}{\varsubsetneq}
\renewcommand{\emptyset}{\varnothing}
\def\vint{\mathop{\mathchoice%
          {\setbox0\hbox{$\displaystyle\intop$}\kern 0.22\wd0%
           \vcenter{\hrule width 0.6\wd0}\kern -0.82\wd0}%
          {\setbox0\hbox{$\textstyle\intop$}\kern 0.2\wd0%
           \vcenter{\hrule width 0.6\wd0}\kern -0.8\wd0}%
          {\setbox0\hbox{$\scriptstyle\intop$}\kern 0.2\wd0%
           \vcenter{\hrule width 0.6\wd0}\kern -0.8\wd0}%
          {\setbox0\hbox{$\scriptscriptstyle\intop$}\kern 0.2\wd0%
           \vcenter{\hrule width 0.6\wd0}\kern -0.8\wd0}}%
          \mathopen{}\int}
\newcommand{\Cp}{{C_p}}
\newcommand{\CpE}{{C_p^E}}
\newcommand{\grad}{\nabla}
\DeclareMathOperator{\diam}{diam}
\DeclareMathOperator{\capp}{cap}
\newcommand{\cp}{\capp_p}
\DeclareMathOperator{\dist}{dist}
\DeclareMathOperator{\fineint}{fine-int}
\DeclareMathOperator*{\essinf}{ess\,inf}
\DeclareMathOperator*{\esssup}{ess\,sup}
\newcommand{\bdry}{\partial}
\newcommand{\bdy}{\bdry}
\newcommand{\loc}{_{\rm loc}}
{\catcode`p =12 \catcode`t =12 \gdef\eeaa#1pt{#1}}      
\def\accentadjtext#1{\setbox0\hbox{$#1$}\kern   
                \expandafter\eeaa\the\fontdimen1\textfont1 \ht0 }
\def\accentadjscript#1{\setbox0\hbox{$#1$}\kern 
                \expandafter\eeaa\the\fontdimen1\scriptfont1 \ht0 }
\def\accentadjscriptscript#1{\setbox0\hbox{$#1$}\kern   
                \expandafter\eeaa\the\fontdimen1\scriptscriptfont1 \ht0 }
\def\accentadjtextback#1{\setbox0\hbox{$#1$}\kern       
                -\expandafter\eeaa\the\fontdimen1\textfont1 \ht0 }
\def\accentadjscriptback#1{\setbox0\hbox{$#1$}\kern     
                -\expandafter\eeaa\the\fontdimen1\scriptfont1 \ht0 }
\def\accentadjscriptscriptback#1{\setbox0\hbox{$#1$}\kern 
                -\expandafter\eeaa\the\fontdimen1\scriptscriptfont1 \ht0 }
\def\itoverline#1{{\mathsurround0pt\mathchoice
        {\rlap{$\accentadjtext{\displaystyle #1}
                \accentadjtext{\vrule height1.593pt}
                \overline{\phantom{\displaystyle #1}
                \accentadjtextback{\displaystyle #1}}$}{#1}}
        {\rlap{$\accentadjtext{\textstyle #1}
                \accentadjtext{\vrule height1.593pt}
                \overline{\phantom{\textstyle #1}
                \accentadjtextback{\textstyle #1}}$}{#1}}
        {\rlap{$\accentadjscript{\scriptstyle #1}
                \accentadjscript{\vrule height1.593pt}
                \overline{\phantom{\scriptstyle #1}
                \accentadjscriptback{\scriptstyle #1}}$}{#1}}
        {\rlap{$\accentadjscriptscript{\scriptscriptstyle #1}
                \accentadjscriptscript{\vrule height1.593pt}
                \overline{\phantom{\scriptscriptstyle #1}
                \accentadjscriptscriptback{\scriptscriptstyle #1}}$}{#1}}}}

\def\cprime{{\mathsurround0pt$'$}}
\newcommand{\al}{\alpha}
\newcommand{\alp}{\alpha}
\newcommand{\ga}{\gamma}
\newcommand{\gat}{\tilde{\gamma}}
\newcommand{\Ga}{\Gamma}
\newcommand{\dmu}{d\mu}
\newcommand{\de}{\delta}
\newcommand{\eps}{\varepsilon}
\newcommand{\la}{\lambda}
\newcommand{\sig}{\sigma}
\newcommand{\Om}{\Omega}
\renewcommand{\phi}{\varphi}
\newcommand{\p}{{$p\mspace{1mu}$}}   
\newcommand{\N}{\mathbf{N}}
\newcommand{\R}{\mathbf{R}}
\newcommand{\Q}{\mathbf{Q}}
\newcommand{\eR}{{\overline{\R}}}
\newcommand{\Z}{\mathbf{Z}}
\newcommand{\ub}{\bar{u}}
\newcommand{\limminus}{{\mathchoice{\raise.17ex\hbox{$\scriptstyle -$}}
                {\raise.17ex\hbox{$\scriptstyle -$}}
                {\raise.1ex\hbox{$\scriptscriptstyle -$}}
                {\scriptscriptstyle -}}}
\newcommand{\limplus}{{\mathchoice{\raise.17ex\hbox{$\scriptstyle +$}}
                {\raise.17ex\hbox{$\scriptstyle +$}}
                {\raise.1ex\hbox{$\scriptscriptstyle +$}}
                {\scriptscriptstyle +}}}
\newcommand{\limpm}{{\mathchoice{\raise.17ex\hbox{$\scriptstyle \pm$}}
                {\raise.17ex\hbox{$\scriptstyle \pm$}}
                {\raise.1ex\hbox{$\scriptscriptstyle \pm$}}
                {\scriptscriptstyle \pm}}}
\newcommand{\Np}{N^{1,p}}
\newcommand{\Nploc}{N^{1,p}\loc}
\newcommand{\Lploc}{L^{p}\loc}
\DeclareMathOperator{\Mod}{Mod}
\newcommand{\Modp}{{\Mod_p}}
\newcommand{\K}{\mathcal{K}}%
\newcommand{\wt}{\widetilde{w}}
\newcommand{\ut}{\tilde{u}}
\newcommand{\uQm}{u_{Q,d\mu}}
\newcommand{\uQx}{u_{Q,dx}}
\newcommand{\uI}{u_{I,dx_1}}
\newcommand{\gt}{\tilde{g}}
\newcommand{\ft}{\tilde{f}}
\newcommand{\gb}{\bar{g}}
\newcommand{\Ct}{\widetilde{C}}
\newcommand{\Dp}{D^p}
\newcommand{\Dploc}{D^{p}\loc}
\newcommand{\fh}{{\hat{f}}}
\makeatletter
\newcommand{\setcurrentlabel}[1]{\def\@currentlabel{#1}}
\makeatother

\numberwithin{equation}{section}
\newcommand{\imp}{\Rightarrow}
\newenvironment{ack}{\medskip{\it Acknowledgement.}}{}

\begin{document}

\authortitle{Anders Bj\"orn and Jana Bj\"orn}
{Obstacle and Dirichlet problems on arbitrary nonopen sets, and fine topology}
\title{Obstacle and Dirichlet problems on arbitrary nonopen sets
in metric spaces, and fine topology}
\author{
Anders Bj\"orn \\
\it\small Department of Mathematics, Link\"opings universitet, \\
\it\small SE-581 83 Link\"oping, Sweden\/{\rm ;}
\it \small anders.bjorn@liu.se
\\
\\
Jana Bj\"orn \\
\it\small Department of Mathematics, Link\"opings universitet, \\
\it\small SE-581 83 Link\"oping, Sweden\/{\rm ;}
\it \small jana.bjorn@liu.se
}

\date{}
\maketitle

\noindent{\small
{\bf Abstract}. 
We study the double obstacle problem for \p-harmonic functions
on arbitrary bounded nonopen
sets $E$ in quite general metric spaces. 
The Dirichlet and single obstacle problems
are included as special cases.
We obtain Adams' criterion for the solubility of the single obstacle problem
and establish  connections with fine potential theory.
We also study when the minimal \p-weak upper gradient of a function
remains minimal when restricted to a nonopen subset.
Most of the results are new for open $E$ (apart from those which are trivial
in this case) and also on $\R^n$.
} 

\bigskip
\noindent
{\small \emph{Key words and phrases}: 
Adams' criterion, Dirichlet problem,
doubling measure, fine potential theory, 
metric space, minimal upper gradient, nonlinear, 
obstacle problem, 
\p-harmonic, Poincar\'e inequality, potential theory,
upper gradient.
}

\medskip
\noindent
{\small Mathematics Subject Classification (2010): 
Primary: 31E05;
Secondary: 31C40, 31C45, 
35D30, 35J20, 35J25, 35J60, 47J20, 
49J40, 49J52, 49Q20, 
58J05, 58J32.
}

\section{Introduction}

Sobolev spaces $W^{1,p}(\Om)$ are usually defined for open sets
$\Om$, and it may be difficult to use the traditional approach to make
reasonable sense of $W^{1,p}(E)$ for nonopen sets $E$. 
One possibility is to let
$f \in W^{1,p}(E)$ if $f \in W^{1,p}(\Om)$
for some open set $\Om \supset E$ depending on $f$, 
but that defies the purpose of the definition a bit.
A more fruitful approach is to consider Sobolev spaces on finely open sets, 
as in Kilpel\"ainen--Mal\'y~\cite{KilMa92} and Mal\'y--Ziemer~\cite{MaZi}.
This is a part of fine potential theory in $\R^n$, which started in the
linear case by Cartan in 1940 and has been further developed also in the
nonlinear case by various authors. See the notes to Chapter~12 in 
Heinonen--Kilpel\"ainen--Martio~\cite{HeKiMa},
especially for the early nonlinear history.

In the 1990s there was a need for studying  Sobolev spaces
on metric measure spaces without any differentiable structure.
Earlier, Sobolev spaces had been extended
to manifolds, Heisenberg groups and other situations with 
a vector-field differentiable structure.
Haj\l asz~\cite{Haj-PA} was the first to give a definition of Sobolev spaces,
so called Haj\l asz spaces, on general metric spaces, while 
Shanmugalingam~\cite{Sh-rev} and Cheeger~\cite{Cheeg} 
a little later introduced  so-called Newtonian spaces.
We follow Shanmugalingam below 
but Cheeger's definition is more or less equivalent.
Let us point out that we only consider first-order Sobolev spaces in this 
discussion.

Since a measurable subset $E$ of a metric measure space $X$ 
can be considered as a metric measure
space on its own, these new definitions are well suited for defining
Sobolev spaces on arbitrary nonopen measurable sets,
e.g.\ of $\R^n$ and other smooth spaces.

In many situations, in particular on (unweighted) $\R^n$, 
both Haj\l asz and 
Newtonian spaces coincide with the usual Sobolev space, see \cite{Sh-rev}. 
However on general open subsets of $\R^n$ it is only the 
Newtonian space that coincides with the usual Sobolev space. 
The Haj\l asz space is in general smaller and the Haj\l asz
gradient is not local, i.e.\ it need not vanish on sets where 
the function is constant,
see e.g.\ Shanmugalingam~\cite{Sh-rev}, Haj\l asz~\cite{Haj03} 
and the discussion in Appendix~B.1 in
Bj\"orn--Bj\"orn~\cite{BBbook}.
It therefore seems that the Newtonian approach is the most suitable,
e.g.\ for solving partial differential equations
 and variational problems on metric spaces
and general subsets of e.g.\ $\R^n$.
Other advantages of Newtonian spaces are that the equivalence classes 
are up to sets of capacity zero and that                              
all Newtonian functions are 
absolutely continuous on \p-almost all curves.
Under suitable assumptions, they are also finely continuous outside
sets of zero capacity (see J.~Bj\"orn~\cite{JB-pfine} and Korte~\cite{korte08}),
which provides another connection to the fine potential theory mentioned above.

In this paper we study the double obstacle problem on general bounded
measurable subsets of a metric space $X$ with a Borel regular measure $\mu$, 
i.e.\ we minimize the \p-energy functional
\begin{equation}     \label{eq-p-energy}
     \int_E g_{u,E}^p \, d\mu,
\end{equation}
among all functions $u$ lying (up to sets of capacity zero)
between two obstacles 
$\psi_1,\psi_2 : E \to \eR:=[-\infty,\infty]$
and with  prescribed boundary values $f$ from the Newtonian space $\Np(E)$
on $E$.
The Dirichlet problem is included as a special case with $\psi_1\equiv -\infty$
and $\psi_2\equiv \infty$.

Here $g_{u,E}$ is the minimal \p-weak upper gradient of $u$ (with respect to
$E$), which is the metric space counterpart of the (modulus of) the
usual gradient.
It depends on the underlying metric space and it is therefore
important for us to understand when
a restriction of a minimal \p-weak upper gradient from 
the underlying metric space $X$ remains minimal on $E$.
This is studied in Section~\ref{sect-minimal-rest}.
In particular, we show that $g_{u,E}=g_{u,X}$ if $E$ is \p-path almost
open, which in unweighted $\R^n$ holds for all finely open sets $E$.
In that case we have $g_{u,E}=g_{u,X}=|\grad u|$ a.e., 
where $\grad u$ is the distributional gradient of $u$.
An interesting example of this phenomenon on a nowhere dense set
$E\subset[0,1]^n\subset\R^n$ with almost full measure in $[0,1]^n$ is
presented in Examples~\ref{ex-emental} and~\ref{ex-emental-p=n}.

Existence and uniqueness (up to sets of capacity zero) of solutions
to the above $\K_{\psi_1,\psi_2,f}(E)$-obstacle problem associated 
with~\eqref{eq-p-energy} is proved in Section~\ref{sect-obst}.
The assumptions under which these results hold, and possibilities to
relax them, are discussed in Section~\ref{sect-exist,uniq,ex}.
We have made an effort to consider the obstacle problem under least
possible assumptions.
In particular, we do not assume that the measure $\mu$ is doubling
and we only use a very weak version of Poincar\'e inequality, 
which moreover can be further relaxed in many situations.
Note that there are infinite-dimensional spaces with  nondoubling
measures supporting a Poincar\'e inequality, see e.g.\ Rajala~\cite{T-Rajala}.
One existence result that we obtain is the following theorem which
follows from Theorem~\ref{thm-obst-solve-E} 
and Remark~\ref{rmk-existence}.

\begin{thm}  \label{thm-ex-with-Lp}
Let $X$ be an arbitrary metric space,
$E\subset X$ be a bounded measurable set,
whose complement has positive capacity,
and $\psi_1,\psi_2\in L^p(E)$, $p>1$.
If $f\in\Np(E)$ is such that $\K_{\psi_1,\psi_2,f}(E)\ne\emptyset$, then
the $\K_{\psi_1,\psi_2,f}(E)$-obstacle problem is soluble.

Moreover, if the\/ $(p,p)$-Poincar\'e inequality for $\Np_0$
holds on $X$ then the assumption that $\psi_1,\psi_2\in L^p(E)$ can be 
omitted and the solution is unique\/ {\rm(}up to sets of capacity zero\/{\rm)}.
\end{thm}

The $(p,p)$-Poincar\'e inequality for $\Np_0$ holds e.g.\ if there is an 
increasing sequence of balls $B_j$ covering $X$, such that for each
$j=1,2,\ldots$, and all $u\in\Np_0(B_j)$, 
\[
        \int_{B_j} |u|^p \,\dmu  
                \le C_j \int_{B_j} g_u^p \,\dmu.
\]
This is usually easier to verify than the classical Poincar\'e inequality,
see Example~\ref{ex-global-but-not-local-PI}.

Along the way, we also discuss alternative definitions of the obstacle problem
and relations between them.
In particular, we compare our obstacle problem with the obstacle problem
defined by means of the global minimal \p-weak upper gradient $g_{u,X}$ and
with the classical obstacle problem on open sets.
Another novelty here (apart from $E$ being nonopen)
is that we allow $f$ to merely belong to the 
Dirichlet space $\Dp(E)$ of measurable functions with an upper gradient 
in $L^p(E)$.
A useful application of our theory to condenser capacities is given
in Theorem~\ref{thm-condenser}.

In Section~\ref{sect-Adams} we establish Adams' criterion for the
solubility of the single obstacle problem with $\psi_2\equiv\infty$. 
We also show by examples that the situation is much more subtle for the 
double obstacle problem.

A natural question is when all the competing functions in $\K_{\psi_1,\psi_2,f}(E)$
coincide (up to sets of capacity zero). 
In this case they are of course all 
solutions of the obstacle problem. 
This happens e.g.\ if $\Np_0(E)$ is trivial (i.e.\ all functions vanish 
outside  a set of capacity zero).
In Section~\ref{sect-nontriv} we characterize those sets where 
this occurs.
It turns out that this problem  has close connections with 
fine potential theory
and that
\[
     \Np_0(E)=\Np_0(\fineint E).
\]
On (unweighted) $\R^n$, our theory comes
together in an elegant way, which we explain in Section~\ref{sect-Rn}.
In particular, we have the following result, which is a special case
of Theorem~\ref{thm-same-obst-pr}
(in view of the results in Section~\ref{sect-Rn}).

\begin{thm}  \label{thm-E-E0-probl-same}
Let $E\subset\R^n$ be a bounded measurable set and $p>1$.
Assume that  $f\in\Dp(E)$ and that  $\K_{\psi_1,\psi_2,f}(E) \ne \emptyset$.
Then the solutions of the $\K_{\psi_1,\psi_2,f}(E)$-problem
coincide with the solutions of the $\K_{\psi_1,\psi_2,f}(E_0)$-problem,
where $E_0$ is the fine interior of $E$.

Moreover, $g_{u,E_0}=g_{u,E}$ a.e.\ in $E_0$ and
if the Lebesgue measure of $E\setm E_0$ is zero,
then also 
the \p-energies~\eqref{eq-p-energy}
associated with these two problems coincide.

If $f\in\Dp(\Om)$ for some open set $\Om\supset E$, then
$g_{u,E_0}=g_{u,E}=|\grad u|$ a.e.\ in $E_0$ and
the above solutions coincide with the solutions of the 
$\K_{\psi_1',\psi_2',f}(\Om)$-problem, where $\psi_j'=\psi_j$ in $E$ and 
$\psi_j'=f$ on\/ $\Om\setm E$, $j=1,2$.
\end{thm}

These results in turn
justify the earlier studies of finely open sets and
the fine obstacle problem on $\R^n$ in the literature, 
as in Kilpel\"ainen--Mal\'y~\cite{KilMa92} and
Mal\'y--Ziemer~\cite{MaZi}.
We hope to use the results from this paper for further development
of fine potential theory in the setting of Newtonian spaces
on metric spaces (with $\R^n$ as an important
special case).
Fine potential theory in this setting has been studied
by Kinnunen--Latvala~\cite{KiLa03},
J.~Bj\"orn~\cite{JB-pfine} and Korte~\cite{korte08}.

\begin{ack}
We would like to thank Olli Martio for asking us the question
when $\Np_0(E)$ is nontrivial. 
We would also like to thank an anonymous referee of the book
Bj\"orn--Bj\"orn~\cite{BBbook} for
pointing out Adams' criterion in \cite{adams81}.

The  authors were supported by the Swedish Research Council
and belong to the
European Science
Foundation Networking Programme \emph{Harmonic and Complex Analysis and
Applications}
and to the Scandinavian Research Network \emph{Analysis and Application}.
\end{ack}

\section{Notation and preliminaries}
\label{sect-prelim}

We assume throughout the paper that $X=(X,d,\mu)$ is a 
metric space equipped
with a metric $d$ and a  measure $\mu$ such that
\[
     0 <  \mu(B)< \infty
\]
for all balls 
$B=B(x_0,r):=\{x\in X: d(x,x_0)<r\}$ in~$X$
(we make the convention that balls are nonempty and open).
We emphasize that the $\sigma$-algebra on which $\mu$ is defined
is obtained by completion of the Borel $\sigma$-algebra.
We also assume that $1 \le p<\infty$
and that $\Om \subset X$ is a nonempty open set.

The measure $\mu$ is \emph{doubling} if there exists
a constant $C>0$ such that 
\begin{equation*}
        0 < \mu(2B) \le C \mu(B) < \infty
\end{equation*}
for all balls $B\subset X$,
where $\lambda B=B(x_0,\lambda r)$.

A \emph{curve} is a continuous mapping from an interval.
We will only consider curves which are nonconstant, compact and rectifiable.
A curve can thus be parameterized by its arc length $ds$. 

We follow Heinonen and Koskela~\cite{HeKo98} in introducing
upper gradients as follows (they called them
very weak gradients).

\begin{deff} \label{deff-ug}
A nonnegative Borel function $g$ on $X$ is an \emph{upper gradient} 
of an extended real-valued function $f$
on $X$ if for all (nonconstant, compact and rectifiable) curves  
$\gamma: [0,l_{\gamma}] \to X$,
\begin{equation} \label{ug-cond}
        |f(\gamma(0)) - f(\gamma(l_{\gamma}))| \le \int_{\gamma} g\,ds,
\end{equation}
where we make the convention that the left-hand side is $\infty$ 
whenever both terms therein are infinite.
If $g$ is a nonnegative measurable function on $X$
and if (\ref{ug-cond}) holds for \p-almost every curve (see below), 
then $g$ is a \emph{\p-weak upper gradient} of~$f$. 
\end{deff}

Here and in what follows, we say that a property holds for 
\emph{\p-almost every curve}
if it fails only for a curve family $\Ga$ with zero \p-modulus, 
i.e.\ there exists $0\le\rho\in L^p(X)$ such that 
$\int_\ga \rho\,ds=\infty$ for every curve $\ga\in\Ga$.
It is easy to show that 
a countable union of curve families with zero \p-modulus also has
zero \p-modulus.
Moreover, if $\Modp(\Ga)=0$ and $\Ga'$ consists of all curves which have
a subcurve in $\Ga$, then $\Modp(\Ga')=0$.

Note that a \p-weak upper gradient need not be a Borel function,
only measurable.
It is implicitly assumed that $\int_{\gamma} g\,ds$ is
defined (with a value in $[0,\infty]$) for
\p-almost every  curve $\ga$,
although this is in fact a consequence of the measurability,
see Bj\"orn--Bj\"orn~\cite{BBpreprint}, Section~3
(which is not in Bj\"orn--Bj\"orn~\cite{BB}).

The \p-weak upper gradients were introduced in
Koskela--MacManus~\cite{KoMc}.
They also showed that if $g \in \Lploc(X)$ is a \p-weak upper gradient of $f$,
then one can find a sequence $\{g_j\}_{j=1}^\infty$
of upper gradients of $f$ such that $g_j-g \to 0$ in $L^p(X)$.
If $f$ has an upper gradient in $\Lploc(X)$, then
it has a \emph{minimal \p-weak upper gradient} $g_f \in \Lploc(X)$
in the sense
 that for every \p-weak upper gradient $g \in \Lploc(X)$ of $f$ we have
$g_f \le g$ a.e.,
see Shan\-mu\-ga\-lin\-gam~\cite{Sh-harm}
and Haj\l asz~\cite{Haj03}.
The minimal \p-weak upper gradient is well defined
up to an equivalence class in the cone of nonnegative functions in $\Lploc(X)$.

For proofs of various facts in this section 
we refer to Bj\"orn--Bj\"orn~\cite{BBbook}. 
(Some of the references we mention here may not 
provide a proof in the generality considered here, but 
such proofs are given in \cite{BBbook}.)

Note that upper gradients and in particular the minimal 
\p-weak upper gradient strongly
depend on the underlying space.
Any measurable $E\subset X$ can be considered as a metric space on its own,
thus giving rise to upper gradients with respect to $E$.
An upper gradient with respect to $X$ is always an upper gradient with 
respect to $E$, but the converse need not be true, see Example~\ref{ex-R-Q-1}.
We denote the minimal \p-weak upper gradient of $u$ with respect to $E$
by  $g_{u,E}$,
whereas $g_u$ always denotes the 
minimal \p-weak upper gradient
with respect to $X$ (also denoted $g_{u,X}$).

Following Shanmugalingam~\cite{Sh-rev}, 
we define a version of Sobolev spaces on the metric space $X$.

\begin{deff}
The \emph{Newtonian space} on $X$ is 
\[
        \Np (X) = \{u: \|u\|_{\Np(X)} <\infty \},
\]
where
\[
        \|u\|_{\Np(X)} = \biggl( \int_X |u|^p \, \dmu 
                + \int_X g_u^p \, \dmu \biggr)^{1/p},
\]
if $u:X\to\overline{\R}$ is an everywhere defined measurable function having an 
upper gradient
in $\Lploc(X)$.

We also say that an everywhere defined
 measurable function $u$ on $X$ belongs to 
the \emph{Dirichlet space} $\Dp(X)$ if it has an upper gradient in $L^p(X)$.
\end{deff}

The local spaces $\Np\loc(X)$ and $\Dp\loc(X)$ are defined by requiring
that for every $x\in X$ there is a ball $B_{x}\subset X$ such that
$u\in\Np(B_x)$ or $u\in\Dp(B_x)$, respectively.
For a measurable set $E \subset X$, the spaces $\Np(E)$, $\Dp(E)$ 
and the corresponding local spaces are defined
by considering $E$ as a metric space on its own.
Note a subtle point here 
(recall that $X$ is \emph{proper} if all closed and bounded sets are compact): 
If $X$ is not proper, then the above
definition of the local spaces need not be equivalent to requiring
that e.g.\ $u\in\Np(K)$ for all compact $K\subset X$.
(See A.~Bj\"orn--Marola~\cite{BMarola} for a related definition
on noncomplete spaces.)
Note that if $\mu$ is doubling then $X$ is proper 
if and only if it is complete.

The space $\Np(X)/{\sim}$, where  $u \sim v$ if and only if $\|u-v\|_{\Np(X)}=0$,
is a Banach space
and a lattice, see Shan\-mu\-ga\-lin\-gam~\cite{Sh-rev}.
Let us here point out that we assume that
functions in Newtonian and Dirichlet spaces are defined everywhere,
and not just up to an equivalence class in the corresponding function space.
This is needed e.g.\ for the definition of upper gradients to make sense.
Shan\-mu\-ga\-lin\-gam~\cite{Sh-rev}
also showed that every $u \in \Dploc(X)$ is absolutely continuous
on \p-almost every curve $\ga$  in $X$, in the sense that $u \circ \ga$
is a real-valued absolutely continuous function.

If $u, v \in \Dploc(X)$, then their minimal \p-weak upper gradients coincide
a.e.\ in the set $\{x \in X : u(x)=v(x)\}$,
in particular $g_{\min\{u,c\}}=g_u \chi_{\{u < c\}}$ a.e.\  for $c \in \R$.
Moreover, $g_{uv} \le |u|g_v + |v|g_u$. 

\begin{deff}
The
\emph{capacity} of a set $E \subset X$ 
is the number 
\begin{equation*} 
  \Cp (E) =\inf    \|u\|_{\Np(X)}^p,
\end{equation*}
where the infimum is taken over all $u\in \Np (X) $ such that
$u=1$ on $E$.

We say that a property 
holds \emph{quasieverywhere} (q.e.)\ 
if the set of points  for which it fails
has capacity zero.
\end{deff}

This capacity was introduced and used for Newtonian spaces
in Shanmugalingam~\cite{Sh-rev}.
It is countably subadditive and the correct gauge 
for distinguishing between two Newtonian functions. 
If $u \in \Nploc(X)$ and $v:X \to \eR$,
then $u \sim v$ if and only if $u=v$ q.e.
Moreover, 
if $u,v \in \Dploc(X)$ and
$u= v$ a.e., then $u=v$ q.e.
See also Appendix~\ref{app-capp}
where the variational capacity is defined.
Note that if $\Cp(E)=0$, then \p-almost every curve in $X$ avoids $E$,
by e.g.\ Lemma~3.6 in Shanmugalingam~\cite{Sh-rev}
or Proposition~1.48 in Bj\"orn--Bj\"orn~\cite{BBbook}.

To be able to compare the boundary values of Newtonian functions
we need a Newtonian space with zero boundary values.
We let
\[
\Np_0(E)=\{f|_{E} : f \in \Np(X) \text{ and }
        f=0 \text{ on } X \setm E\}.
\]
One can replace the assumption ``$f=0$ on $X \setm E$''
with ``$f=0$ q.e.\ on $X \setm E$''
without changing the obtained space
$\Np_0(E)$.
Functions from $\Np_0(E)$ can be extended by zero q.e.\ in $X\setm E$ and we
will regard them in that sense if needed.
Note that if $\Cp(X \setminus E) = 0$, then $\Np_0(E) = \Np(E) = \Np(X)$,
since \p-almost every curve in $X$ avoids $X\setm E$.

The following lemma is useful for proving that certain functions belong
to $\Np_0(E)$.
For open $E$, it was obtained in Bj\"orn--Bj\"orn~\cite{BB}.
The proof of the general case 
can be found in 
Bj\"orn--Bj\"orn~\cite{BBbook}.

\begin{lem} \label{lem-police}
Assume that $E \subset X$ is measurable. 
Let $u \in \Np(E)$ and $v,w \in \Np_0(E)$ be such that
$v \le u \le w$ q.e.\ in $E$.
Then $u \in \Np_0(E)$.
\end{lem}

The following Poincar\'e inequality is often assumed in the literature.
Because of the dilation $\la$ in the right-hand side, it is sometimes called
weak Poincar\'e inequality.

\begin{deff} \label{def-PI}
We say that $X$ supports a \emph{$(q,p)$-Poincar\'e inequality},
$q \ge 1$, if
there exist constants $C>0$ and $\lambda \ge 1$
such that for all balls $B \subset X$
and all integrable $u\in\Dp\loc(X)$,
\begin{equation} \label{PI-ineq}
        \biggl(\vint_{B} |u-u_B|^q \,\dmu\biggl)^{1/q}
        \le C (\diam B) \biggl( \vint_{\lambda B} g_u^{p} \,\dmu \biggr)^{1/p},
\end{equation}
where $ u_B 
 :=\vint_B u \,\dmu 
:= \int_B u\, d\mu/\mu(B)$.
\end{deff}

Using the above-mentioned results on \p-weak upper gradients from
Koskela--MacManus~\cite{KoMc},
it is easy to see that \eqref{PI-ineq} can equivalently be required for 
all upper gradients $g$ of $u$.
If $X$ supports a $(1,p)$-Poincar\'e inequality and $\mu$ is doubling,
then by Theorem~5.1 in Haj\l asz--Koskela~\cite{HaKo},
it supports a $(q,p)$-Poincar\'e inequality for some $q>p$,
and in particular a $(p,p)$-Poincar\'e inequality.
Moreover, under these assumptions, Lipschitz functions
are dense in $\Np(X)$, see Shan\-mu\-ga\-lin\-gam~\cite{Sh-rev}.
If $X$ is also  complete 
then  functions in $\Np(X)$
as well as in $\Np(\Om)$ are quasicontinuous,
see Bj\"orn--Bj\"orn--Shan\-mu\-ga\-lin\-gam~\cite{BBS5}.
It also follows
that $\Np_0(\Om)$ for open $\Om$
can equivalently be defined as the closure of
Lipschitz functions with compact support in $\Om$,
see Shanmugalingam~\cite{Sh-harm} or Theorem~5.45 in 
Bj\"orn--Bj\"orn~\cite{BBbook}.
For a general set $E$ this is not always possible
and the above definition of $\Np_0(E)$ seems to be the natural one.

Moreover, if $X$ is unweighted $\R^n$ and $u \in \Dploc(X)$, then
$g_u=|\nabla u|$ a.e., where $\nabla u$ is the distributional gradient
of $u$.
This means that in the Euclidean setting, $\Np(\Om)$, $\Om \subset \R^n$,
 is the 
refined Sobolev space as defined on p.\ 96 of
Heinonen--Kilpel\"ainen--Martio~\cite{HeKiMa}.
See Haj\l asz~\cite{Haj03} or 
Appendix~A.1 in \cite{BBbook} for a full
proof of this fact for unweighted $\R^n$, and
Appendix~A.2 in \cite{BBbook} for a proof for weighted $\R^n$
(requiring $p>1$).

For most results in this paper we will need some kind of
Poincar\'e inequality, but it is enough with a considerably
weaker one than the one in Definition~\ref{def-PI}.
Let us therefore introduce the following notion,
which will be useful e.g.\ when proving the 
existence and uniqueness of the solutions of our obstacle problems.
Note that it follows from, but does not imply, the Poincar\'e inequality 
as in Definition~\ref{def-PI}, see Lemma~\ref{lem-PI-NP0}
and Example~\ref{ex-global-but-not-local-PI}.

\begin{deff} \label{def-relax-PI}
We say that $X$ supports a \emph{$(p,p)$-Poincar\'e inequality for $\Np_0$}
if for every bounded $E\subset X$ with $\Cp(X\setm E)>0$
there exists $C_E>0$  such that for all  $u\in\Np_0(E)$
\textup{(}extended by $0$ outside $E$\textup{)},
\begin{equation}
        \int_X |u|^p \,\dmu  
                \le C_E \int_{X} g_u^p \,\dmu.
\label{PI-NP0}
\end{equation}
\end{deff}

A direct consequence is that $\|u\|_{\Np(X)}^p \le \Ct_E \|g_u\|^p_{L^p(X)}$ 
for $u\in\Np_0(E)$.
If $E$ is measurable, then the integrals 
and the norms can equivalently
be taken with respect to $E$.
As in \eqref{PI-ineq}, one can equivalently verify \eqref{PI-NP0} for
all upper gradients $g$ of $u$.
If $X$ is unbounded then the condition $\Cp(X \setm E)>0$ is of course
redundant. 
On the other hand, if $X$ is bounded then it is essential, as otherwise
$1 \in \Np_0(E)$ violates \eqref{PI-NP0}.

We will also need the space 
\[ 
   \Dp_0(E)=\{f|_E : f \in \Dp(X) \text{ and } f=0 \text{ on } X \setm E\}.
\]
As we shall now see, it will for us coincide with $\Np_0(E)$ in most cases, 
and then we prefer to write $\Np_0(E)$.

\begin{prop} \label{prop-Dp0=Np0}
Assume that $X$ supports a\/ $(p,p)$-Poincar\'e inequality for $\Np_0$
and that  $E\subset X$ is bounded and $\Cp(X\setm E)>0$.
Then 
\[
     \Dp_0(E) = \Np_0(E).
\]
\end{prop}

\begin{proof}
Let $u \in \Dp_0(E)$ and extend $u$ by $0$ outside $E$.
Let $g \in L^p(X)$ be an upper gradient of $u$,
and let $u_k=\max\{\min\{u,k\},-k\}$, $k=1,2,\ldots$,
be the truncations of $u$ at levels $\pm k$.
Then $g$ is an upper gradient also of $u_k$.
As $E$ is bounded, $u_k \in L^p(X)$ and thus $u_k \in \Np_0(E)$.
Hence, by monotone convergence and 
the $(p,p)$-Poincar\'e inequality for $\Np_0$,
\[
    \int_X |u|^p \, d\mu 
    = \lim_{k \to \infty}     \int_X |u_k|^p \, d\mu 
    \le C_E  \int_X g^p \, d\mu 
    <\infty.
\]
Thus $u \in \Np(X)$ and hence $u \in \Np_0(E)$.
This proves one inclusion, while the converse inclusion is trivial.
\end{proof}

Finally, we make the convention that,
unless otherwise stated, the letter $C$ denotes various positive constants
whose exact values are not important and may vary with each usage.

\section{Restrictions of minimal 
\texorpdfstring{\p-weak}{p-weak} upper gradients}
\label{sect-minimal-rest}

In the next section, we will define and study the obstacle problem, 
in which we minimize the \p-energy functional \eqref{eq-p-energy} 
on general sets.
Since the energy functional is defined using the minimal \p-weak 
upper gradient, it is natural to study how this notion depends on the
underlying set.
This will be done in this section.
We point out that for this we do not impose any assumptions on $X$,
such as the doubling property of $\mu$ or the Poincar\'e inequality.

If $\Om$ is open and $f \in \Dploc(X)$ then 
the minimal \p-weak upper gradient of $f$
with respect to $X$ remains minimal when restricted to $\Om$, i.e.\ 
with respect to $\Dploc(\Om)$.
This is folklore but the interested reader can find a proof in 
Bj\"orn--Bj\"orn~\cite{BBbook}, Lemma~2.23. 
We will need a generalization of this result to \p-path almost open sets,
see Proposition~\ref{prop-gu-on-p-path-open}.

\begin{deff}   \label{def-p-path-open}
The set $G\subset X$ is \p-\emph{path open}
(in $X$) if for \p-almost every curve
$\ga:[0,l_\ga]\to X$, the set $\ga^{-1}(G)$ 
is\/ \textup{(}relatively\/\textup{)} open in $[0,l_\ga]$.

Further, $G \subset X$ is 
\p-\emph{path almost open}
(in $X$) if for \p-almost every curve
$\ga:[0,l_\ga]\to X$, the set $\ga^{-1}(G)$ is the union
of an open set and a set with zero one-dimensional Lebesgue measure.
\end{deff}

The \p-path open sets were introduced by Shanmugalingam~\cite{Sh-harm}, 
Remark~3.5.
The name ``\p-path almost open'' is perhaps a little misleading, as we 
\emph{do not} allow $\ga^{-1}(G)$ to be an open set \emph{minus} a set of measure zero.
For our purposes there are  counterexamples showing that we
cannot allow for this, see Example~\ref{ex-R-Q-1} below.

Clearly, every open set is \p-path open,
and every \p-path open set is \p-path almost open.
The following observation gives some light on which sets are
\p-path (almost) open.

\begin{lem} \label{lem-p-path-almost-open}
Let $E,G \subset X$.

If $G$ is \p-path open and $\Cp(E \setm G)=\Cp(G \setm E)=0$,
then $E$ is \p-path open.

If $G$ is \p-path almost open and $\Cp(G \setm E)=\mu(E \setm G)=0$,
then $E$ is \p-path almost open.
In particular, if $\mu(E \cap \bdy E)=0$, then $E$ is \p-path almost open.
\end{lem}

\begin{proof}
Assume first that $G$ is \p-path open and that
$\Cp(E \setm G)=\Cp(G \setm E)=0$.
Then  \p-almost every curve $\ga$ avoids $(E\setm G)\cup(G\setm E)$
and hence $\ga^{-1}(E)=\ga^{-1}(G)$ is (relatively) open 
for \p-almost every curve $\ga$, 
i.e.\ $E$ is \p-path open.

Assume next that 
$G$ is \p-path almost open
and $\Cp(G \setm E)=\mu(E \setm G)=0$.
Then \p-almost every curve $\ga$ avoids $G\setm E$ and 
is such that $\ga^{-1}(E\setm G)$ has 
zero one-dimensional Lebesgue measure, by e.g.\ Lemma~1.42 
in Bj\"orn--Bj\"orn~\cite{BBbook}.
For all such curves we have $\ga^{-1}(E)=\ga^{-1}(G)\cup\ga^{-1}(E\setm G)$,
i.e.\ $E$ is \p-path almost open.
\end{proof}

\begin{remark} \label{rmk-8}
The collection of all \p-path open sets does not (in general)
form a 
topology on $X$. Consider e.g.\ unweighted $\R^n$ with $n>1$
and $1 \le p \le n$,
in which case 
all singleton sets are \p-path open since they have capacity zero.
If the \p-path open sets formed a topology it would follow that
any set on $\R^n$ would be \p-path open.
However it is quite easy,
using Lemma~A.1 in Bj\"orn--Bj\"orn~\cite{BBbook},  
to see that $\R^{n-1} \times \Q$ is not \p-path open.
If singletons have positive capacity (e.g.\ if $X=\R^n$ and $p>n$), 
then any \p-path open set is open, 
and thus the family of \p-path open sets does form a topology.

Similarly, the \p-path almost open sets do not (in general)
from a topology on $X$.
On $\R^n$, any singleton set is \p-path almost open
(but not \p-path open if $p>n$).
The set $\R^{n-1} \times \Q$ is \p-path almost open,
but $\R^{n-1} \times (\R \setm \Q)$ is not.
Thus, the \p-path almost open sets do not form a topology on $\R^n$.
\end{remark}

If there are no nonconstant rectifiable curves in $X$, as e.g.\
on the von Koch snowflake curve, then all sets are \p-path open, and
thus in this case the \p-path open sets form 
a topology, and so do the \p-path almost open sets. 
This also shows that \p-path open sets need not be measurable.

A consequence of Lemma~\ref{lem-p-path-almost-open} is that 
the union of a \p-path open set and a set of measure zero is 
\p-path almost open. 

\begin{openprob}
Can every \p-path almost open set be written as 
a union of a \p-path open set and a set of measure zero?
\end{openprob}

The following result shows that \p-path almost open sets preserve the minimal
\p-weak upper gradients in the same way as open sets do.
Recall that by $g_u$ we always mean the \p-weak upper gradient
of $u$ with respect to $X$. 

\begin{prop}    \label{prop-gu-on-p-path-open}
Let $G$ be a \p-path almost open measurable set and let $u \in \Dploc(X)$.
Then $g_{u,G}=g_u$ a.e.\ in $G$, 
i.e.\ $g_u|_G$ is a minimal \p-weak upper gradient of $u$ with respect
to $G$.
\end{prop}

Before proving this result it may be worth observing that some condition
on $G$ is really necessary. 

\begin{example}
\label{ex-R-Q-1}
Let $X=\R$ and $E=(0,1)\setm\Q$.
Since $E$ contains no rectifiable curves, the minimal \p-weak upper gradient 
taken with respect to $E$ is zero
for every function on $E$.
On the other hand, the minimal \p-weak upper gradient with respect to 
$\R$ is just the modulus of the distributional derivative.
For example, if $u(x)=x$, then $g_{u}=1$ a.e., while $g_{u,E}=0$ a.e.
Note also that $E$ has full measure in the open interval $I=(0,1)$ for which
$g_{u,I}=g_u=1$ a.e.
\end{example}

\begin{proof}[Proof of Proposition~\ref{prop-gu-on-p-path-open}]
Clearly, $g_{u,G}\le g_u$ a.e.\ in $G$.
We shall show that the function
\[
	g=\begin{cases}
                  g_{u,G} & \text{in } G, \\
                  g_u & \text{in } X\setm G,
           \end{cases}
\]
is a \p-weak upper gradient of $u$ in $X$.
Let $\Ga_0$ consist of all curves $\ga$ in $X$ for which $\ga^{-1}(G)$
is not a union of an open set and a set with zero one-dimensional Lebesgue
measure.
Let also $\Ga_1$  be the collection of all curves in $G$ 
on which \eqref{ug-cond} fails for $u$ and $g_{u,G}$.
Similarly, let $\Ga_2$ consist of all those curves in $X$ 
on which \eqref{ug-cond} fails for $u$ and $g_u$.
Finally, let $\Ga_3$ consist of all those curves in $X$ 
for which $\int_\ga g_u\,ds=\infty$.
By assumptions, we have $\Modp(\Ga_0\cup\Ga_1\cup\Ga_2\cup\Ga_3)=0$.

Let $\ga:[0,l_\ga]\to X$ be a curve having no 
subcurve in $\Ga_0\cup\Ga_1\cup\Ga_2\cup\Ga_3$.
By Lemma~1.34 in Bj\"orn--Bj\"orn~\cite{BBbook},
\p-almost every curve in $X$ has this property.
Then $\ga^{-1}(G)=G' \cup A$, where $G'$ is open in $(0,l_\ga)$ and
$A$ has zero one-dimensional Lebesgue measure.
The set $G'$ can be written 
as a countable union 
$\bigcup_{j=1}^\infty I_j$
of open intervals $I_j=(a_j,b_j)$, $j=1,2,\ldots$\,.
(Here we allow some of the intervals $I_j$ to be empty.)
We then have
\begin{align*}
|u(\ga(0))-u(\ga(l_{\ga}))| 
     & \le |u(\ga(0))-u(\ga(a_1))| 
            + |u(\ga(a_1))-u(\ga(b_1))| \\
     & \quad           + |u(\ga(b_1))-u(\ga(l_{\ga}))| 
      \le \int_{\ga|_{I_1}} g_{u,G}\,ds 
                + \int_{\ga|_{[0,l_\ga]\setm I_1}} g_u\,ds.
\end{align*}
Continuing in this way, we obtain for all $j=1,2,\ldots$,
\[
|u(\ga(0))-u(\ga(l_\ga))| 
      \le  \int_{\ga|_{\bigcup_{i=1}^j I_i}} g_{u,G}\,ds 
            + \int_{\ga|_{[0,l_\ga]\setm \bigcup_{i=1}^j I_i}} g_u\,ds.
\]
Since $\int_\ga g_u\,ds<\infty$, letting $j\to\infty$
and using monotone and dominated convergence show that
\begin{align*}
   |u(\ga(0))-u(\ga(l_{\ga}))| 
      &\le  \int_{\ga|_{G'}} g_{u,G}\,ds 
         + \int_{\ga|_{[0,l_\ga]\setm G'}} g_u\,ds \\
      & =  \int_{\ga|_{G'\cup A}} g_{u,G}\,ds 
        + \int_{\ga|_{[0,l_\ga]\setm (G' \cup A)}} g_u\,ds
    = \int_{\ga} g\,ds.
\end{align*}
Thus, $g$ is a \p-weak upper gradient of $u$ in $X$ and hence
$g_u\le g$ a.e.\ in $X$.
It follows that $g_u\le g_{u,G}$ a.e.\ in $G$, which finishes the proof.
\end{proof}

\begin{cor}  \label{cor-guG-guE}
Let $E\subset X$ be measurable and $G\subset E$ be a \p-path almost open\/
\textup{(}with respect to $X$\textup{)}
measurable set.
If $u\in\Dploc(X)$, then
\[
g_{u,G}=g_{u,E}=g_u \quad \text{a.e.\ in } G.
\]
\end{cor}

\begin{proof}
Clearly, $g_{u,G}\le g_{u,E} \le g_u$ a.e.\ in $G$.
Since $G$ is \p-path almost open, 
Proposition~\ref{prop-gu-on-p-path-open} shows that equality must hold a.e.\ in $G$.
\end{proof}

\begin{remark}
Note that Corollary~\ref{cor-guG-guE} can also be applied to $E$ 
instead of $X$, giving that for $u\in\Dp\loc(E)$, 
$g_{u,G}= g_{u,E}$ a.e.\ in $G$,
whenever $G\subset E$ is measurable and \p-path almost open with 
respect to $E$, in particular if it is measurable and \p-path almost open with 
respect to $X$.
\end{remark}

Another application of \p-path open sets is
the following sufficient condition 
for when a function belongs to $\Np_0(E)$. 
This generalizes Theorem~2.147 and Corollary~2.162 in 
Mal\'y--Ziemer~\cite{MaZi}.
See also Lemma~\ref{lem-police}, Theorem~\ref{thm-fineint}
and Proposition~\ref{prop-fineint-locally}
for related results, and
Proposition~\ref{prop-char-Np0-qe-bdry} where this is 
combined with fine topology on $\R^n$.

\begin{lem}  \label{lem-Np0-0-qe}
Let $E_1\subset E_2\subset X$ with $E_1$ and $X\setm E_2$ being \p-path open.
If $u\in\Np(E_2)$ and $u=0$ q.e.\ in $E_2\setm E_1$ then the zero
extension of $u$ belongs to $\Np(X)$ and in particular $u\in\Np_0(E_1)$.
\end{lem}

Note that ``\p-path open'' in Lemma~\ref{lem-Np0-0-qe} cannot
be replaced by ``\p-path almost open'', as the example with
$E_1=E_2=(0,1)\subset\R=X$ and $u=\chi_{E_1}$ shows.
The most common usage of Lemma~\ref{lem-Np0-0-qe} is perhaps when
$E_1$ and $E_2$ are the interior and the closure of some set, respectively.

\begin{proof}
We shall show that $g_{u,E_2}$ (extended by zero) 
is a \p-weak upper gradient of $u$ (extended by zero) in $X$.
Let $\Ga$ be the family of curves in $E_2$ on which \eqref{ug-cond} fails for
$u$ and $g_{u,E_2}$.
Then $\Modp(\Ga)=0$.
Let also $A=\{x\in E_2\setm E_1: u(x)\ne0\}$.
Since $\Cp(A)=0$, we conclude that \p-almost every curve $\ga:[0,l_\ga]\to X$  
avoids $A$, does not have a subcurve in $\Ga$ and is such that
both $\ga^{-1}(E_1)$ and $\ga^{-1}(X\setm E_2)$ are relatively open.

Let $\ga$ be such a curve.
We can assume that $\ga$ passes through both $E_1$ and $X\setm E_2$.
Otherwise, there is nothing to prove, since $g_{u,E_2}$ is 
a \p-weak upper gradient in $E_2$ and $u=0$ outside $E_1\cup A$.
By splitting $\ga$ into two parts and reversing the orientation, 
if necessary, we can assume that
$\ga(0)\in E_1$ and $\ga(l_\ga)\in X\setm E_2$.

Let $c= \inf\{ t\in[0,l_\ga]:\ga(t)\in X\setm E_2\}$.
Since both $\ga^{-1}(E_1)$ and $\ga^{-1}(X\setm E_2)$ are relatively
open in $[0,l_\ga]$, we conclude that $\ga(c)\in (E_2\setm E_1) \setm A$,
i.e.\ $u(\ga(c))=0=u(\ga(l_\ga))$.
Hence
\begin{align*}
    |u(\ga(0))-u(\ga(l_\ga))| 
    &= |u(\ga(0))-u(\ga(c))|
    \le\int_{\ga|_{[0,c]}} g_{u,E_2}\,ds 
    \le \int_{\ga} g_{u,E_2}\,ds.\qedhere
\end{align*}
\end{proof}

As $\Np_0(E)$ is defined through $\Np(X)$, it is natural that
the minimal \p-weak upper gradients of functions in $\Np_0(E)$
are taken with respect to $X$.
The following result is therefore important for our considerations. 
(This result holds for $u \in \Dp_0(E)$ even in situations when 
$\Np_0(E) \subsetneq \Dp_0(E)$ so we formulate it in this generality.
In fact it even holds for 
$u \in \Dp_{{\rm loc},0}(E):=\{f|_{E} : f \in \Dploc(X) \text{ and }
        f=0 \text{ on } X \setm E\}$.)

\begin{prop} \label{prop-min-grad}
Let $E \subset X$ be measurable and $u \in \Dp_0(E)$ with a minimal 
\p-weak upper gradient $g_u$
\textup{(}with respect to $X$, and with $u=0$ outside $E$\textup{)}.
Then $g_{u,E}=g_u|_E$ a.e.\ in $E$,
i.e.\ $g_u|_E$ is a minimal \p-weak upper gradient of $u$
with respect to $E$.
\end{prop}

Note that the corresponding result for arbitrary $u\in\Np(X)$ is false, 
see Example~\ref{ex-R-Q-1}.

\begin{proof}
Clearly, $g_u|_E$ is a \p-weak upper gradient of $u$
in $E$.  
To show that it is minimal, 
we shall show that the function
\[
	g=\begin{cases}
                  g_{u,E} & \text{in } E, \\
                  0 & \text{in } X\setm E,
           \end{cases}
\]
is a \p-weak upper gradient of $u$ in $X$.
Let $\Ga$ be the set of curves in $E$ on which \eqref{ug-cond} fails for
$u$ and $g_{u,E}$.
Then $\Modp(\Ga)=0$.

Let $\ga:[0,l_\ga]\to X$ be a curve such that $u$ is absolutely
continuous along it and $\ga$ does not have any subcurve in $\Ga$.
As $u\in\Dp_0(E)$ and $\Modp(\Ga)=0$, 
\p-almost every curve in $X$ has these properties.
We can also assume that $\ga$ passes through both $E$ and $X\setm E$.
Otherwise, there is nothing to prove, since $g_{u,E}$ is 
a \p-weak upper gradient in $E$ and $u=0$ outside $E$.
By splitting $\ga$ into two parts and reversing the orientation, 
if necessary, we can assume that
$\ga(0)\in E$ and $\ga(l_\ga)\in X\setm E$.

Let $c= \inf\{ t\in[0,l_\ga]:\ga(t)\in X\setm E\}$.
Since $u=0$ in $X\setm E$, the absolute continuity of $u$ along $\ga$
implies that $u(\ga(c))=0=u(\ga(l_\ga))$.
If $c>0$, then
\begin{align*}
    |u(\ga(0))-u(\ga(l_\ga))| 
    &= |u(\ga(0))-u(\ga(c))|
     = \lim_{t\to c\limminus}  |u(\ga(0))-u(\ga(t))| \\
    &\le \lim_{t\to c\limminus} \int_{\ga|_{[0,t]}} g_{u,E}\,ds 
    \le \int_{\ga} g\,ds,
\end{align*}
by the absolute continuity of $u$ along $\ga$.
For $c=0$, these estimates are trivial.
Thus $g$ is a \p-weak upper gradient of $u$ in $X$, 
and hence $g\ge g_u$ a.e.\ in $X$. 
It follows that $g_{u,E} \le g_u \le g_{u,E}$ a.e.\ in $E$,
which finishes the proof.
\end{proof}

\section{The obstacle problem}
\label{sect-obst}

We shall now consider the obstacle and Dirichlet 
problems on general sets.
Let us start by formulating \emph{our} obstacle problem.

\medskip

\emph{Throughout this section we assume that $p>1$ and that  $X$ 
supports a\/ $(p,p)$-Poincar\'e inequality for $\Np_0$,
see Definition~\ref{def-relax-PI}.
We also assume 
that $E \subset X$ is a bounded measurable set such that $\Cp(X \setm E)>0$.}

\medskip

In Section~\ref{sect-exist,uniq,ex} we will discuss when these assumptions
can be relaxed.
Observe that we do not assume that $\mu$ is doubling nor that $X$ is complete,
although we will need to add these assumptions for parts of the theory 
in Sections~\ref{sect-nontriv} and~\ref{sect-compare}.

\begin{deff} \label{deff-obst-E}
Let $A \subset X$ be a bounded measurable set such that $\Cp(X \setm A)>0$.
Let $f \in \Dp(A)$ and $\psi_1,\psi_2 : A \to \eR$. 
Then we define
\begin{equation*}  
    \K_{\psi_1,\psi_2,f}(A)=\{v \in \Dp(A) : v-f \in \Np_0(A) 
            \text{ and } \psi_1 \le v \le \psi_2 \ \text{q.e. in } A\}.
\end{equation*}

Furthermore, a function $u \in \K_{\psi_1,\psi_2,f}(A)$
is a \emph{solution of the $\K_{\psi_1,\psi_2,f}(A)$-obstacle problem}
if 
\begin{equation} \label{eq-obst-def-E}
       \int_A g^p_{u,A} \, d\mu 
       \le \int_A g^p_{v,A} \, d\mu 
       \quad \text{for all } v \in \K_{\psi_1,\psi_2,f}(A).
\end{equation}
If $A=E$ we often drop the set from the notation and merely write
$\K_{\psi_1,\psi_2,f}:=\K_{\psi_1,\psi_2,f}(E)$.
Similarly, we often drop $\psi_2$ from the notation when $\psi_2 \equiv \infty$,
i.e.\ when there is no upper obstacle.
Such an obstacle problem is called the \emph{single obstacle problem}.
\end{deff}

The Dirichlet problem is a special case of the obstacle
problem, with the obstacles $\psi_1 \equiv -\infty$ and $\psi_2 \equiv\infty$.
Note that the boundary data $f$ are only required to belong to $\Dp(A)$,
i.e.\ $f$ need not be defined on $\bdry A$.

Since we consider boundary values $f \in \Dp(A)$ 
rather than $f\in \Np(A)$,
it would be natural to consider the obstacle problem 
with $\Dp_0(A)$ instead of $\Np_0(A)$.
However, by
Proposition~\ref{prop-Dp0=Np0}, this is exactly what we do, even 
though we prefer to write $\Np_0(A)$.
At the same time, in the more general situations discussed in 
Section~\ref{sect-exist,uniq,ex} the equality
$\Dp_0(A)=\Np_0(A)$ may not hold, and it will 
be essential to consider the obstacle problem with $\Np_0(A)$,
at least for our proof of Theorem~\ref{thm-obst-solve-E}
(through the use of Lemma~\ref{lem-mazur-consequence}).

The \p-weak upper gradients $g_{u,A}$ and $g_{v,A}$
in Definition~\ref{deff-obst-E} are taken with
respect to $A$, but the notion of q.e.\ is taken with respect to $X$. 
We shall below comment on obstacle problems with q.e.\ taken 
with respect to $E$ and with a.e.-inequalities.

Obstacle and Dirichlet problems have traditionally been solved on open sets
$\Om$, in which case $g_{u,\Om}=g_u$ a.e.
See, however,  Kilpel\"ainen--Mal\'y~\cite{KilMa92} and
Mal\'y--Ziemer~\cite{MaZi} where they are studied on finely open sets 
in $\R^n$.
In metric spaces the single obstacle problem was studied
by Kinnunen--Martio~\cite{KiMa02}, while the double obstacle problem
was studied by Farnana~\cite{Fa1}.
In both cases they studied the obstacle problems for bounded open sets 
in a complete metric space $X$  supporting a 
$(1,p)$-Poincar\'e inequality
and with a doubling measure $\mu$
(and with boundary values in the Newtonian space).

The Dirichlet problem on metric spaces was first studied
by Shanmugalingam~\cite{Sh-harm}.
She studied it  on bounded, not necessarily open, sets in 
a complete metric space $X$ with a doubling measure $\mu$  
supporting a $(1,p)$-Poincar\'e inequality,
under the stronger requirement that $f \in \Np(X)$.

In all the above cases, the \p-energy functional was defined by means of
the global minimal \p-weak upper gradient $g_u$.
Thus, the Dirichlet problem studied by Shanmugalingam~\cite{Sh-harm}
differs in general from the Dirichlet problem considered here.
Similarly, for a nonopen set $E$, another possible generalization 
of the obstacle problem would be
to require that the boundary 
data $f$ belong to $\Dp(\Om)$ for some open set $\Om\supset E$ 
and to minimize the energy $\int_E g^p_v \, d\mu$ among all 
$v\in\K'_{\psi_1,\psi_2,f}$, where
\begin{equation}   \label{eq-def-K'}
\K'_{\psi_1,\psi_2,f} = \{v \in \Dp(\Om) : v-f \in \Np_0(E) 
            \text{ and } \psi_1 \le v \le \psi_2 \ \text{q.e. in } E\}.
\end{equation}

As $\Om$ is open, 
the minimal \p-weak upper gradients and the notion of q.e.\ are taken 
with respect to $\Om$ or equivalently $X$.
If we let 
\begin{equation} \label{eq-psi'j}
	\psi'_j=\begin{cases}
                  \psi_j & \text{in } E, \\
                  f & \text{in } \Om \setm E, 
           \end{cases}
     \quad j=1,2,
\end{equation}
then $\K'_{\psi_1,\psi_2,f}=\K_{\psi_1',\psi_2',f}(\Om)$, 
where we use our convention that $v-f \in \Np_0(E)$ can be 
extended by zero in $\Om\setm E$.
Moreover, for any $v \in \K'_{\psi_1,\psi_2,f}$,
\[
        \int_\Om g^p_v \, d\mu
        = \int_{\Om \setm E} g^p_f \, d\mu + \int_E g^p_v \, d\mu,
\]
as $v=f$ q.e.\ in $\Om \setm E$.
Hence the minimizers of
the energies $\int_E g^p_v \, d\mu$ 
and $\int_\Om g^p_v \, d\mu$ among 
$v \in \K'_{\psi_1,\psi_2,f}=\K_{\psi_1',\psi_2',f}(\Om)$ 
coincide and the theory for this generalization follows
directly from the theory for open sets.
Observe however that we study the obstacle problem on more general
metric measure spaces 
than  previously done, also for open sets,
see e.g.\ Example~\ref{ex-global-but-not-local-PI} and Section~\ref{sect-an ex},
and that we only require $f$ to belong to the Dirichlet space $\Dp$.

Here we have ignored one subtle point, viz.\ we require $\Cp(X \setm E)>0$,
but it is not clear if one can find an open set
$\Om \supset E$ such that $\Cp(X \setm \Om)>0$.
This is always possible if $X$ is unbounded, and also if $\mu(X \setm E)>0$,
by the regularity of the measure and the measurability of $E$.
Similarly, if $E$ is a $G_\de$ set, then $\Om$ can be found using an analogue 
for the $\Cp$-capacity of the
property~\ref{cp-Choq-E-sum} in Theorem~\ref{thm-cp}.
Moreover, if $X$ is a 
complete metric space supporting a $(1,p)$-Poincar\'e inequality,
$\mu$ is doubling, and $X \setm E$ is Souslin 
(in particular if $E$ is Borel),
then the same follows from Choquet's capacitability theorem, see
Theorem~6.11 in Bj\"orn--Bj\"orn~\cite{BBbook}.

In our approach, we only assume that the boundary data belong to 
$\Dp(E)$ and the minimal \p-weak upper gradient is taken with respect to $E$.
This leads to a different obstacle problem since $g_{u,E}$ is in general
smaller than $g_u$, see Example~\ref{ex-R-Q-1}.
The two definitions of obstacle problems will be further compared in 
Section~\ref{sect-compare}.

Note that even though we take the gradients with respect to $E$, we require
the obstacle inequalities $\psi_1\le u\le\psi_2$ to hold q.e., 
where q.e.\ is taken with respect to $X$. 
This may seem unnatural, but there are several reasons for this choice.
First of all, this is the natural condition for $\Np_0(E)$
and means that the uniqueness we obtain in Theorem~\ref{thm-obst-solve-E}
is precisely up to sets of capacity zero with respect to $X$ (\emph{not} $E$).
It also turns out to be essential for Adams' criterion
(Theorem~\ref{thm-Choq-int-est}).

Second, we could actually have developed the theory with $E$-q.e., 
i.e.\ quasieverywhere taken 
with respect to $E$, which is a coarser condition, or the even coarser 
condition a.e. 
The latter was used by  Kinnunen--Martio~\cite{KiMa02}.
See also the discussion 
on q.e.- and a.e.-obstacle problems in Farnana~\cite{Fa1}.
In particular, if $\Cp(A)>0=\CpE(A)$, where $\CpE(A)$ 
is the capacity of $A$ with respect to $E$,
then the zero function belongs to $\K_{\chi_A,0}$
(and solves the obstacle problem) with q.e.\ taken with respect to
$E$, but not when taken with respect to $X$.
The $E$-q.e.\ theory would fall in between the q.e.- and a.e.-theories,
and it is easy to adapt most of our results to this setting if need arises,
but there is e.g.\ no direct counterpart of Adams' criterion.

We have the following existence and uniqueness theorem.

\begin{thm} \label{thm-obst-solve-E}
Let $f \in \Dp(E)$ 
and $\psi_1,\psi_2 : E \to \eR$. 
If $\K_{\psi_1,\psi_2,f} \ne \emptyset$, then
there is a unique solution\/ 
\textup{(}up to sets of capacity zero in $X$\textup{)}
of the $\K_{\psi_1,\psi_2,f}$-obstacle problem.
\end{thm}

\begin{proof}
Let 
\[
       I=\inf_{v \in \K_{\psi_1,\psi_2,f}} \int_E g^p_{v,E} \,d\mu.
\]
Since $\K_{\psi_1,\psi_2,f} \ne \emptyset$, we have $0 \le I < \infty$.
Let $\{u_j\}_{j=1}^\infty \subset \K_{\psi_1,\psi_2,f}$ 
be a minimizing sequence
such that
\[
       \int_E g^p_{u_j,E} \,d\mu \searrow I,
	\quad \text{as }j \to \infty.
\]
Then $\{g_{u_j,E}\}_{j=1}^\infty$ is bounded in $L^p(E)$.
Remember that $u_j\in\Dp(E)$ and that $g_{u_j,E}$ is taken with respect to $E$.

Using \eqref{PI-NP0}  
 and Proposition~\ref{prop-min-grad} we find that
\begin{equation*} 
\int_E |u_j - f|^p \,d\mu
       \le C        \int_E g_{u_j - f}^p \,d\mu
       =  C        \int_E g_{u_j - f,E}^p \,d\mu
       \le C        \int_E g_{u_j,E}^p \,d\mu
       +        C        \int_E g_{f,E}^p \,d\mu.
\end{equation*}
Thus $\{u_j-f\}_{j=1}^\infty$ is bounded in $\Np(E)$.
By Lemma~\ref{lem-mazur-consequence} (with $X$ replaced by $E$), 
we can find convex combinations 
$v_j=\sum_{k=j}^{N_j} a_{j,k} u_k$
with \p-weak upper gradients $g_j=\sum_{k=j}^{N_j} a_{j,k} g_{u_k,E}$ on $E$
and limit functions $v$ and $g$ such that
$v-f \in \Np(E)$, both $v_j-v \to 0$ and $g_j \to g$ in
$L^p(E)$, 
as $j \to \infty$, 
and such that
$g$ is a \p-weak upper gradient of $v$ with respect to $E$.

Further, $w_j:=v_j -f \in \Np_0(E)$ and we can thus consider $w_j$ to
be identically zero outside of $E$. 
Let also $w=v-f$, $g'_j= g_j +g_{f,E}$ and $g'=g+g_{f,E}$, all three 
considered to be identically zero outside of $E$.
Proposition~\ref{prop-min-grad} implies that
\[
g_{w_j} = g_{w_j,E} \le g_j + g_{f,E} = g'_j
\quad \text{a.e.\ in } E.
\]
As $g_{w_j}=0$ a.e.\ in $X\setm E$, we see that $g'_j$ is a \p-weak upper
gradient of $w_j$ in $X$, $j=1,2,\ldots$\,.
We also have that $w_j \to w$ and $g'_j \to g'$ in $L^p(X)$,
as $j \to \infty$.
Proposition~\ref{prop-Fuglede-consequence} yields that there exists
$\wt \in \Np(X)$ such that $w=\wt$ a.e.\ in $X$.
Then $u:=f+\wt\in\Dp(E)$ and $u=v$ a.e.\ in $E$.
Since $u,v \in \Dp(E)$, 
we have $u=v$ $E$-q.e.\ in $E$ (i.e.\ q.e.\ with respect to E), 
 and thus $g$ is a \p-weak
upper gradient also of $u$ with respect to $E$.

Proposition~\ref{prop-Fuglede-consequence} also implies that a subsequence 
of $\{w_j\}_{j=1}^\infty$ converges q.e.\ (with respect to $X$) to $\wt$.
As $\psi_1\le v_j\le \psi_2$ q.e.\ in $E$, this implies that 
$\psi_1\le u\le \psi_2$ q.e.\ in $E$.
Moreover, it implies that $\wt=0$ q.e.\ in $X \setm E$ and thus
$u-f=\wt\in\Np_0(E)$.
Hence $u\in\K_{\psi_1,\psi_2,f}$. 
Since
\[
       I \le \int_E g^p_{u,E} \,d\mu \le \int_E g^p \,d\mu
         = \lim_{j \to \infty} \int_E g_j^p \,d\mu = I,
\]
we conclude that $u$ is the desired minimizer.

To prove the uniqueness,
assume that $u_1$ and $u_2$ are solutions.
Then also $u'=\frac{1}{2}(u_1+u_2)\in \K_{\psi_1,\psi_2,f}$
and thus
\begin{align*}
I &\le \|g_{u',E}\|_{L^p(E)} 
    \le \bigl\| \tfrac{1}{2} (g_{u_1,E}+g_{u_2,E})\bigr\|_{L^p(E)} \\
&    \le \tfrac12 \|g_{u_1,E}\|_{L^p(E)} + \tfrac12 \|g_{u_2,E}\|_{L^p(E)} = I.
\end{align*}
Hence  $g_{u_1,E}=g_{u_2,E}$ a.e.\ in $E$ by the strict convexity of
$L^p(E)$. 
We shall show that $g_{u_1-u_2,E}=0$ a.e.\ in $E$. 
Since $u_1-u_2\in\Np_0(E)$,  \eqref{PI-NP0}
and Proposition~\ref{prop-min-grad} then yield
$\|u_1-u_2\|_{L^p(E)}=0$. 
From this it follows that 
$u_1-u_2=0$ a.e.\ in $E$ and thus in $X$
(when we set $u_1-u_2:=0$ in $X \setm E$).
As $u_1-u_2\in \Np_0(E) \subset \Np(X)$, we obtain $u_1-u_2=0$ q.e.\ in $X$,
and hence $u_1=u_2$ q.e.\ in $E$.
(Note that since we consider q.e.\ with respect to $X$, we have to use 
the fact that $u_1-u_2\in\Np(X)$ rather than $u_1-u_2\in\Np(E)$.)

To show that  $g_{u_1-u_2,E}=0$ a.e.\ in $E$, let $c\in\R$ and
\[
u=\max\{u_1,\min\{u_2,c\}\}.
\]
Then $u-f\in\Np(E)$ and $\psi_1 \le u\le\psi_2$ q.e.\ in $E$.
Also,
\[
u-f \le \max\{u_1,u_2\}-f = \max\{u_1-f,u_2-f\} \in\Np_0(E)
\]
and 
$u-f\ge u_1-f \in\Np_0(E)$.
Lemma~\ref{lem-police} shows that $u-f\in\Np_0(E)$ and hence
$u\in\K_{\psi_1,\psi_2,f}$.

Let $V_c=\{x\in E: u_1(x)<c<u_2(x)\}$ and note that 
$V_c\subset\{x\in E:u(x)=c\}$ and hence $g_{u,E}=0$ a.e.\ in $V_c$.
The minimizing property of $g_{u_1,E}$ then implies that
\begin{equation}
\int_E g_{u_1,E}^p\,\dmu \le \int_E g_{u,E}^p\,\dmu 
         = \int_{E\setm V_c} g_{u,E}^p\,\dmu
         = \int_{E\setm V_c} g_{u_1,E}^p\,\dmu,
                   \label{eq-gu1=0}
\end{equation}
since 
$g_{u,E}=g_{u_1,E}=g_{u_2,E}$ a.e.\  in $E \setm V_c$.
From \eqref{eq-gu1=0} we conclude that $g_{u_2,E}=g_{u_1,E}=0$ a.e.\ in $V_c$
for all $c\in\R$.
Now, 
\[
\{x\in E: u_1(x) < u_2(x)\} \subset \bigcup_{c\in\Q} V_c
\] 
and hence $g_{u_2,E}=g_{u_1,E}=0$ a.e.\ in $\{x\in E: u_1(x) < u_2(x)\}$,
and similarly in $\{x\in E: u_1(x) > u_2(x)\}$.
It follows that
\[
g_{u_1-u_2,E}\le (g_{u_1,E}+g_{u_2,E}) \chi_{\{x\in E: u_1(x) \neq u_2(x)\}} =0
         \quad \mbox{a.e.\ in }E,
\]
and thus $u_1=u_2$ q.e.\ by the above argument.

It remains to show that
if $u$ is a solution and $v=u$ q.e., then $v$ is also a solution.
Indeed, it follows directly that $v \in \K_{\psi_1,\psi_2,f}$.
Moreover, $v=u$ $E$-q.e., and thus $g_{u,E}=g_{v,E}$ a.e., so that 
\[
   \int_E g_{v,E}^p \, d\mu 
   = \int_E g_{u,E}^p \, d\mu,
\]
showing that $v$ must also be a solution.
\end{proof}

The following comparison principle follows from the uniqueness of the
solutions and is useful in various applications.
Note again that the boundary data $f$ and $f'$ are only defined on $E$.
But if $f,f'\in\Np(\itoverline{E})$ and $f\le f'$ q.e.\  on $\bdry E$,
then Lemma~\ref{lem-Np0-0-qe} implies that
the condition $(f-f')_\limplus \in \Np_0(E)$  is satisfied.
Recall that $f_\limplus=\max\{f,0\}$ and $f_\limminus=\max\{-f,0\}$.

\begin{cor} \label{cor-obst-le}
\textup{(Comparison principle)}
Let $f,f' \in \Dp(E)$ 
and  $\psi_j,\psi'_j : E \to \eR$, $j=1,2$,
be such that $\K_{\psi_1,\psi_2,f}$ and  $\K_{\psi'_1,\psi'_2,f'}$ are nonempty.
Let further $u$ and $u'$ be 
solutions of the 
$\K_{\psi_1,\psi_2,f}$- and  $\K_{\psi'_1,\psi'_2,f'}$-obstacle problems, 
respectively.
If  $\psi_j \le \psi_j'$ q.e.\ in $E$, $j=1,2$,
and\/ $(f-f')_\limplus \in \Np_0(E)$,
then $u \le u'$ q.e.\ in\ $E$.
\end{cor}

In the next section we discuss relaxations of the conditions imposed in 
this section.
For the comparison principle to hold it is enough that one of the obstacle problems is
q.e.-uniquely soluble (and the other soluble). 
(In the proof, the uniqueness of the $\K_{\psi_1,\psi_2,f}$-obstacle problem
is used, 
but by symmetry one can equally well use the uniqueness of 
the $\K_{\psi'_1,\psi'_2,f'}$-obstacle problem.)

\begin{proof}
Let $w=\min\{u,u'\}$
and  $h=u-f-(u'-f') \in \Np_0(E)$.
It follows that
\[
   -(f-f')_\limplus-h_\limminus 
    = -(f'-f)_\limminus-h_\limminus 
    \le \min\{f'-f,h\} 
        \le h.
\]
Lemma~\ref{lem-police} then implies that $\min\{f'-f,h\} \in \Np_0(E)$
and hence
\[
       w-f =\min\{u'-f,u-f\}
             =u'-f'+\min\{f'-f,h\}\in \Np_0(E).
\]
As $\psi_1 \le w \le \psi_2$ q.e.\ in $E$, we get 
$w \in \K_{\psi_1,\psi_2,f}$.

Similarly $v=\max\{u,u'\} \in \K_{\psi_1',\psi_2',f'}$.
Let $A=\{x \in E : u(x) > u'(x)\}$.
Since $u'$ is a solution of the $\K_{\psi_1',\psi_2',f'}$-obstacle problem,
we have
\[
         \int_E g_{u',E}^p \, d\mu
         \le  \int_E g_{v,E}^p \, d\mu
         =  \int_A g_{u,E}^p \, d\mu
         +  \int_{E \setm A} g_{u',E}^p \, d\mu.
\]
Thus
\[
        \int_A g_{u',E}^p \, d\mu
         \le  \int_A g_{u,E}^p \, d\mu.
\]
It follows that
\[
        \int_E g_{w,E}^p \, d\mu
        =       \int_A g_{u',E}^p \, d\mu
        +       \int_{E \setm A} g_{u,E}^p \, d\mu
        \le     \int_A g_{u,E}^p \, d\mu
        +       \int_{E \setm A} g_{u,E}^p \, d\mu
        =       \int_E g_{u,E}^p \, d\mu.
\]
Since $u$ is a solution of the $\K_{\psi_1,\psi_2,f}$-obstacle problem,
so is $w$.
By uniqueness $u=w=\min\{u,u'\}$ q.e.\ in $E$, 
and thus $u \le u'$ q.e.\ in $E$.
\end{proof}

\section{Assumptions and examples}
\label{sect-exist,uniq,ex}

Both in the existence and the uniqueness parts of the proof of
Theorem~\ref{thm-obst-solve-E} we used the ``extra'' assumptions that $p>1$
(through the use of Lemma~\ref{lem-mazur-consequence}
and the strict convexity of $L^p$), that $\Cp(X \setm E)>0$
and that $X$ supports a $(p,p)$-Poincar\'e inequality for $\Np_0$.
It may be worth discussing when these assumptions hold and whether they
could possibly be dropped or weakened.
Let us start by discussing  the $(p,p)$-Poincar\'e inequality for $\Np_0$.
By the following lemma it follows from
the $(p,p)$-Poincar\'e inequality on large balls.

\begin{lem}      \label{lem-PI-NP0}
Assume that 
for every ball $B \subset X$
there is a constant $C_B>0$ 
such that for all $u \in \Np_0(B)$, 
\begin{equation} \label{PI-ineq-relax}
        \int_{B} |u-u_{B}|^p \,\dmu
        \le C_B  \int_{B} g_u^{p} \,\dmu.
\end{equation}
Then $X$ supports a\/ $(p,p)$-Poincar\'e inequality for $\Np_0$.
\end{lem}

Since $g_u=0$ outside $B$ 
there is no reason to have a dilation constant $\la$ in \eqref{PI-ineq-relax},
as in Definition~\ref{def-PI}.
Note also that the doubling property of $\mu$ is not needed.
The proof of Lemma~\ref{lem-PI-NP0} has been inspired by
Theorem~10.1.2 in Maz\cprime ya~\cite{Mazya-Sob} and
Proposition~3.2 in J.~Bj\"orn~\cite{BjIll},
but is slightly simpler and sufficient for our purpose.

For unbounded $X$ we always have $\Cp(X \setm B)>0$ and hence 
\eqref{PI-ineq-relax} follows from \eqref{PI-NP0}
by means of the H\"older and Minkowski
inequalities. 
Thus, the $(p,p)$-Poincar\'e inequality
for $\Np_0$   and \eqref{PI-ineq-relax} are equivalent
in unbounded spaces.
The case when $X$ is bounded is more subtle, since we cannot take $E=X$
in \eqref{PI-NP0}.
We do not know if the equivalence is true in this case.

It may also be worth observing that
contrary to the classical Poincar\'e inequalities, 
here
it is enough to require \eqref{PI-ineq-relax} or \eqref{PI-ineq} 
for large balls, i.e.\ that
for every ball $B'$ there is a ball $B \supset B'$
such that \eqref{PI-ineq-relax} or \eqref{PI-ineq} holds. 
(If $X$ is bounded it suffices to assume
that \eqref{PI-ineq-relax} or \eqref{PI-ineq} holds for $B=X$.)
The following example 
shows that 
this is not equivalent to \eqref{PI-ineq}
holding for all balls.

\begin{example}   \label{ex-global-but-not-local-PI}
Let $X\subset\R^2$ be the graph of the function $y=x^\al \sin(\pi\log_2 x)$,
$0<\al<1$, $0\le x\le1$, with the $\R^2$-Euclidean metric and 
the arc length measure 
$\Lambda_1$.
It is easily verified that $L:=\Lambda_1(X)<\infty$.
Let $\ga:[0,L] \to X$ be an arc length parameterized curve such 
that $\ga(0)=(0,0)$ and $\ga(L)=(1,0)$.
Since $\ga$ gives a natural bijection between $X$ and $[0,L]$,
every function in $\Np(X)=\Np_0(X)$ 
is absolutely continuous on $X$ with $g_u(\ga(t))=|(u \circ \ga)'(t)|$ a.e.

Let $z=(2^{-k},0)\in X$ and $2^{-k-1}<r<2^{-k}$, $k=1,2,\ldots$\,.
Then the ball $B=B(z,r)$ is not connected and
$B$ does not even belong to one component of  $\la_k B$, 
where $\la_k = 2^{k(1-\al)-1}$.
Letting $k\to\infty$ shows that $X$ cannot support any Poincar\'e
inequality with the same dilation constant $\la$ for all balls.
At the same time, the $(p,p)$-Poincar\'e inequality for $\Np_0$ holds on $X$,
since 
\[
\int_X |u-u_X|^p\,d\Lambda_1 
= \int_0^L |u(\ga(t))-u_X|^p\,dt
\le C  \int_0^L |(u \circ\ga)'(t)|^p\,dt
= C  \int_0^L g_u^p \, d\Lambda_1,
\]
by the $(p,p)$-Poincar\'e inequality for $[0,L]$.
\end{example}

\begin{proof}[Proof of Lemma~\ref{lem-PI-NP0}]
Let $E \subset X$ be bounded and such that $\Cp(X \setm E)>0$.
Let $u\in\Np_0(E)$, extended  by zero in $X\setm E$.
We can assume that the left-hand side in \eqref{PI-NP0} is nonzero.

If $X$ is unbounded, let $B\supset E$ be a ball such that 
$\mu(E)<\mu(B)$.
Then 
\begin{equation}  \label{eq-u-B-mu}
     \biggl(\int_{B} |u|^p \, d\mu\biggr)^{1/p} 
     \le  \biggl(\int_{B} |u-u_B|^p \, d\mu\biggr)^{1/p} 
          + |u_{B}|\mu(B)^{1/p}. 
\end{equation}
The first term on the right-hand side is estimated using
\eqref{PI-ineq-relax} and for the second term we have,
using H\"older's inequality and the fact that $u$ vanishes 
outside $E$, that
\[
|u_{B}|\mu(B)^{1/p} \le \frac{1}{\mu(B)^{1-1/p}} \int_B|u|\,d\mu
\le \biggl( \frac{\mu(E)}{\mu(B)} \biggr)^{1-1/p} 
        \biggl(   \int_{B} |u|^p \,d\mu \biggr)^{1/p}.
\]
Since $\mu(E)<\mu(B)$, inserting this into \eqref{eq-u-B-mu}
and subtracting the last term from
both sides of \eqref{eq-u-B-mu} proves \eqref{PI-NP0} for unbounded $X$.

If $X$ is bounded, let 
\[
        \ub = \biggl( \vint_{X} |u|^p\,\dmu \biggr)^{1/p}.
\]
Then $v:=1-u/\ub$ is admissible in the definition of $\Cp(X\setm E)$ 
and hence
\begin{align}
 0<\Cp(X\setm E) &\le \int_X v^p \,\dmu + \int_X g_v^p \,\dmu  
         \le \frac{1}{\ub^p} \biggl( \int_{X} |u-\ub|^p \,\dmu
       + \int_{X} g_u^p \,\dmu \biggr).
\label{est-cap-ubar}
\end{align} 
The first integral on the right-hand side can be estimated as
\begin{align*}
        \|u-\ub\|_{L^p(X)}
                &\le \|u-u_{X}\|_{L^p(X)} + \|\ub-u_{X}\|_{L^p(X)},
\end{align*} 
where for the second term we have
\begin{align*}
\|\ub-u_{X}\|_{L^p(X)}
      &= \bigl| \|u\|_{L^p(X)} - \|u_{X}\|_{L^p(X)} \bigr| 
        \le \|u-u_{X}\|_{L^p(X)}.
\end{align*}
Inserting this into~\eqref{est-cap-ubar} and
using \eqref{PI-ineq-relax} with $B=X$ we obtain
\begin{equation*}
\ub^p \le \frac{C}{\Cp(X\setm E)} \int_{X} g_u^p \,\dmu.\qedhere
\end{equation*}
\end{proof}

The following two examples show that neither the existence
nor the uniqueness of solutions remain valid for $p=1$.
Note that in Examples~\ref{ex-1}--\ref{ex-Koch}
we have $f \in \Np(E)$ and $E$ is open.

\begin{example} \label{ex-1}
Let $X=\R$,
$p=1$, $E=(0,1)$,  $d\mu= w \, dx$, where 
\[
   w(x)= \begin{cases}
    1+x, & 0<x<1, \\
    1, & \text{otherwise},
     \end{cases}
\]
$f(x)=x$ and  $\psi=-\infty$, i.e.\ we consider a weighted Dirichlet problem.
Note that $\mu$ is a doubling measure supporting a $(1,1)$-Poincar\'e inequality.
Let further $u_j(x)=\min\{jx,1\}\in \K_{\psi,f}$, $j =1,2,\ldots$,
so that 
\[
   \int_E g_{u_j} \, d\mu 
   = \int_0^{1/j} j \,d\mu
   = \int_0^{1/j} j (1+x) \,dx
   = 1+ \frac{1}{2j} \to 1,
   \quad \text{as } j \to \infty.
\]
On the other hand, for any $v \in \K_{\psi,f}$
we have that
\[
   1 = v(1)-v(0) =  \int_0^1 v'\,dx 
   \le \int_0^1 |v'| \, dx
   =   \int_0^1 g_v\,dx 
   <   \int_0^1 g_v\,d\mu, 
\]
since $g_v$ cannot vanish a.e. 
This shows that the minimum is not attained and thus
there are no minimizers.
Hence the assumption $p>1$ cannot be removed for the existence part.
\end{example}

\begin{example}
Let $X=\R$ (unweighted), 
$p=1$, $E=(0,1)$, $f(x)=x$ and $\psi=-\infty$.
In this case any increasing absolutely continuous function 
$u:[0,1] \to [0,1]$  with $u(0)=0$ and $u(1)=1$ will be a solution
of the $\K_{\psi,f}$-obstacle problem (i.e.\ of the Dirichlet
problem with $f$ as boundary values).
Thus the assumption $p>1$ cannot be omitted for the uniqueness part either.
\end{example}

As for when the Poincar\'e inequality is essential, 
the situation is more complicated.
Let us first look at the question of existence of solutions.

\begin{example} \label{ex-Xpm}
Let $1<p < 2$ and 
\begin{align*}
    X &= \{(x,y)\in [-2,2]^2: xy\ge0\}, \\
    X_\limplus& = \{(x,y)\in X: x\ge0\} \setm \{(0,0)\} = [0,2]^2 \setm \{(0,0)\}, \\
    X_\limminus & = \{(x,y)\in X: x\le0\} = [-2,0]^2.
\end{align*} 
Then there are \p-almost no curves between $X_\limplus$ and $X_\limminus$
(since $\Cp(\{0\})=0$) which means that in this context they can be 
thought of as disconnected,
see Example~5.6 in Bj\"orn--Bj\"orn~\cite{BBbook}.
In particular, $u=\chi_{X_\limplus} \in \Np_0(X_\limplus)$ with $g_u=0$,
showing that the $(p,p)$-Poincar\'e inequality for $\Np_0$ is violated.
(This also shows that the zero \p-weak upper gradient property,
introduced below, fails at $(0,0)$.)

Let $f=0$, $1-2/p < \alp < 0$, $E = X_\limplus$,
\[
      \psi(x,y)=\begin{cases}
         |(x,y)-(1,1)|^\alp, & (x,y) \in X_\limplus , \\
         0, & (x,y) \in X_\limminus, 
        \end{cases}
\]
and $u_j=\max\{\psi,j\} \chi_{X_\limplus}$.

Then $\psi \in \Np(X)$ and
$u_j -f \in \Np_0(X_\limplus)$, i.e.\ 
$u_j \in \K_{\psi,f}(X_\limplus)$.
Moreover,
\[
     \int_{X_+} g_{u_j}^p \, d\mu \to 0,
     \quad \text{as } j \to \infty.
\]
On the other hand, if $\Np(X)\ni v\ge\psi$ q.e.\ in $X_\limplus$
then necessarily $\int_{X_+} g_{v}^p \, d\mu > 0$,
and there does not exist any minimizer for the 
$\K_{\psi,f}(X_\limplus)$-obstacle
problem.
\end{example}

A problem with Example~\ref{ex-Xpm} is that $\Cp(\bdry X_\limplus)=0$ even though 
$\Cp(X\setm X_\limplus)>0$, allowing for $u_j\in\Np_0(X_\limplus)$.
Similarly, the same functions $u_j$ show that the 
$\K_{\psi,f}(\Om)$-obstacle problem with $\Om=\{(x,y)\in X:x>-1\}$ is not
soluble either.
Here, the problem is that $\Om$ is essentially disconnected
and thus the boundary values $f$ have
no influence in $X_\limplus$, even though $\Cp(\bdry\Om)>0$.
(In fact, $\Om$ itself need not be connected, but it should not have a component
which is essentially disconnected from $\Om$'s complement.)

A Poincar\'e inequality of some kind  prevents these problems and 
guarantees solubility of the obstacle problem.
The above $\K_{\psi,f}(\Om)$-obstacle problem also shows that it is not enough
to just replace the assumption $\Cp(X\setm E)>0$ with
$\Cp(\bdry E)>0$.
Under a Poincar\'e inequality and for open $E$, these two conditions
are equivalent by Lemma~4.5 in Bj\"orn--Bj\"orn~\cite{BBbook}.
For open $E$ in general spaces, the latter condition is stronger,
as seen above.
On the other hand, the former condition can be stronger for nonopen sets.
We have therefore chosen to use the condition $\Cp(X\setm E)>0$, as
it is closely related to $\Np_0(E)$.

\begin{remark} \label{rmk-existence}
On the other hand, if the data $f$, $\psi_1$ and $\psi_2$
are bounded then we can drop both the assumption of Poincar\'e inequality
and the assumption $\Cp(X \setm E)>0$.
This will be important for  Theorem~\ref{thm-condenser}.
In the existence part of the proof, they
were only used to 
deduce that $\{u_j-f\}_{j=0}^\infty$ is  bounded in $L^p(E)$,
and this can be deduced more directly if the data are bounded.
More precisely, consider the following two cases:
\begin{enumerate}
\item \label{case-a}
$\psi_j \in L^p(E)$, which in particular holds if $\psi_j\in L^\infty(E)$, $j=1,2$; 
\item \label{case-b}
$C_0:=\esssup_E |f| <\infty$, $C_1:=\esssup_E \psi_1 <\infty$
and  $C_2:=\essinf_E \psi_2 > -\infty$.
\end{enumerate}
In case \ref{case-a},  
the $L^p$-boundedness of $\{u_j\}_{j=1}^\infty$
follows directly since 
$\psi_1\le u_j\le\psi_2$ a.e.
In case \ref{case-b} we may replace $u_j$ by the truncations
\[
   u_j':=\max\{\min\{u_j,\max\{C_0,C_1\}\},\min\{-C_0,C_2\}\}.
\]
at levels $\max\{C_0,C_1\}$ and $\min\{-C_0,C_2\}$.
Then $g_{u_j',E} \le g_{u_j,E}$ and the sequence $\{u_j'\}_{j=1}^\infty$
is bounded in $\Np(E)$.
(In both cases one uses the $L^p$-boundedness of $\{u_j\}_{j=1}^\infty$
(or $\{u_j'\}_{j=1}^\infty$) 
rather than the $L^p$-boundedness of  $\{u_j-f\}_{j=1}^\infty$,
 in the proof of Theorem~\ref{thm-obst-solve-E}.
This also makes the proof a little easier.)
\end{remark}

Let us now turn to the question of uniqueness.
The following example shows that we cannot drop
the Poincar\'e inequality entirely.

\begin{example} \label{ex-Koch}
Let $X$ be the von Koch snowflake curve. Let $a,b \in X$, $ a\ne b$,
and let $E$ be one of the two components of $X \setm \{a,b\}$.
Let further $f=0$ and $\psi=-\infty$.
Since there are no rectifiable curves in $X$, we have
$\Np_0(E)=L^p(E)$ and $g_u \equiv 0$ for all $u \in \Np_0(E)$,
which means that any $u \in L^p(E)$ is a solution
of the $\K_{\psi,f}$-obstacle problem (i.e.\ of the Dirichlet
problem with $f$ as boundary values).
Thus the assumption that $X$ supports some kind of
Poincar\'e inequality
cannot be omitted for the uniqueness part.
Similar arguments apply to other spaces without rectifiable curves,
or with \p-almost no rectifiable curves.
\end{example}

\begin{remark} \label{rmk-uniq}
Even though the Poincar\'e inequality cannot be omitted for the
uniqueness part, it can be weakened.
In the uniqueness part of the proof, the $(p,p)$-Poincar\'e inequality for $\Np_0$
was only used
to deduce that $v:=u_1-u_2=0$ a.e.\ in $E$ from the fact that
$g_{v}=0$ in $X$ and $v=0$ outside $E$.

In A.~Bj\"orn~\cite{ABcluster} 
the following weaker property was introduced:
$X$ has the \emph{zero \p-weak
upper gradient property} 
if every 
measurable function $f$, 
which has $0$ as a \p-weak upper gradient in some 
ball $B(x,r)$, 
is  essentially constant in 
some (possibly smaller) ball  $B(x,\de)$, which can depend both on $f$ and
$B(x,r)$.
By considering the bounded function
$h =\arctan f$ with $g_h=g_f/(1+f^2)$, we easily conclude that
one can equivalently consider only bounded measurable functions in the
definition of the zero \p-weak upper gradient property.

Thus, Lemma~3.2 in \cite{ABcluster} shows that
when proving uniqueness of the solutions
we may replace the Poincar\'e inequality by
the zero \p-weak upper gradient
property, together with the fact that $\Cp(G \setm E)>0$ 
for every component $G$ of $X$.
The latter is essential since there are nonconnected spaces having
the zero \p-weak upper gradient property, 
e.g.\ $X=[0,1]^2 \cup [2,3]^2$ in $\R^2$
and $X=[0,1]\cup[2,3]\subset\R$.

The zero \p-weak upper gradient property is strictly weaker
than supporting a $(1,p)$-Poincar\'e inequality
(as the two examples above show).
On the other hand, the following example shows that
$X$ can support a $(p,p)$-Poincar\'e inequality for $\Np_0$
and at the same time fail to have
the zero \p-weak upper gradient property.
\end{remark}

\begin{example} \label{ex-6}
Let $1<p \le  2$ and let 
\begin{align*}
    X &= \{(x,y)\in\R^2: xy\ge0\}, \\
    X_\limplus& = \{(x,y)\in X: x\ge0\}, \\
    X_\limminus & = \{(x,y)\in X: x\le0\}.
\end{align*}
As in Example~\ref{ex-Xpm}, the function $\chi_{X_\limplus}$ shows that the
zero \p-weak upper gradient property fails for all balls centered at the origin.

On the other hand, as both $X_\limplus$ and $X_\limminus$ support 
$(p,p)$-Poincar\'e inequalities, they support  
$(p,p)$-Poincar\'e inequalities for $\Np_0$, by e.g.\ Lemma~\ref{lem-PI-NP0}.
Considering $u|_{X_\limplus}$ and
$u|_{X_\limminus}$ separately shows that for all bounded $E\subset X$ and
all $u\in\Np_0(E)$ we have
\[
\int_{X_\limpm} |u|^p\,d\mu 
    \le C_{E_\limpm} \int_{X_\limpm}  g_u^p\,d\mu,
\]
where $E_\limpm=E \cap X_\limpm$.
The 
$(p,p)$-Poincar\'e inequality for $\Np_0$ on $X$ then follows
by adding the $L^p$-norms on $X_\limplus$ and $X_\limminus$.
\end{example}

Let us finally discuss the assumption $\Cp(X \setm E)>0$.
If it fails  (and thus necessarily $X$ is bounded) we lose
existence in general. 
This is easily seen by letting $X=[0,2]^2\subset\R^2$
and using the construction in Example~\ref{ex-Xpm}.
However, we do have solubility
if we assume boundedness of the data
as in \ref{case-a} or \ref{case-b} of Remark~\ref{rmk-existence}.
Moreover, the Dirichlet problem (i.e.\ the obstacle problem without obstacles)
is always soluble if $\Cp(X \setm E)=0$
since the zero function is a solution with any boundary data.

On the other hand, uniqueness always fails if $\Cp(X \setm E)=0$
in the single obstacle problem
(when it is soluble), by the following result. In particular
it fails for the Dirichlet problem.

\begin{prop}
Let $X$ be bounded and $E \subset X$ be measurable and
such that $\Cp(X \setm E)=0$.
Let also $f\in \Np(E)$ and $\psi : E \to \eR$.
Let $u$ be a solution of the $\K_{\psi,f}$-obstacle problem
and $a \in \R$.
Then $v:=\max\{u,a\}$ is another solution of the 
$\K_{\psi,f}$-obstacle problem.
\end{prop}

\begin{proof}
As $\Np_0(E)=\Np(X)$ we see that $v \in \K_{\psi,f}$.
Moreover, $g_v \le g_u$ a.e.\ in $E$, and thus
$v$ must also be a solution.
\end{proof}

In fact it follows from this proof that the $\K_{\psi,f}(X)$-obstacle problem
for bounded $X$
has a solution only if there is some  function 
$u \in \K_{\psi,f}(X)$ with $g_u =0$ a.e.
If $X$ supports a $(p,p)$-Poincar\'e inequality for $\Np_0$,
then this happens if and only if there is some
constant (real-valued) function  $u \in \K_{\psi,f}(X)$,
which in turn happens if and only if $\esssup_X \psi < \infty$.

Let us end this discussion by a comment on  
the case when $E$ is unbounded.
In this case we may also lose existence, as the following example
shows.

\begin{example}
Let $X=\R$ \textup{(}unweighted\/\textup{)}, 
$p>1$, $E=(0,\infty)$, $f(x)=(1-x)_\limplus$ and $\psi=0$.
Let further $f_j(x)=(1-x/j)_\limplus$, $j=1,2,\ldots$\,.
Then $f_j \in \K_{\psi,f}$ and
\[
     \int_E g_{f_j}^p \, d\mu 
     = \int_0^j \frac{1}{j^p} \, dx = j^{1-p}
     \to 0,
     \quad \text{as } j \to \infty.
\]
This shows that a solution of the $\K_{\psi,f}$-obstacle problem
must have zero energy, and thus must be constant 
a.e. 
The boundary condition would require a solution $u$ to satisfy $u=1$ a.e.,
but then $u \notin \K_{\psi,f}$.
\end{example}

We conclude this section with an application
of our theory to condenser capacities.
On metric spaces, such capacities have been used and studied 
under various assumptions  by e.g.\ 
Heinonen--Koskela~\cite{HeKo98}, Kallunki--Shanmugalingam~\cite{KaSh}
and Adamowicz--Bj\"orn--Bj\"orn--Shan\-mu\-ga\-lin\-gam~\cite{ABBSprime}.

\begin{deff} \label{deff-cond-cap}
Let $\Om \subset X$ be a nonempty bounded open set,
and let $A_0,A_1 \subset \Om$ be disjoint.
Then the \emph{capacity}  of the condenser
$(A_0,A_1,\Om)$ is
\[
    \cp(A_0,A_1,\Om) = \inf_u \int_\Om g_u^p\,d\mu,
\]
where the infimum is  taken over all $u\in N^{1,p}(\Om)$
satisfying $0\le u \le1$ in $\Om$, $u = 0$ in $A_0$ and $u =1$ in $A_1$.
\end{deff}

Note that $\cp(A_0,A_1,\Om)=\cp(A_1,A_0,\Om)$.
Since the equivalence classes in $\Np(\Om)$ 
are up to sets of capacity zero, 
we can equivalently require the equalities 
in $A_0$ and $A_1$ to hold q.e.
This is thus a double obstacle problem in $\Om$ but without boundary values.
We obtain the following consequences of the results in this and the previous
section.

\begin{thm} \label{thm-condenser}
Assume that $p>1$.
Let\/ $\Om \subset X$ be a nonempty bounded open set,
and let $A_0,A_1 \subset \Om$ be disjoint 
sets such that\ $\cp(A_0,A_1,\Om)<\infty$
\textup{(}which in particular happens if $\dist(A_0,A_1)>0$\textup{)}.

Then there is a minimizer for the condenser\/ $(A_0,A_1,\Om)$, i.e.\
a function $u\in N^{1,p}(\Om)$
such that\/ $0\le u \le1$ in $\Om$,
$u =  0$  in $A_0$, $u = 1$  in $A_1$ and
\begin{equation} \label{eq-cp-int}
    \cp(A_0,A_1,\Om) = \int_\Om g_u^p\,d\mu.
\end{equation}

If $X$ has the zero \p-weak upper gradient property,
$\Om$ is connected, 
and $\Cp(A_0 \cup A_1)>0$, then the minimizer
is unique\/ \textup{(}up to sets of capacity zero\/\textup{)}.
\end{thm}

By Lemma~3.4 in A.~Bj\"orn~\cite{ABcluster}, the zero \p-weak upper gradient property
for $X$ holds e.g.\ if $X$ supports a $(1,p)$-Poincar\'e inequality.
For the uniqueness in Theorem~\ref{thm-condenser} 
it is actually enough if $\Om$ has
the zero \p-weak upper gradient property,
as can be seen from the proof below.

Observe that if $\Cp(A_0)=\Cp(A_1)=0$, then 
any constant function with a value in $[0,1]$ is a minimizer 
(after redefinition on $A_0 \cup A_1$),
which is thus not unique.

\begin{proof}
\emph{Existence.}
Let $f=0$,  $\psi_1=\chi_{A_1}$ and $\psi_2=\chi_{\Om \setm A_0}$.
It is then easy to see that every solution
of the $\K_{\psi_1,\psi_2,f}(\Om)$-obstacle problem taken with respect
to the ambient space $\Om$
is a  minimizer for the condenser
(after redefinition on a subset of $A_0 \cup A_1$ of capacity zero).
The existence thus follows from
Theorem~\ref{thm-obst-solve-E} and Remark~\ref{rmk-existence}.

\emph{Uniqueness.}
By symmetry, we may assume that $\Cp(A_0)>0$.
Assume that $u$ and $u'$ are two minimizers of the condenser and 
let 
\[
Z=\{x \in \Om : u(x)=u'(x)=0\},
\] 
which is a measurable set containing $A_0$.
Let also $E=\Om \setm Z$, $f=0$ and $\psi=\chi_{A_1}$.
It is again easy to see that 
both $u|_E$ and 
$u'|_E$ are solutions of the $\K_{\psi,f}(E)$-obstacle problem
taken with $\Om$ as ambient space. 
(Recall that for $u\in\Np_0(E;\Om)$ we have $g_{u,E}=g_{u,\Om}=g_u$ a.e., 
by Proposition~\ref{prop-min-grad}, and hence the energies considered
for the condenser $(A_0,A_1,\Om)$ and in the $\K_{\psi,f}(E)$-obstacle problem
coincide.
Here $\Np_0(E;\Om)$ is $\Np_0(E)$ taken with respect to the ambient space $\Om$.)
Since $X$ has the zero \p-weak upper gradient property, so does $\Om$,
as it is a local property. 
Since 
$\Cp(\Om \setm E)=\Cp(Z)\ge\Cp(A_0)>0$ and
$\Om$ is connected, 
the uniqueness thus follows from Remark~\ref{rmk-uniq}.
\end{proof}

Observe that in the existence part of the proof $f$ does not play
any role as the boundary is empty.
This is allowed by Remark~\ref{rmk-existence}.
The uniqueness, however, cannot be deduced using the obstacle problem
without boundary values, and hence a different obstacle problem
needs to be considered in the second part of the proof.

Next, we prove another application of our results, and in particular
of Theorem~\ref{thm-condenser}.
It turns out to be useful in connection with ends and prime ends 
on metric spaces in the paper
Adamowicz--Bj\"orn--Bj\"orn--Shan\-mu\-ga\-lin\-gam~\cite{ABBSprime},
cf.\  Lemma~A.11 therein.

\begin{prop}\label{prop-cpt-equiv}
Assume 
that $X$ is complete and supports a\/ $(1,p)$-Poincar\'e inequality,
that $\mu$ is doubling and that $p>1$.
Let\/ $\Om$ be a nonempty bounded connected open set,
and\/ $\{E_k\}_{k=1}^\infty$ be a decreasing sequence of subsets of\/ $\Om$
such that\/ $\bigcap_{k=1}^\infty \itoverline{E}_k \subset \bdy \Om$.

Then\/ $\lim_{j\to\infty}\cp(E_j,K,\Om)=0$
for every compact $K\subset\Om$ if and only if\/
$\lim_{j\to\infty}\cp(E_j,K_0,\Om)=0$
for some compact  $K_0\subset\Om$ with $\Cp(K_0)>0$.
\end{prop}

\begin{proof}
Assume that $\lim_{j\to\infty}\cp( E_j,K_0,\Om)=0$ for some compact set
$K_0\subset\Om$ with positive capacity,
and let $K\subset\Om$ be compact.
By Lemma~4.49 in Bj\"orn--Bj\"orn~\cite{BBbook}, there is an open connected set
$G \Subset \Om$ such that $K_0 \cup K \subset G$.
We can also find $k_0$ such that $E_{k_0} \cap G = \emptyset$.
Let us only consider $k \ge k_0$ below.

Let $u_k$ be a minimizer for $\cp(E_k,K_0,\Om)$,
which exists and is unique (up to sets of capacity zero)
by Theorem~\ref{thm-condenser}.
Note that $u_k=0$ on $E_k$ and $u_k=1$ on $K_0$.
Moreover, $u_k$ is a superminimizer
 in
$\Om \setm \itoverline{E}_k\supset G$
(see Kinnunen--Martio~\cite{KiMa02} or \cite{BBbook} for the definitions of
superminimizers and superharmonic functions).
Indeed, if $0 \le \varphi \in \Np(X)$ and $\varphi=0$ outside
$\Om \setm \itoverline{E}_k$,
then
$v=\min\{u_k+\varphi,1\}$ is admissible for $\cp(E_k,K_0,\Om)$
and hence
\[
   \int_{\Om \setm K_0} g_{u_k}^p \, d\mu
   \le \int_{\Om \setm K_0} g_v^p \, d\mu
   \le \int_{\Om \setm K_0} g_{u_k+\varphi}^p \, d\mu.
\]
By Theorem~5.1 in \cite{KiMa02}
(or Theorem~8.22  in \cite{BBbook}),
\[
           u_k^*(x):= \lim_{r \to 0} \essinf_{B(x,r)} u_k
\]
equals $u_k$ q.e.\ in $G$,
and by Proposition~7.4 in \cite{KiMa02}
(or Proposition~9.4 in \cite{BBbook}) $u_k^*$ is superharmonic
in $G$.
As $u_k^*$ is lower semicontinuous,
the minimum $\de_k:=\min_{K} u_k$ is attained at some point in $K$.
Since $u_k^*(x)=1$ for some $x \in K_0$ (as $\Cp(K_0)>0$)
we see that $u_k^* \not\equiv 0$ in $G$.
Hence, as $G$ is connected,
the strong minimum principle in $G$ (Theorem~9.13 in \cite{BBbook})
shows that $\de_k>0$.

By Corollary~\ref{cor-obst-le},
we have  $u_{k} \ge u_{k_0}$ q.e., and thus $\de_k \ge \de_{k_0}>0$.
It follows that $\min\{u_k/\de_{k_0},1\}$ is admissible for
$\cp(E_k,K,\Om)$ as $u_k / \de_{k_0} \ge 1$ on $K$.
The monotonicity of $\cp$ then yields that
\begin{align*}
   \cp(E_k,K,\Om)
   &    \le \frac{1}{\de_{k_0}^p} \int_{\Om} g_{u_k}^p \, d\mu
    = \frac{1}{\de_{k_0}^p} \cp(E_k,K_0,\Om)
   \to 0,
     \quad \text{as } k \to \infty.
\end{align*}
The converse implication is trivial.
\end{proof}

\section{Adams' criterion for when \texorpdfstring{$\K_{\psi,f} \ne \emptyset$}{}}
\label{sect-Adams}

In this section, we study when the single obstacle problem is soluble,
i.e.\ when $\K_{\psi,f}$ is nonempty.
In the characterization, we shall use the variational capacity with respect to 
nonopen sets, see Appendix~\ref{app-capp}.

\medskip

\emph{As in Section~\ref{sect-obst},
we assume that $p>1$ and that  $X$ 
supports a $(p,p)$-Poincar\'e inequality for $\Np_0$.
We also assume 
that $E \subset X$ is a bounded measurable set such that $\Cp(X \setm E)>0$.}

\begin{thm} \label{thm-Choq-int-est} 
\textup{(Adams' criterion)}
Let $f\in\Dp(E)$ and $\psi : E \to \eR$.
Then $\K_{\psi,f}\ne \emptyset$  if and only if
\begin{equation} \label{eq-Choq-int-finite}
     \int_0^\infty t^{p-1}\cp(\{x: \psi(x)-f(x)>t\},E) \,dt < \infty.
\end{equation}
\end{thm}

In the linear case on unweighted $\R^n$ and with $E$ open and
$f \in N^{1,2}(E)$ (or rather $f \in W^{1,2}(E)$ quasicontinuous)
this result was
obtained by Adams~\cite{adams81}.
For open $E$ in metric spaces and $f \in \Np(E)$, it is
included in Bj\"orn--Bj\"orn~\cite{BBbook}.

The Cavalieri principle says that if $f: X \to [0,\infty]$
is a $\nu$-measurable function then
\[ 
   \int_{X}  f^{p}\,d\nu  = p\int_{0}^\infty 
          t^{p-1}\nu(\{x : f(x)>t\})\,dt.
\]
By analogy, it is natural to write 
\eqref{eq-Choq-int-finite} as
\[
    \int_E (\psi-f)_\limplus^{p} \,d\cp(\,\cdot\,,E) < \infty,
\]
even though $\cp(\,\cdot\,,E)$ is \emph{not} a measure.
Such integrals are called \emph{Choquet integrals} and their
study goes back to Choquet~\cite{choquet}.

Let us also point out that for Theorem~\ref{thm-Choq-int-est}
to hold it is important
that the obstacle problem is defined by requiring the obstacle
inequality to hold q.e.\ (with respect to $X$).
If the inequality is only required to hold a.e.,
as e.g.\ in 
Heinonen--Kilpel\"ainen--Martio~\cite{HeKiMa}
or Kinnunen--Martio~\cite{KiMa02}, only one 
implication in Theorem~\ref{thm-Choq-int-est} is true.
To see this let $E=B(0,1) \subset \R^n$, $f \equiv 0$ 
and $\psi=\infty \chi_F$, where $F \subset E$
is a set such that $\mu(F)=0< \Cp(F)$.
By Lemma~\ref{lem-Cp<=>cp}, $\cp(F,E)>0$,
and thus by Adams' criterion, $\K_{\psi,f}= \emptyset$.
On the other hand,  $0$ is a solution of the a.e.-obstacle problem.

The same is true if we had used $E$-q.e.\ in the definition
of the obstacle problem.
In this case, we let
$E=B(0,1) \setm \Q \subset \R$, $f= 0$ and $\psi=\infty \chi_F$,
where $F$ is a nonempty set with $\CpE(F)=0 <\Cp(F)$,
which is easily accomplished as in this case $\CpE(A)=\mu(A)$ for all sets 
$A \subset E$.
Again, 
$\cp(F,E)>0$, by Lemma~\ref{lem-Cp<=>cp},
and thus $\K_{\psi,f}= \emptyset$,
by Adams' criterion, while
$0$ is a solution of the $E$-q.e.\ (and also the a.e.) obstacle
problem.

For the double obstacle problem it is much more
difficult to obtain a characterization of when 
$\K_{\psi_1,\psi_2,f} \ne \emptyset$.
The following two examples demonstrate this.

\begin{example} \label{ex-doubl-obst-difficult}
Let $X=\R$, $p>1$, $E=(0,1)$, $f(x)=x$ and
let $\psi_1, \psi_2 : \R \to \eR$ be defined by
\[
	\psi_1=\begin{cases}
		x^{1-1/p}, & 0 < x < 1, \\
		- \infty, & \text{otherwise},
		\end{cases}
	\quad
	\psi_2=\begin{cases}
		x^{1-1/p}, & 0 < x < 1, \\
		 \infty, & \text{otherwise}.
		\end{cases}
\]
Then $\K_{\psi_1,\psi_2,f}=\emptyset$, as the function
$x \mapsto x^{1-1/p}$ does not belong to $\Np(E)$.
\end{example}

In the above example, we had $\psi_1=\psi_2$ on a large set. 
We shall next see that it is possible to have $\psi_2-\psi_1=\infty$ 
everywhere while $\K_{\psi_1,\psi_2,f}$ is empty.

\begin{example}  \label{ex-infty-obst-nonsol}
Let $X=\R$, $p>1$, $\Om\subset\R$ be open, 
\[
\psi_1=-\infty\chi_\Q \quad \text{and} \quad
\psi_2=\infty(1-\chi_\Q).
\]
Note that $\psi_2-\psi_1=\infty$ everywhere.
Let $u\in\Np(\Om)$ be such that $\psi_1\le u\le \psi_2$ q.e.
Since every function in $\Np(\Om)$ is (absolutely) continuous, 
this implies that $u\ge0$ a.e.\ (and hence everywhere) in $\Om$.
On the other hand, as $\Q$ is dense in $\Om$, the continuity of $u$
and the fact that $u\le0$ on $\Q\cap\Om$ yield that $u\le0$ in $\Om$.

Hence $u=0$ in $\Om$ and $\K_{\psi_1,\psi_2,f}=\emptyset$ whenever 
$f\notin\Np_0(\Om)$.
Moreover, similar arguments show that if $\psi_1'=\psi_1+1$, then
$\K_{\psi_1',\psi_2,f}=\emptyset$ for all $f\in\Dp(\Om)$.
\end{example}

To prove Theorem~\ref{thm-Choq-int-est}
we will use the following lemma.

\begin{lem} \label{lem-Choq-Mazya-int-est}
Let $a>1$,
$u\in \Np_0(E)$ and 
$E_t=\{x\in E: |u(x)|>t\}$, $t>0$.
Then
\begin{equation} \label{eq-lem-Choq-Mazya-a}
     \int_0^\infty t^{p-1}\cp(E_{at},E_t) \,dt
	\le \frac{\log a}{(a-1)^p} \int_E g^p_u \, d\mu.
\end{equation}
Equivalently, with $b=1/a\in(0,1)$,
\begin{equation} \label{eq-lem-Choq-Mazya-b}
     \int_0^\infty t^{p-1}\cp(E_{t},E_{bt}) \,dt
	\le \frac{-\log b}{(1-b)^p} \int_E g^p_u \, d\mu.
\end{equation}
\end{lem}

\begin{proof}
As $g_u=g_{|u|}$ a.e., we may assume that $u \ge 0$.
For $t>0$, let 
\[
u_t=\min\{(u-t)_\limplus, (a-1)t\}
\]
be the truncations of $u$ at levels $t$ and $at$, $t>0$.
Then the function $v_t:=u_t/(a-1)t$ is admissible in the definition of 
$\cp(E_{at},E_t)$ and
$g_{v_t}= g_u \chi_{\{t<u<at\}}/(a-1)t$ a.e.
Using Fubini's theorem we get that
\begin{align*}
\int_0^\infty t^{p-1}\cp(E_{at},E_t) \,dt
     &\le \int_0^\infty \biggl( \frac{1}{(a-1)t} \biggr)^p
           t^{p-1} \int_X g_u^p \chi_{\{t<u<at\}} \,d\mu \,dt \\
     &= \frac{1}{(a-1)^p} \int_X g_u(x)^p \int_{u(x)/a}^{u(x)} \,
           \frac{dt}{t} \, d\mu(x) \\
     &= \frac{\log a}{(a-1)^p} \int_X g_u^p \, d\mu,
\end{align*}
which proves \eqref{eq-lem-Choq-Mazya-a}. 
(To get the last equality we used the fact that $g_u=0$ a.e.\ 
in $\{x : u(x)=0\}$.)
The second inequality follows by the substitution $s=bt$.
\end{proof}

It follows directly from the definition that
$\cp(E_{t},E) \le \cp(E_{t},E_{bt})$ and hence the capacity in the left-hand 
side of \eqref{eq-lem-Choq-Mazya-b} can be replaced by $\cp(E_{t},E)$. 
Letting $b=1/p$ yields the following result.

\begin{cor} \label{cor-Choq-int-est}
\textup{(Maz{\cprime}ya's capacitary inequality)}
It is true that for all $u\in \Np_0(E)$,
\begin{equation} \label{eq-cor-Choq}
     \int_0^\infty t^{p-1}\cp(\{x: |u(x)|>t\},E) \,dt
	\le \frac{p^p \log p}{(p-1)^p} \int_E g^p_u \, d\mu.
\end{equation}
\end{cor}

Using the notation introduced above, \eqref{eq-cor-Choq}
can be stated as
\[
    \int_E |u|^p \,d\cp(\,\cdot\,,E) 
	\le \frac{p^{p+1} \log p}{(p-1)^p} \int_E g^p_u \, d\mu.
\]

By minimizing the constant on the right-hand side in \eqref{eq-lem-Choq-Mazya-b}
for $b\in(0,1)$, one can optimize the result.
An easy calculation shows that the minimum is attained when 
$1/b-1=-p\log b$.

In Section~2.3.1 in Maz\cprime ya~\cite{Mazya-Sob}, the inequality
\eqref{eq-cor-Choq} was proved with the constant $p^p/(p-1)^{p-1}$
(for unweighted $\R^n$).
See also Maz{\cprime}ya~\cite{Mazya-JFA}.
Note that $\log p<p-1$ for all $p>1$ and is comparable to $p-1$
for $p$  close to 1, 
while for large $p$, $\log p \ll p-1$.

\begin{proof}[Proof of Theorem~\ref{thm-Choq-int-est}]
As $\K_{\psi,f} = f+\K_{\psi-f,0}$ we can assume,
without loss of generality, that $f \equiv 0$.

Assume first that there is some $\ut \in \K_{\psi,0}$.
Then  $u:=\max\{\ut,\psi\}= \ut$ q.e.\ in $E$, and thus
also $u \in \K_{\psi,0}$.
Hence, by Corollary~\ref{cor-Choq-int-est} we have
\begin{align*}
     \int_0^\infty t^{p-1}\cp(\{x: \psi(x)>t\},E) \,dt 
	&\le      \int_0^\infty t^{p-1}\cp(\{x: u(x)>t\},E) \,dt \\
	& \le C \int_E g^p_u \, d\mu
       < \infty.
\end{align*}

Conversely, assume that \eqref{eq-Choq-int-finite} holds.
As $\cp(\{x: \psi(x)>t\},E)$ is nonincreasing with respect to $t$,
it follows that $\cp(\{x: \psi(x)>t\},E)< \infty$ for all $t>0$.
Thus we can find $u_k \in \Np_0(E)$, for $k \in \Z$, such that
$\chi_{\{\psi > 2^k\}} \le u_k \le 1$
and 
\begin{equation}
     \int_E g^p_{u_k}\,d\mu < \cp(\{x: \psi(x)>2^k\},E)+ 2^{-|k|-(k+1)p}.
\label{eq-uk-almost-cap}
\end{equation}
Let 
\begin{alignat}{3}
      v_N &= \sup_{k \le N} 2^{k+1} u_k, 
	& \quad  g_N &= \sup_{k \le N} 2^{k+1} g_{u_k}, &\quad &N \in \Z, 
\nonumber\\
      v &= \sup_{k \in \Z} 2^{k+1} u_k = \sup_{N \in \Z} v_N,
	& g & = \sup_{k \in \Z} 2^{k+1} g_{u_k} = \sup_{N \in \Z} g_N .
\nonumber
\end{alignat}
(Here we take the same representative of $g_{u_k}$ in all places.)
Then $v \ge 2^{k+1}$ when $\psi > 2^k$, in particular
when $2^k < \psi \le 2^{k+1}$, $k\in \Z$,
from which it follows that $v \ge \psi$ in $E$.

By Lemma~1.52 in Bj\"orn--Bj\"orn~\cite{BBbook}, $g_N$
is a \p-weak upper gradient of $v_N$.
Moreover
\begin{align*}
	\int_E g^p \,d\mu
	&=   \int_E \Bigl( \sup_{k \in \Z} 2^{k+1} g_{u_k}\Bigr)^p \, d\mu \\
	&\le  \int_E \sum_{k=-\infty}^\infty ( 2^{k+1} g_{u_k})^p \, d\mu 
	=    \sum_{k=-\infty}^\infty 2^{(k+1)p} \int_E   g_{u_k}^p \, d\mu.
\end{align*}
Using \eqref{eq-uk-almost-cap} we obtain
\begin{align*}
	\int_E g^p \,d\mu
	&<    \sum_{k=-\infty}^\infty 2^{(k+1)p}
	      (\cp(\{x: \psi(x)>2^k\},E)+  2^{-|k|-(k+1)p}) \\
	&\le 3 +    \sum_{k=-\infty}^\infty 2^{(k+1)p} 
	       \vint_{2^{k-1}}^{2^k}
		\biggl(\frac{t}{2^{(k-1)}}\biggr)^{p-1}
		\cp(\{x: \psi(x)>t\},E) 
		\,dt \\
	&= 3 + 4^p\int_0^\infty t^{p-1}\cp(\{x: \psi(x)>t\},E) \,dt.
\end{align*}
The assumption \eqref{eq-Choq-int-finite} now yields that
$\int_E g^p \,d\mu<\infty$.
Since $g_N \nearrow g$ pointwise in $X$, dominated convergence implies that
$g_N \to g$ in $L^p(X)$.
Monotone convergence and \eqref{PI-NP0} 
then yield
\begin{equation} \label{eq-limit}
      \int_E |v|^p\,d\mu 
	= \lim_{N \to \infty} \int_E |v_N|^p\,d\mu 
	\le C_E \int_E g_{v_N}^p \,d\mu 
	\le C_E \int_E g^p \,d\mu <\infty.
\end{equation}
Thus $v_N \to v$ both pointwise and in $L^p(X)$, by dominated convergence.
Proposition~\ref{prop-Fuglede-consequence}
then shows that $v \in \Np(X)$.
As $v =0$ in $X \setm E$, we get
$v \in \Np_0(E)$ and therefore $v \in \K_{\psi,0}$.
\end{proof}

If the obstacle $\psi \in \Np(E)$, then there is a much 
easier criterion for when $\K_{\psi,f} \ne \emptyset$.

\begin{prop} \label{prop-obst-solve-crit}
Let $f,\psi \in \Np(E)$
\textup{(}or more generally $f,\psi \in \Dp(E)$ be 
such that $f-\psi \in \Np(E)$\textup{)}.
Then $\K_{\psi,f}\ne \emptyset$  if and only if\/
$(\psi-f)_\limplus \in \Np_0(E)$.
\end{prop}

\begin{proof}
Assume first that there is $u \in \K_{\psi,f}$.
Then
\[
    0 \le (\psi-f)_\limplus \le (u-f)_\limplus 
    \quad \text{q.e.}
\]
Hence, $(\psi-f)_\limplus \in \Np_0(E)$, by Lemma~\ref{lem-police}.

Conversely, assume that 
$(\psi-f)_\limplus \in \Np_0(E)$ and let $u=\max\{\psi,f\}$.
Then $u-f=(\psi-f)_\limplus \in \Np_0(E)$.
As $u \ge \psi$ in $E$, it follows that $u \in \K_{\psi,f}$.
\end{proof}

\begin{remark}
In this section, we only used the Poincar\'e inequality for $\Np_0$ 
and the assumption $\Cp(X \setm E)>0$ in
the proof of Theorem~\ref{thm-Choq-int-est} (apart from in some examples).
More specifically these assumptions were used in 
\eqref{eq-limit}, where it is enough if
\eqref{PI-NP0} holds for the specific $E$ under consideration.
Neither of these two assumptions can be dropped for Theorem~\ref{thm-Choq-int-est},
which is seen by letting $\psi\equiv\infty$ and $f\equiv0$
and either consider $E=X_\limplus$ in Example~\ref{ex-Xpm},
or an arbitrary $E$ such that $\Cp(X \setm E)=0$ (and $\mu(X)>0$).
Note that in both cases $\cp(E,E)=0$ so that the integral
in \eqref{eq-Choq-int-finite} converges while $\K_{\psi,f}=\emptyset$.

All other results in this section hold without Poincar\'e inequality.
\end{remark}

\section{Nontriviality of the obstacle problem and of \texorpdfstring{$\Np_0$}{}}

\label{sect-nontriv}

\emph{Assume in this section 
that $X$ is complete and supports a\/ $(1,p)$-Poincar\'e inequality,
that $\mu$ is doubling and that $p>1$.}

\medskip

These assumptions are needed to be able to use some results from
fine potential theory.

In the obstacle problem it is natural to ask when the obstacle problem
is trivial, i.e.\ when all functions $v \in \K_{\psi_1,\psi_2,f}$
agree q.e.
This happens in particular when $\Np_0(E)$ is trivial.
In the double obstacle problem it can happen also in other cases, e.g.\
if $\psi_1 \equiv \psi_2$ or in 
Examples~\ref{ex-doubl-obst-difficult} and~\ref{ex-infty-obst-nonsol}.
For the single obstacle problem the situation is simpler and
we have the following characterization.

\begin{prop} \label{prop-obst-pr-trivial}
Let $E \subset X$ be a bounded measurable set with $\Cp(X \setm E)>0$,
$f \in \Dp(E)$ and $\psi:E \to \eR$.
Then $\K_{\psi,f}$ is trivial  \textup{(}in the sense that $u=v$ q.e.\ whenever
$u,v \in \K_{\psi,f}$\textup{)} if and only if
either $\K_{\psi,f} = \emptyset$ or $\Np_0(E)$ is trivial\/
\textup{(}i.e.\ $u=0$ q.e.\ for all $u \in \Np_0(E)$\textup{)}.
\end{prop}

Observe that Adams' criterion  (Theorem~\ref{thm-Choq-int-est})
shows when $\K_{\psi,f} = \emptyset$.
Note also that if $\K_{\psi,f}$ is nonempty but trivial, then 
$\K_{\psi,f}=\{u:u=f \text{ q.e.}\}$.

\begin{proof}
If $\K_{\psi,f}=\emptyset$, then the equivalence is clear.
Assume therefore that $\K_{\psi,f}\ne\emptyset$.
If $\Np_0(E)$ is trivial, then all $v \in \K_{\psi,f}$ agree
with $f$ q.e., and thus $\K_{\psi,f}$ is trivial.

Conversely assume that $\Np_0(E)$ is nontrivial.
Then there is $u \in \Np_0(E)$ such that $\Cp(\{x : u(x) \ne 0\})>0$.
Let $v \in \K_{\psi,f}$ and $w=v+|u|$.
Then $w \in \K_{\psi,f}$ and as $w$ and $v$ do not agree q.e.\ the
nontriviality of $\K_{\psi,f}$ follows.
\end{proof}

Our aim is now to characterize when $\Np_0(E)$ is trivial.
We get the following result.
(Definitions of the involved concepts follow below.)

\begin{thm} \label{thm-nontrivial2}
Let $E \subset X$ be arbitrary. 
Then the following are equivalent\/\textup{:}
\begin{enumerate}
\item \label{nontrivial}
$\Np_0(E)$ is nontrivial\textup{;}
\item \label{fine}
$E$ contains a nonempty finely open set, 
or in other terms\/ $\fineint E \ne \emptyset$\textup{;}
\item \label{thin}
there is a point $x \in E$ such that $X \setm E$ is thin at $x$\textup{;}
\item \label{cap<cap}
there are a point $x \in E$ and $s>0$ such that 
\[
\cp(B(x,s) \setm E,B(x,2s)) < \cp(B(x,s),B(x,2s)).
\]
\end{enumerate}
\end{thm}

Note that if $\mu(E)=0$ then all the statements are false,
since in this case $f \in \Np_0(E)$ implies that $f=0$ a.e.\ in $X$,
and hence
$f=0$ q.e.\ in $X$, i.e.\ $\Np_0(E)$ is trivial.

The following result gives a more precise description of $\Np_0(E)$
and will be used to establish Theorem~\ref{thm-nontrivial2}.

\begin{thm} \label{thm-fineint}
Let $E \subset X$ be arbitrary. 
Then 
\[
    \Np_0(E)= \Np_0(\fineint E).
\]
\end{thm}

Here we follow our convention that functions in $\Np_0$ can be extended
by zero q.e.
Observe that we do not require $E$ to be measurable
in Theorems~\ref{thm-nontrivial2} and~\ref{thm-fineint}.
See also Section~\ref{sect-Rn} for some further consequences
of Theorem~\ref{thm-fineint} in the special case $X=\R^n$.

\begin{cor} \label{cor-fineint}
Let  $E,E_0 \subset X$ be measurable sets such that
\[ 
     \fineint E \subset E_0 \subset E.
\]
If $f\in \Dp(E)$ and $\K_{\psi_1,\psi_2,f}(E) \ne \emptyset$, then 
\[
   \K_{\psi_1,\psi_2,f}(E) = \K_{\psi_1,\psi_2,f}(E_0).
\]
\end{cor}

Of course, the main interest is when $E_0=\fineint E$. 
But here, contrary to Theorem~\ref{thm-fineint}, we also need
measurability and we do not know  in general 
if $\fineint E$ is always measurable, 
cf.\ Section~\ref{sect-Rn}.

\begin{remark}
Note that it is possible 
to have $\K_{\psi_1,\psi_2,f}(E)=\emptyset\ne\K_{\psi_1,\psi_2,f}(E_0)$. 
Indeed, this happens  exactly if
$\K_{\psi_1,\psi_2,f}(E_0) \ne \emptyset$
and 
\begin{equation} \label{eq-contr}
   \Cp(\{x \in E \setm E_0 : \psi_1(x) > f(x) \text{ or } 
       \psi_2(x) < f(x) \} >0.
\end{equation}
To see this, note that since    $\Np_0(E)= \Np_0(E_0)$
it follows that any function in $\Np_0(E)$ is $0$ q.e.\ in 
$E \setm E_0$. 
Hence, if $u \in \K_{\psi_1,\psi_2,f}(E)$, then $u=f$ q.e.\ 
in $E \setm E_0$ which is impossible
if $\psi_1 \le u \le \psi_2$ q.e.\ at the same time as
\eqref{eq-contr} holds.
Conversely, if $u \in\K_{\psi_1,\psi_2,f}(E_0)$
and \eqref{eq-contr} fails, then we extend $u$ as $f$ in
$E \setm E_0$, so that $u \in \K_{\psi_1,\psi_2,f}(E)$
showing that $\K_{\psi_1,\psi_2,f}(E)$ is nonempty.
\end{remark}

To make the above results precise we need  a few more
definitions. 
See Appendix~\ref{app-capp} for the definition and some properties
of the variational capacity $\cp$.

\begin{deff}
A set $A$ is \emph{thin} at $x$ if
\begin{equation}             \label{def-thin}
       \int_0^1 \biggl(\frac{\cp(A \cap B(x,r),B(x,2r))}
        {\cp(B(x,r),B(x,2r))}\biggr)^{1/(p-1)}
                  \frac{dr}{r} <\infty.
\end{equation}
A set $A$ is \emph{finely open} if $X \setm A$ is thin at all $x \in A$.
Using the monotonicity and subadditivity of the capacity, it is easy to verify
that finely open sets form a topology on $X$.
The \emph{fine interior} $\fineint E$ of $E$ is the largest 
finely open set contained in $E$.
\end{deff}

Since our variational capacity is the same as the
one in 
Heinonen--Kilpel\"ainen--Martio~\cite{HeKiMa}
(see Bj\"orn--Bj\"orn~\cite{BBvarcap} for a proof of this fact),
we see that this definition coincides with the definition
in \cite{HeKiMa}, p.\ 221, when $X$ is weighted $\R^n$
with a \p-admissible weight.
If $X=\R^n$ (unweighted) then it 
is also equivalent to Definition~2.47
in Mal\'y--Ziemer~\cite{MaZi}.

In the definition of thinness we make the convention that the integrand
is 1 whenever $\cp(B(x,r),B(x,2r))=0$.
This happens e.g.\ if $X=B(x,2r)$ is bounded,
but never e.g.\ if $r < \frac{1}{2}\diam X$. 
Note that thinness is a local property, i.e.\ if $\de >0$,
then $E$ is thin at $x$
if and only if $E\cap B(x,\de)$ is thin at
$x$.

To prove Theorem~\ref{thm-fineint}, we shall use the
following result which was obtained by J.~Bj\"orn~\cite{JB-pfine}, Theorem~4.6,
and independently by Korte~\cite{korte08}, Corollary~4.4
(the result can also be found in Bj\"orn--Bj\"orn~\cite{BBbook}, Theorem~11.40).
A function $u$, defined on a finely open set $U$, is
\emph{finely continuous} at $x\in U$ if for every
$\eps>0$ there exists a finely open set $V\ni x$ such that
$|u(y)-u(x)|<\eps$ for all $y\in V$
(in particular $u(x) \in \R$).

\begin{thm} \label{thm-qcont-fine}
Every $u\in\Np(X)$ is finely continuous at q.e.\  $x \in X$.
\end{thm}

\begin{proof}[Proof of Theorem~\ref{thm-fineint}]
Let $u \in \Np_0(E)$ 
and extend $u$ by $0$ on $X \setm E$, so that $u \in \Np(X)$.

Let $G=\{x \in E: u(x) \ne0\}$.
By Theorem~\ref{thm-qcont-fine}, there exists a set $F$ with $\Cp(F)=0$ such that
$u$ is finely continuous at every $x\in X\setm F$.
Hence, for every $x\in G\setm F$, there exists a finely open neighbourhood 
$V_x$ of $x$ such that $|u-u(x)|<|u(x)|$ in $V_x$.
Note that $u\ne0$ in $V_x$ and hence $V_x\subset G \subset E$.

Letting 
$V=\bigcup_{x\in G\setm F} V_x$,
we obtain a finely open set $V$ such that $G\setm F \subset V \subset E$.
As $X\setm V \subset (X\setm G)\cup F$, we see that $u=0$ q.e.\ in $X \setm V$, 
and hence $u \in \Np_0(V) \subset \Np_0(\fineint E)$.
Since $u\in\Np_0(E)$ was arbitrary, this shows that 
$\Np_0(E)\subset\Np_0(\fineint E)$.

The converse inclusion is obvious.
\end{proof}

\begin{proof}[Proof of Theorem~\ref{thm-nontrivial2}]
$\neg$ \ref{fine} $\imp$ $\neg$ \ref{nontrivial}
By Theorem~\ref{thm-fineint},
\[
   \Np_0(E)=\Np_0(\fineint E)=\Np_0(\emptyset),
\]
and thus $\Np_0(E)$ is trivial.

\ref{fine} $\imp$ \ref{thin}
Let $G \subset E$ be a nonempty finely open set.
Then $X \setm E \subset X\setm G$ is thin at every $x \in G$.

\ref{thin} $\imp$ \ref{cap<cap}
Let for simplicity $B_r=B(x,r)$.
Since 
\[
       \int_0^1 \biggl(\frac{\cp(B_r \setm E,B_{2r})}
               {\cp(B_r,B_{2r})}\biggr)^{1/(p-1)}
                  \frac{dr}{r} <\infty.
\]
We see that 
\[
   \liminf_{r\to 0\limplus} 
      \frac{\cp(B_r \setm E,B_{2r})}{\cp(B_r,B_{2r})} =0
\]
and \ref{cap<cap} follows.
(Actually the limit exists and equals $0$, but we will not need that here.)

\ref{cap<cap} $\imp$ \ref{nontrivial}
Let for simplicity $B_r=B(x,r)$.
Theorem~\ref{thm-cp}\,\ref{cp-Choq-E-sum}
implies that
\[
\cp(B_s,B_{2s}) = \sup_{t<s} \cp(B_t,B_{2s}).
\]
Hence, there exists $t<s$ such that
\[
\cp(B_s\setm E,B_{2s}) < \cp(B_t,B_{2s}).
\]
Thus there exists a function $h \in \Np(X)$ such that
$0 \le h \le 1$, $h=1$ on $B_s \setm E$, $h=0$ on $X \setm B_{2s}$
and $\|g_h\|_{L^p(X)}^p < \cp(B_t,B_{2s})$.
Let $F=\{x \in B_t : h(x)<1\}$.
If $\Cp(F)$ were $0$, then we would have
\[
    \cp(B_t,B_{2s}) \le \|g_{h+\chi_F}\|_{L^p(X)}^p
    =  \|g_h\|_{L^p(X)}^p < \cp(B_t,B_{2s}),
\]
a contradiction.
Thus $\Cp(F)>0$.

Let now $f$ be a Lipschitz function such that $0 \le f \le 1$, 
$f=1$ on $B_t$ and $f=0$ on $X \setm B_s$.
Let further $k=(f-h)_\limplus \in \Np(X)$.
It follows directly that $k=0$ 
on $(B_s \setm E) \cup (X \setm B_s) \supset X \setm E$,
and thus $k \in \Np_0(E)$.
Since $F=\{x \in B_t : k(x)>0\}$
and $\Cp(F)>0$, we see that 
$k \not\sim 0$, i.e.\ $\Np_0(E)$
is nontrivial.
\end{proof}

The following characterization of the fine interior is useful in applications
and examples, as it is easier and more explicit to verify \eqref{def-thin}
for $X\setm E$ than for $X\setm\fineint E$, see Examples~\ref{ex-emental}
and~\ref{ex-emental-p=n}.
Analogues of this result in $\R^n$ can be found in Theorem~2.136 in
Mal\'y--Ziemer~\cite{MaZi} and in Theorem~12.5 in 
Heinonen--Kilpel\"ainen--Martio~\cite{HeKiMa}.
The proof given here is different and does not use any characterization
of finely open sets by superharmonic functions.

\begin{prop}               \label{prop-fineint-locally}
Let $E\subset X$ be arbitrary.
Then $x\in \fineint E$ if and only if $x\in E$ and 
$X\setm E$ is thin at $x$.
\end{prop}

\begin{proof}
Let $E_0=\fineint E \subset E$.
If $x\in E_0$, then by definition $X\setm E_0$ (and hence also $X\setm E$)
is thin at $x$.

Conversely, assume that $X\setm E$ is thin at $x \in E$, i.e.
\[
       \int_0^1 \biggl(\frac{\cp(B_r\setm E,B_{2r})}
        {\cp(B_r,B_{2r})}\biggr)^{1/(p-1)}
                  \frac{dr}{r} <\infty.
\]
where we abbreviate $B_r=B(x,r)$.
For $0<r<1$, let $F_r$ be the fine closure of $B_r\setm E$, i.e.\ 
the smallest finely closed set containing $B_r\setm E$.
Then $B_r\setm F_r$ is finely open and contained in $E$. 
To conclude the proof, it suffices to show that $F_r$ is thin
at $x$, as then $(B_r\setm F_r)\cup\{x\}$ is also finely open and
contained in $E$, which implies that 
\[
(B_r\setm F_r)\cup\{x\}\subset E_0,
\]
and in particular $x\in E_0$.

We shall show that $B_r\cap F_r$ is thin at $x$.
Since $X\setm E$ is thin at $x$, it suffices to show that
\[
\cp(B_\rho\cap F_r,B_{2\rho})
\le \cp(B_\rho\setm E,B_{2\rho})
  \quad \text{for } 0 < \rho \le r.
\]
First, we note that $B_\rho\cap F_r\subset F_\rho$.
Indeed, $F_\rho\cup (X\setm B_\rho)$ is finely closed and contains 
$X\setm E$ (and hence also $F_r$).
It follows that
\[
B_\rho\cap F_r \subset B_\rho\cap (F_\rho\cup(X\setm B_\rho))
\subset F_\rho.
\]
This and Corollary~4.5 in J.~Bj\"orn~\cite{JB-pfine} 
(or Corollary~11.39 in Bj\"orn--Bj\"orn~\cite{BBbook})
now yield that
\[
\cp(B_\rho\cap F_r,B_{2\rho})
\le \cp(F_\rho,B_{2\rho})
= \cp(B_\rho\setm E,B_{2\rho}).
\]
From this and the thinness of $X\setm E$ at $x$, we  conclude that $F_r$ 
is thin at $x$, which finishes the proof.
\end{proof}

The following direct consequence of Proposition~\ref{prop-fineint-locally} 
characterizes fine closures and fine boundaries, cf.\ 
Definition~2.134 in Mal\'y--Ziemer~\cite{MaZi}.

\begin{cor}
Let $E\subset X$ be arbitrary. 
Then the fine closure of $E$ is the set
\[
E\cup \{x\in X\setm E: E \text{ is not thin at } x\}
\]
and the fine boundary of $E$ is 
\begin{align*}
&\{x\in E: X\setm E \text{ is not thin at } x\} 
\cup \{x\in X\setm E: E \text{ is not thin at } x\}.
\end{align*}
In particular, the fine boundary of $E$ is a subset of $\bdry E$.
\end{cor}

\section{Comparing obstacle problems}
\label{sect-compare}

If the boundary data belong to $\Dp(\Om)$ for some open $\Om\supset E$,
then we have two possible definitions of obstacle problems on $E$, 
viz.\ Definition~\ref{deff-obst-E} and \eqref{eq-def-K'}.
The following lemma relates the admissible sets in these two definitions.

\medskip

\emph{In this section we
assume 
that $E \subset X$ is a bounded measurable set such that $\Cp(X \setm E)>0$.}

\begin{lem} \label{lem-K-K'}
If $f\in \Dp(\Om)$ for some open set\/ $\Om\supset E$, then 
$\K'_{\psi_1,\psi_2,f}=\K_{\psi_1,\psi_2,f}$.
\end{lem}

Recall that $\K'_{\psi_1,\psi_2,f}$ was defined in~\eqref{eq-def-K'}.
By saying that $\K'_{\psi_1,\psi_2,f}=\K_{\psi_1,\psi_2,f}$
we really mean
that $\{f|_E : f \in \K'_{\psi_1,\psi_2,f}\}=\K_{\psi_1,\psi_2,f}$
and that every $f \in \K_{\psi_1,\psi_2,f}$ corresponds to a
unique (up to capacity zero) $\ft \in \K'_{\psi_1,\psi_2,f}$.
Note that already in Section~\ref{sect-obst} we observed
that $\K'_{\psi_1,\psi_2,f}=\K_{\psi'_1,\psi'_2,f} (\Om)$,
where $\psi_1'$ and $\psi_2'$ are given by \eqref{eq-psi'j}.

\begin{proof}
Clearly, $\K'_{\psi_1,\psi_2,f}\subset\K_{\psi_1,\psi_2,f}$.
To prove the other inclusion, let $v\in\K_{\psi_1,\psi_2,f}$, 
i.e.\ $v\in\Dp(E)$ and $v-f=w\in\Np_0(E)$.
Then $w$ (extended by zero outside of $E$) belongs to $\Np(\Om)$
and hence $v=f+w\in\Dp(\Om)$, from which the result follows.
\end{proof}

Note that even though $\K'_{\psi_1,\psi_2,f}=\K_{\psi_1,\psi_2,f}$ for 
$f\in \Dp(\Om)$, the minimal \p-weak upper gradients considered in these
two 
obstacle problems are different.
The minimal \p-weak upper gradient in the $\K_{\psi_1,\psi_2,f}$-obstacle problem 
is taken with respect to $E$ and is in general smaller than the minimal
\p-weak upper gradient with respect to $\Om$ or $X$, considered in
the $\K'_{\psi_1,\psi_2,f}$-obstacle problem.
 
\begin{example}
\label{ex-R-Q}
Let, as in Example~\ref{ex-R-Q-1}, $X=\R$ and $E=(0,1)\setm\Q$,
and recall that
the minimal \p-weak upper gradient 
(and thus the \p-energy integral) taken with respect to $E$ is zero
for every function on $E$,
while the minimal \p-weak upper gradient with respect to 
$\R$ is just the modulus of the distributional derivative.

However, since $(0,1)\setm E$ is dense in $(0,1)$ and all functions in 
$\Np(X)$ are absolutely continuous, the space $\Np_0(E)$ is trivial and 
so is $\K_{\psi_1,\psi_2,f}$, cf.\ Proposition~\ref{prop-obst-pr-trivial}.
Hence, the  only solution (if it exists) of both the $\K_{\psi_1,\psi_2,f}$- 
and the $\K'_{\psi_1,\psi_2,f}$-obstacle problem is $f$ itself.
\end{example}

The last observation in Example~\ref{ex-R-Q}
holds in much more generality, as
we shall now see.
Recall that \p-path almost open sets were introduced in 
Definition~\ref{def-p-path-open}.

\begin{thm}  \label{thm-same-obst-pr}
Assume 
that $X$ is complete and supports a\/ $(1,p)$-Poincar\'e inequality,
that $\mu$ is doubling and that $p>1$.
Let $E_0$ be a 
\p-path almost open measurable set
such that\/ $\fineint E \subset E_0 \subset E$,
and let $f\in\Dp(E)$ and $\psi_j:E \to \eR$, $j=1,2$,  be such 
that $\K_{\psi_1,\psi_2,f}(E) \ne \emptyset$.
Then the solutions of the $\K_{\psi_1,\psi_2,f}(E)$-problem
coincide with the solutions of the $\K_{\psi_1,\psi_2,f}(E_0)$-problem.

Moreover, if $\mu(E\setm E_0)=0$ then also the \p-energies associated 
with these two problems coincide.
In particular, this holds if $\mu(\bdry E)=0$.

If $f\in\Dp(\Om)$ for some open set\/ $\Om\supset E$, then
the above solutions coincide with the solutions of the 
$\K'_{\psi_1,\psi_2,f}(E)$-problem.
\end{thm}

Of course, the main interest is when $E_0=\fineint E$. 
But as we do not know if $\fineint E$ is always measurable
and \p-path almost open, we have given the formulation above.
See, however, Section~\ref{sect-Rn} and Theorem~\ref{thm-E-E0-probl-same} 
for an improvement in the case $X=\R^n$.

Note that even if the solutions coincide, the corresponding \p-energies
can in general be different for these obstacle problems.
Indeed, even though 
\[
g_{u,E_0}=g_{u,E} \text{ a.e.\ in } E_0
\]
for every $u\in\K_{\psi_1,\psi_2,f}(E)$, by Corollary~\ref{cor-guG-guE},
we only get 
\[
     \int_{E_0} g_{u,E_0}^p \, d\mu 
     = \int_{E_0} g_{u,E}^p \, d\mu 
     \le  \int_{E} g_{u,E}^p \, d\mu 
\]
with strict inequality unless $g_{u,E}=0$ a.e.\ in $E \setm E_0$
(which holds in particular if $\mu(E \setm E_0)=0$).

If $f \in \Dp(\Om)$ for some open $\Om\supset E$,
then 
\[
     g_{u,E_0}=g_{u,E}=g_u \text{ a.e.\ in } E_0
\]
for every $u\in\K_{\psi_1,\psi_2,f}(E)$, by Corollary~\ref{cor-guG-guE},
but we  have only
\[
g_{u,E}=g_{f,E}\le g_f=g_u \text{ a.e.\ in } E\setm E_0
\]
for those $u$, and the inequality in the middle can be strict,
see Example~\ref{ex-R-Q} where $E_0$ is empty.
Thus, the two \p-energies 
$\int_E g_{u,E}^p \, d\mu$ and 
$\int_E g_{u}^p \, d\mu$ 
will coincide only if $g_{f,E}=g_f$ a.e.\ in $E\setm E_0$,
in particular if $\mu(E\setm E_0)=0$.

\begin{proof}
To simplify the notation, we omit 
the subscripts $\psi_1$, $\psi_2$ and $f$ and only write $\K(E)$, $\K(E_0)$ 
and $K'(E)$ in this proof.

By Corollary~\ref{cor-fineint},  
we have $\K(E)=\K(E_0)$.
Since $E_0$ is \p-path almost open, Corollary~\ref{cor-guG-guE} 
(with $X$ replaced by $E$) yields that for all $v\in\K(E)=\K(E_0)$,
\begin{equation}  \label{eq-g-E=g-E0}
g_{v,E} = g_{v,E_0} \quad \text{a.e.\ in } E_0.
\end{equation}
Moreover, as $v-f\in\Np_0(E_0)$, we have $v=f$ q.e.\ in $E\setm E_0$
and hence
\begin{equation*}
g_{v,E} = g_{f,E} \quad \text{a.e.\ in } E\setm E_0.
\end{equation*}
Similarly, if $f\in\Dp(\Om)$ for some open $\Om\supset E$, 
then $\K'(E)=\K(E)$, by Lemma~\ref{lem-K-K'},
and for all $v\in\K(E)$,
\begin{equation}  \label{eq-g-E=g}
g_{v,E} = g_{v}  \text{ a.e.\ in } E_0
\quad \text{and} \quad
g_{v} = g_{f} \text{ a.e.\ in } \Om\setm E_0,
\end{equation}
again by Corollary~\ref{cor-guG-guE} (with $X$ replaced by $\Om$).

Let $u$ be a solution of the $\K(E_0)$-problem.
Then \eqref{eq-g-E=g-E0} implies that for all $v\in\K(E_0)=\K(E)$,
\begin{align}
\int_{E_0} g_{u,E}^p\,d\mu 
&= \int_{E_0} g_{u,E_0}^p\,d\mu 
\le \int_{E_0} g_{v,E_0}^p\,d\mu = \int_{E_0} g_{v,E}^p\,d\mu.
\label{eq-ineq-E0}
\end{align}
Similarly, if $f \in \Dp(\Om)$ for some open $\Om\supset E$,
and $u'$ is a solution of the $\K'(E)$-problem, then
\eqref{eq-g-E=g} implies that for all $v\in\K'(E)=\K(E)$,
\begin{align} \label{eq-ineq-E}
\int_{E_0} g_{u',E}^p\,d\mu 
&= \int_{E_0} g_{u'}^p\,d\mu 
= \int_{E} g_{u'}^p\,d\mu - \int_{E\setm E_0} g_{f}^p\,d\mu \\
\nonumber
&\le \int_{E} g_{v}^p\,d\mu - \int_{E\setm E_0} g_{f}^p\,d\mu
= \int_{E_0} g_{v}^p\,d\mu = \int_{E_0} g_{v,E}^p\,d\mu.
\end{align}
Adding $\int_{E\setm E_0} g_{f,E}^p\,d\mu$ to
both sides in \eqref{eq-ineq-E0} and \eqref{eq-ineq-E}
shows that both $u$ and $u'$ are solutions of the $\K(E)$-problem
(the latter assuming that $f \in \Dp(\Om)$).
By uniqueness, they coincide q.e.\ in $E$ and are the only (up to q.e.)
solutions of the $\K(E)$-obstacle problem.
\end{proof}

\section{\texorpdfstring{$\R^n$}{}}
\label{sect-Rn}

The situation gets somewhat simpler in $\R^n$ (unweighted).
In this case Theorem~2.144 in Mal\'y--Ziemer~\cite{MaZi} (which goes back
to Fuglede~\cite{fuglede}) shows that every finely open set $G$ is 
\emph{quasiopen},
i.e.\ for every $\eps>0$ there exists an open set $V$ with $\Cp(V)<\eps$
such that $G\cup V$ is open.
In particular, the fine interior $\fineint E$ of every set $E\subset\R^n$
is quasiopen.

\begin{remark} \label{rmk-quasiopen}
If $p>n$, then quasiopen sets are open in $\R^n$, and thus
the quasiopen, finely open  and open sets coincide.
There are immediate consequences of this for the results in 
Section~\ref{sect-compare} which we leave to the reader to formulate explicitly.

If $1<p \le n$, then for every $x>0$ and $\eps>0$ there is an open
set $V \ni x$ with $\Cp(V)<\eps$, and thus $\{x\}$ is quasiopen.
Since not all sets are quasiopen, by Lemma~\ref{lem-quasiopen} and 
Remark~\ref{rmk-8}, it follows that the quasiopen sets do not form a topology.
\end{remark}

To be able to state Theorem~\ref{thm-same-obst-pr} without additional assumptions
on $E_0$, we recall the following results which 
hold in general metric spaces.

\begin{lem} 
\label{lem-quasiopen}
\textup{(Shanmugalingam~\cite{Sh-harm}, Remark~3.5)}
Every quasiopen set is \p-path open.
\end{lem}

\begin{lem} 
Every quasiopen set $G$ is measurable.
\end{lem}

\begin{proof}
For every $j=1,2,\ldots$, there is an open set $V_j$
such that $\Cp(V_j)< 1/j$ and $G_j:=G \cup V_j$ is open.
Let $A_j=G_j \setm V_j \subset G$, $A=\bigcup_{j=1}^\infty A_j \subset G$
and $E=\bigcap_{j=1}^\infty G_j \supset G$ which are all Borel sets.
Then $A \subset G \subset E$
and 
\[ 
    \mu(E \setm A) \le \mu (G_j \setm A_j) = \mu(V_j) 
    \le  \Cp(V_j) < 1/j
    \quad \text{for } j=1,2,\ldots.
\]
Letting $j \to \infty$ shows that $G$ is measurable.
\end{proof}

Hence, if $E\subset\R^n$ then $\fineint E$ is measurable and \p-path open,
and Theorem~\ref{thm-same-obst-pr} 
turns into Theorem~\ref{thm-E-E0-probl-same} in the introduction.
We also have the following consequence of Lemma~\ref{lem-Np0-0-qe}
and Theorem~\ref{thm-fineint}, which generalizes
Theorem~2.147 in Mal\'y--Ziemer~\cite{MaZi}.
See also Remark~2.148 in~\cite{MaZi} for another description of
$W^{1,p}_0(\Om)$ in $\R^n$.

\begin{prop}   \label{prop-char-Np0-qe-bdry}
Let $E\subset\R^n$ be arbitrary and $u\in\Np(\itoverline{E}^p)$,
where $\itoverline{E}^p$ is the fine closure of $E$.
Then $u\in\Np_0(E)$ if and only if $u=0$ q.e.\ on the fine boundary 
$\itoverline{E}^p\setm\fineint E$ of $E$.
\end{prop}

\begin{proof}
By the discussion at the beginning of this section, both $\fineint E$
and $\R^n\setm\itoverline{E}^p$ are \p-path open.
Lemma~\ref{lem-Np0-0-qe} with $E_1=\fineint E$ and 
$E_2=\itoverline{E}^p$ then yields that $u\in\Np_0(\fineint E)$ if and only if
$u=0$ q.e.\ on the fine boundary of $E$.
Theorem~\ref{thm-fineint} concludes the proof.
\end{proof}

In general metric spaces, the missing link is the implication that finely
open sets are quasiopen.
This is a part of fine potential theory on metric spaces which 
we plan to further develop in the future.

The following two examples illustrate some of the results in this paper,
in particular the special situation in $ \R^n$.
They provide us with a closed nowhere dense set $E\subset[0,1]^n\subset\R^n$ 
with almost full measure in $[0,1]^n$, but whose fine interior 
has full measure in $E$.
In particular, the fine boundary of $E$ has zero measure even though
the Euclidean boundary $\bdry E=E$.
This implies that for every $u\in\Dp(\R^n)$,
\[
g_{u,\fineint E} = g_{u,E} = g_{u,\R^n} = |\grad u| \quad \text{a.e.\ in } E,
\] 
where $\grad u$ is the distributional gradient of $u$,
and that energies and obstacle problems on $E$ and its fine interior coincide.
Examples~\ref{ex-emental} and~\ref{ex-emental-p=n} are for $1< p <n$
and $p=n$, respectively. 
By Remark~\ref{rmk-quasiopen} there are no similar examples for $p>n$.

Recall that  for $q,x \in \R^n$ and $r,s>0$,
\begin{equation} \label{cp-Rn-est1}
     \cp(B(q,s) \cap B(x,r),B(x,2r)) \le 
        \begin{cases} C(n,p) s^{n-p}, & \text{if } 1 < p<n, \\
                      C(n) \biggl( \log \displaystyle\frac{2r}{s} \biggr)^{1-n}, 
                                   & \text{if } p=n,
       \end{cases}
\end{equation}
and that
\begin{equation} \label{cp-Rn-est2}
    \cp(B(x,r),B(x,2r)) =C(n,p) r^{n-p}, \quad 1< p \le n,
\end{equation}
see Example~2.12 in Heinonen--Kilpel\"ainen--Martio~\cite{HeKiMa}.

\begin{figure}[t]
\begin{center}

\setlength{\unitlength}{0.0004in}
\begingroup\makeatletter\ifx\SetFigFont\undefined%
\gdef\SetFigFont#1#2#3#4#5{%
  \reset@font\fontsize{#1}{#2pt}%
  \fontfamily{#3}\fontseries{#4}\fontshape{#5}%
  \selectfont}%
\fi\endgroup%
{%
\begin{picture}(8266,8285)(0,-10)
\thicklines
\path(533,7733)(7733,7733)(7733,533)
	(533,533)(533,7733)
\put(1433,6833){\ellipse{200}{200}}
\put(2333,6833){\ellipse{200}{200}}
\put(3233,6833){\ellipse{200}{200}}
\put(4133,6833){\ellipse{200}{200}}
\put(5033,6833){\ellipse{200}{200}}
\put(5933,6833){\ellipse{200}{200}}
\put(6833,6833){\ellipse{200}{200}}
\put(1433,5933){\ellipse{200}{200}}
\put(1433,5033){\ellipse{200}{200}}
\put(2333,5033){\ellipse{200}{200}}
\put(3233,5933){\ellipse{200}{200}}
\put(5033,5933){\ellipse{200}{200}}
\put(6833,5933){\ellipse{200}{200}}
\put(6833,5033){\ellipse{200}{200}}
\put(5933,5033){\ellipse{200}{200}}
\put(5033,5033){\ellipse{200}{200}}
\put(4133,5033){\ellipse{200}{200}}
\put(3233,5033){\ellipse{200}{200}}
\put(1433,4133){\ellipse{200}{200}}
\put(3233,4133){\ellipse{200}{200}}
\put(1433,3233){\ellipse{200}{200}}
\put(2333,3233){\ellipse{200}{200}}
\put(3233,3233){\ellipse{200}{200}}
\put(4133,3233){\ellipse{200}{200}}
\put(5033,4133){\ellipse{200}{200}}
\put(6833,4133){\ellipse{200}{200}}
\put(5033,3233){\ellipse{200}{200}}
\put(5933,3233){\ellipse{200}{200}}
\put(6833,3233){\ellipse{200}{200}}
\put(1433,2333){\ellipse{200}{200}}
\put(3233,2333){\ellipse{200}{200}}
\put(5033,2333){\ellipse{200}{200}}
\put(6833,2333){\ellipse{200}{200}}
\put(6833,1433){\ellipse{200}{200}}
\put(5933,1433){\ellipse{200}{200}}
\put(5033,1433){\ellipse{200}{200}}
\put(4133,1433){\ellipse{200}{200}}
\put(3233,1433){\ellipse{200}{200}}
\put(2333,1433){\ellipse{200}{200}}
\put(1433,1433){\ellipse{200}{200}}
\put(5933,5933){\ellipse{500}{500}}
\put(4133,5933){\ellipse{500}{500}}
\put(2333,5933){\ellipse{500}{500}}
\put(2333,4133){\ellipse{500}{500}}
\put(5933,4133){\ellipse{500}{500}}
\put(5933,2333){\ellipse{500}{500}}
\put(4133,2333){\ellipse{500}{500}}
\put(2333,2333){\ellipse{500}{500}}
\put(4133,4133){\ellipse{1250}{1250}}
\put(1883,7283){\ellipse{80}{80}}
\put(2783,7283){\ellipse{80}{80}}
\put(2333,7283){\ellipse{80}{80}}
\put(1433,7283){\ellipse{80}{80}}
\put(983,7283){\ellipse{80}{80}}
\put(3233,7283){\ellipse{80}{80}}
\put(3683,7283){\ellipse{80}{80}}
\put(4583,7283){\ellipse{80}{80}}
\put(5033,7283){\ellipse{80}{80}}
\put(5483,7283){\ellipse{80}{80}}
\put(5933,7283){\ellipse{80}{80}}
\put(6383,7283){\ellipse{80}{80}}
\put(6833,7283){\ellipse{80}{80}}
\put(7283,7283){\ellipse{80}{80}}
\put(983,6833){\ellipse{80}{80}}
\put(1883,6833){\ellipse{80}{80}}
\put(2783,6833){\ellipse{80}{80}}
\put(3683,6833){\ellipse{80}{80}}
\put(4583,6833){\ellipse{80}{80}}
\put(5483,6833){\ellipse{80}{80}}
\put(6383,6833){\ellipse{80}{80}}
\put(7283,6833){\ellipse{80}{80}}
\put(7283,4133){\ellipse{80}{80}}
\put(4133,7283){\ellipse{80}{80}}
\put(983,6383){\ellipse{80}{80}}
\put(983,4133){\ellipse{80}{80}}
\put(4133,983){\ellipse{80}{80}}
\put(1433,6383){\ellipse{80}{80}}
\put(1883,6383){\ellipse{80}{80}}
\put(2333,6383){\ellipse{80}{80}}
\put(2783,6383){\ellipse{80}{80}}
\put(3233,6383){\ellipse{80}{80}}
\put(3683,6383){\ellipse{80}{80}}
\put(4133,6383){\ellipse{80}{80}}
\put(4583,6383){\ellipse{80}{80}}
\put(5033,6383){\ellipse{80}{80}}
\put(5483,6383){\ellipse{80}{80}}
\put(5933,6383){\ellipse{80}{80}}
\put(6383,6383){\ellipse{80}{80}}
\put(6833,6383){\ellipse{80}{80}}
\put(7283,6383){\ellipse{80}{80}}
\put(983,5933){\ellipse{80}{80}}
\put(1433,5483){\ellipse{80}{80}}
\put(983,5483){\ellipse{80}{80}}
\put(1883,5483){\ellipse{80}{80}}
\put(1883,5933){\ellipse{80}{80}}
\put(2783,5933){\ellipse{80}{80}}
\put(3683,5933){\ellipse{80}{80}}
\put(4583,5933){\ellipse{80}{80}}
\put(5483,5933){\ellipse{80}{80}}
\put(6383,5933){\ellipse{80}{80}}
\put(7283,5933){\ellipse{80}{80}}
\put(2333,5483){\ellipse{80}{80}}
\put(2783,5483){\ellipse{80}{80}}
\put(3233,5483){\ellipse{80}{80}}
\put(3683,5483){\ellipse{80}{80}}
\put(4133,5483){\ellipse{80}{80}}
\put(4583,5483){\ellipse{80}{80}}
\put(5033,5483){\ellipse{80}{80}}
\put(5483,5483){\ellipse{80}{80}}
\put(5933,5483){\ellipse{80}{80}}
\put(6383,5483){\ellipse{80}{80}}
\put(6833,5483){\ellipse{80}{80}}
\put(7283,5483){\ellipse{80}{80}}
\put(983,5033){\ellipse{80}{80}}
\put(1883,5033){\ellipse{80}{80}}
\put(2783,5033){\ellipse{80}{80}}
\put(3683,5033){\ellipse{80}{80}}
\put(4583,5033){\ellipse{80}{80}}
\put(5483,5033){\ellipse{80}{80}}
\put(6383,5033){\ellipse{80}{80}}
\put(7283,5033){\ellipse{80}{80}}
\put(1433,4583){\ellipse{80}{80}}
\put(983,4583){\ellipse{80}{80}}
\put(1883,4583){\ellipse{80}{80}}
\put(2333,4583){\ellipse{80}{80}}
\put(2783,4583){\ellipse{80}{80}}
\put(3233,4583){\ellipse{80}{80}}
\put(3683,4583){\ellipse{80}{80}}
\put(4583,4583){\ellipse{80}{80}}
\put(5033,4583){\ellipse{80}{80}}
\put(5483,4583){\ellipse{80}{80}}
\put(5933,4583){\ellipse{80}{80}}
\put(6383,4583){\ellipse{80}{80}}
\put(6833,4583){\ellipse{80}{80}}
\put(7283,4583){\ellipse{80}{80}}
\put(1883,4133){\ellipse{80}{80}}
\put(2783,4133){\ellipse{80}{80}}
\put(5483,4133){\ellipse{80}{80}}
\put(6383,4133){\ellipse{80}{80}}
\put(1433,3683){\ellipse{80}{80}}
\put(983,3683){\ellipse{80}{80}}
\put(1883,3683){\ellipse{80}{80}}
\put(2333,3683){\ellipse{80}{80}}
\put(2783,3683){\ellipse{80}{80}}
\put(3233,3683){\ellipse{80}{80}}
\put(3683,3683){\ellipse{80}{80}}
\put(4583,3683){\ellipse{80}{80}}
\put(5033,3683){\ellipse{80}{80}}
\put(5483,3683){\ellipse{80}{80}}
\put(5933,3683){\ellipse{80}{80}}
\put(6383,3683){\ellipse{80}{80}}
\put(6833,3683){\ellipse{80}{80}}
\put(7283,3683){\ellipse{80}{80}}
\put(7283,3233){\ellipse{80}{80}}
\put(6383,3233){\ellipse{80}{80}}
\put(5483,3233){\ellipse{80}{80}}
\put(4583,3233){\ellipse{80}{80}}
\put(3683,3233){\ellipse{80}{80}}
\put(2783,3233){\ellipse{80}{80}}
\put(1883,3233){\ellipse{80}{80}}
\put(983,3233){\ellipse{80}{80}}
\put(983,2783){\ellipse{80}{80}}
\put(1433,2783){\ellipse{80}{80}}
\put(1883,2783){\ellipse{80}{80}}
\put(2333,2783){\ellipse{80}{80}}
\put(2783,2783){\ellipse{80}{80}}
\put(3233,2783){\ellipse{80}{80}}
\put(3683,2783){\ellipse{80}{80}}
\put(4133,2783){\ellipse{80}{80}}
\put(4583,2783){\ellipse{80}{80}}
\put(5033,2783){\ellipse{80}{80}}
\put(5483,2783){\ellipse{80}{80}}
\put(5933,2783){\ellipse{80}{80}}
\put(6383,2783){\ellipse{80}{80}}
\put(6833,2783){\ellipse{80}{80}}
\put(7283,2783){\ellipse{80}{80}}
\put(7283,2333){\ellipse{80}{80}}
\put(6383,2333){\ellipse{80}{80}}
\put(5483,2333){\ellipse{80}{80}}
\put(4583,2333){\ellipse{80}{80}}
\put(3683,2333){\ellipse{80}{80}}
\put(2783,2333){\ellipse{80}{80}}
\put(1883,2333){\ellipse{80}{80}}
\put(983,2333){\ellipse{80}{80}}
\put(983,1883){\ellipse{80}{80}}
\put(983,1433){\ellipse{80}{80}}
\put(1433,1883){\ellipse{80}{80}}
\put(1883,1883){\ellipse{80}{80}}
\put(1883,1433){\ellipse{80}{80}}
\put(2333,1883){\ellipse{80}{80}}
\put(2783,1883){\ellipse{80}{80}}
\put(2783,1433){\ellipse{80}{80}}
\put(3233,1883){\ellipse{80}{80}}
\put(3683,1883){\ellipse{80}{80}}
\put(3683,1433){\ellipse{80}{80}}
\put(4133,1883){\ellipse{80}{80}}
\put(4583,1883){\ellipse{80}{80}}
\put(4583,1433){\ellipse{80}{80}}
\put(5033,1883){\ellipse{80}{80}}
\put(5483,1883){\ellipse{80}{80}}
\put(5483,1433){\ellipse{80}{80}}
\put(5933,1883){\ellipse{80}{80}}
\put(6383,1883){\ellipse{80}{80}}
\put(6383,1433){\ellipse{80}{80}}
\put(6833,1883){\ellipse{80}{80}}
\put(7283,1883){\ellipse{80}{80}}
\put(7283,1433){\ellipse{80}{80}}
\put(1433,983){\ellipse{80}{80}}
\put(983,983){\ellipse{80}{80}}
\put(1883,983){\ellipse{80}{80}}
\put(2333,983){\ellipse{80}{80}}
\put(2783,983){\ellipse{80}{80}}
\put(3233,983){\ellipse{80}{80}}
\put(3683,983){\ellipse{80}{80}}
\put(4583,983){\ellipse{80}{80}}
\put(5033,983){\ellipse{80}{80}}
\put(5483,983){\ellipse{80}{80}}
\put(5933,983){\ellipse{80}{80}}
\put(6383,983){\ellipse{80}{80}}
\put(6833,983){\ellipse{80}{80}}
\put(7283,983){\ellipse{80}{80}}
\end{picture}
}
\caption{\label{fig1}%
The set $E$ in Examples~\ref{ex-emental} and~\ref{ex-emental-p=n}.}
\end{center}
\end{figure}
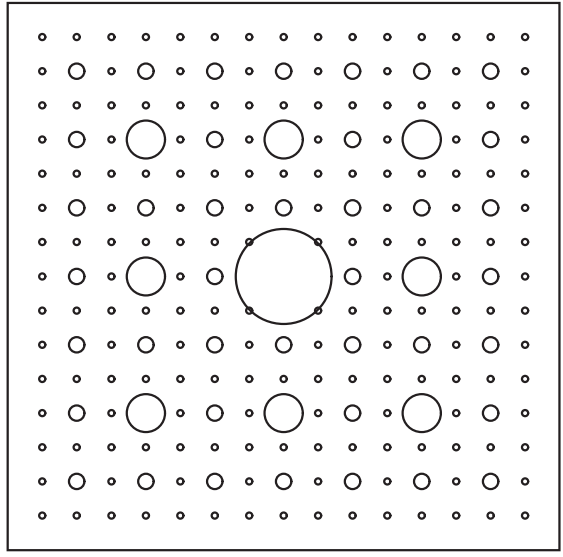

\begin{example} \label{ex-emental}
Let $Q_k=((0,1) \cap 2^{-k} \N)^n$ be a 
bounded lattice in unweighted $\R^n$, $n \ge 2$,
$k=1,2,\ldots$\,.
Let also $a_k=2^{-k}$ and $r_k=\de a_k^\al$, $k=1,2,\ldots$,  for some $0<\de<\tfrac{1}{2}$ 
and $\al>n/(n-p)$, where $1<p<n$.
Note that for a fixed $k$, the balls $\{B(q,r_k)\}_{q \in Q_k}$ are disjoint.
Let finally, 
\[
E = [0,1]^n\setm \bigcup_{k=1}^\infty \bigcup_{q\in Q_k} B(q,r_k),
\]
see Figure~\ref{fig1}.
Then $E\subset\R^n$ is a closed set with empty interior and
\[
m([0,1]^n\setm E) \le C \sum_{k=1}^\infty ( 2^{k}-1)^n r_k^n
\le C \de^n \sum_{k=1}^\infty 2^{kn(1-\al)}
\le C\de^n,
\]
where $m$ denotes the Lebesgue measure in $\R^n$.
Thus, for small $\de>0$, $E$ has almost full measure in $[0,1]^n$.
We shall show that the set $E$ has nonempty fine interior,
and that $m(E \setm \fineint E)=0$.

For a fixed $0<\theta<1-1/\al$ 
and all $0<\eps<\de$, we define
\[
E_\eps = [0,1]^n\setm \bigcup_{k=1}^\infty \bigcup_{q\in Q_k} 
          B(q,r_k+\eps a_k^{1+\theta}).
\]
Note that by the mean-value theorem, 
\begin{align*}
m(E\setm E_\eps) 
&\le \sum_{k=1}^\infty 2^{kn}
        m(B(0,r_k+\eps a_k^{1+\theta}) \setm B(0,r_k)) \\
&\le C \sum_{k=1}^\infty  2^{kn} r_k^n
   \biggl( \biggl( 1+\frac{\eps a_k^{1+\theta}}{r_k} \biggr)^n -1\biggr) \\
&\le C \sum_{k=1}^\infty 2^{kn} r_k^n \frac{\eps a_k^{1+\theta}}{r_k}
        n \biggl( 1+\frac{\eps a_k^{1+\theta}}{r_k} \biggr)^{n-1}.
\end{align*}
As $\eps<\tfrac12$ and $a_k^{1+\theta}/r_k>1$, the last estimate can be simplified as 
\begin{align*}
m(E\setm E_\eps) 
&\le C \eps \sum_{k=1}^\infty 2^{kn} r_k^{n-1} a_k^{1+\theta}
       \biggl(\frac{2a_k^{1+\theta}}{r_k} \biggr)^{n-1}\\
&=  C \eps \sum_{k=1}^\infty 2^{kn} a_k^{n(1+\theta)} 
=  C \eps \sum_{k=1}^\infty 2^{-kn\theta} 
\to 0, \quad \text{as }\eps\to0.
\end{align*}
It follows that $m(E\setm\bigcup_{\eps>0}E_\eps)=0$.

We claim that $\bigcup_{\eps>0}E_\eps\subset\fineint E$.
For this, it suffices to show that for every $x\in E_\eps$, 
the set $X\setm E$
is thin at $x$, in view of Proposition~\ref{prop-fineint-locally}.
Let therefore $0<\eps<\de$ and $x\in E_\eps$ be fixed.
We need to show that
\begin{equation}   \label{eq-show-thin-at-x-j-eps}
       \sum_{j=j_\eps}^\infty \biggl(\frac{\cp(B(x,2^{-j})\setm E,B(x,2^{1-j}))}
        {\cp(B(x,2^{-j}),B(x,2^{1-j}))}\biggr)^{1/(p-1)} <\infty
\end{equation}
for some $j_\eps$ (possibly depending on $x$, $\al$, $\theta$ 
and $\eps$).
We therefore let  $r=2^{-j}$ and estimate $\cp(B(x,r)\setm E,B(x,2r))$.

We shall first estimate how many balls $B(q,r_k)$, with $q\in Q_k$ and
$k<j$ (i.e.\ $a_k\ge2r$), can intersect $B(x,r)$.  
Since for every $q\in Q_k$, $k=1,2,\ldots$, we have
\[
\dist(x,B(q,r_k))\ge \eps a_k^{1+\theta},
\]
the intersection will be nonempty only if
$\eps 2^{-k(1+\theta)} < 2^{-j}$. 
This is equivalent to 
\begin{equation}  \label{eq-how-many-k}
k > \frac{1}{1+\theta} (j+\log_2 \eps)
\ge \frac{(1-\theta^2)j}{1+\theta} = {(1-\theta)j},
\end{equation}
provided that 
\begin{equation} \label{eq-jeps}
j\ge \frac{1}{\theta^2} (-\log_2 \eps).
\end{equation}
In particular, for each $\eps$ and $\theta$ there exists $j_\eps$ such
that \eqref{eq-jeps} holds for all $j\ge j_\eps$.

Moreover, for each $k<j$ as in \eqref{eq-how-many-k}, 
there are at most $2^n$ balls $B(q,r_k)$, $q\in Q_k$, intersecting $B(x,r)$,
since $a_k \ge2r$.
By  \eqref{cp-Rn-est1} their total capacity is at most
\(
C 2^n r_k^{n-p}.
\)
Summing up over all $k\in\N$, such that $(1-\theta)j < k <j$, yields
the estimate for the capacity 
\begin{align}
\sum_{(1-\theta)j < k <j} C r_k^{n-p}
&= C\de^{n-p} \sum_{(1-\theta)j < k <j} 2^{-k\al(n-p)} 
\le C\de^{n-p} 2^{-j\al(1-\theta)(n-p)}. 
\label{eq-est-cap-small-k}
\end{align}
Let now $k\ge j$, i.e.\ $a_k\le r$.
For each such $k$, there are at most $(4r/a_k)^n$ balls $B(q,r_k)$, $q\in Q_k$,
intersecting $B(x,r)$.
Their total capacity is at most
\[
C \biggl(\frac{4r}{a_k}\biggr)^n r_k^{n-p}
\le C 2^{n(k-j)} \de^{n-p} 2^{-k\al(n-p)}.
\]
Summing up over all $k=j,j+1,\ldots$ and combining this with 
\eqref{eq-est-cap-small-k} yields for $r=2^{-j}$, $j\ge j_\eps$,
\[
\cp(B(x,r)\setm E,B(x,2r)) 
\le C\de^{n-p} \biggl( 2^{-j\al(1-\theta)(n-p)} 
+ 2^{-jn} \sum_{k=j}^\infty 2^{-k(\al(n-p)-n)} \biggr).
\]
As $\al>n/(n-p)$, the last series converges with the sum
$C2^{jn-j\al(n-p)}$ and we conclude that
\begin{align*}
\cp(B(x,r)\setm E,B(x,2r)) 
&\le C\de^{n-p} 2^{-j\al(1-\theta)(n-p)}.
\end{align*}
Inserting this and \eqref{cp-Rn-est2} into \eqref{eq-show-thin-at-x-j-eps} shows that
for each $x\in E_\eps$ the sum in \eqref{eq-show-thin-at-x-j-eps}
is majorized by
\begin{align*}   
\sum_{j=j_\eps}^\infty \biggl(\frac{C\de^{n-p}2^{-j\al(1-\theta)(n-p)}}
        {2^{-j(n-p)}}\biggr)^{1/(p-1)} 
&= C \de^{(n-p)/(p-1)} \sum_{j=j_\eps}^\infty 2^{-j(\al(1-\theta)-1)(n-p)/(p-1)}\\
&<\infty,
\end{align*}
since $\al(1-\theta)>1$.

Thus, $X\setm E$ is thin at each $x\in E_\eps$ and 
Proposition~\ref{prop-fineint-locally} shows that
\[
\bigcup_{\eps>0}E_\eps\subset\fineint E.
\]
Hence $m(E\setm\fineint E)=0$ and Theorem~\ref{thm-E-E0-probl-same}
implies that the minimal \p-weak upper
gradients with respect to $E$ and $\R^n$
coincide, i.e.\  for every $u\in\Dp(\R^n)$,
\[
g_{u,\fineint E} = g_{u,E} = g_{u,\R^n} = |\grad u| \quad \text{a.e.\ in } E.
\] 
Moreover, by Theorem~\ref{thm-nontrivial2}, $\Np_0(E)$ is nontrivial
and solutions of obstacle and Dirichlet problems on $E$ are in general
not equal to their boundary data.
By Theorem~\ref{thm-E-E0-probl-same} again, the solutions of
the $\K_{\psi_1,\psi_2,f}(E)$-
and $\K_{\psi_1,\psi_2,f}(\fineint E)$-obstacle problems coincide
and have the same energies. 
\end{example}

The following example is a modification of Example~\ref{ex-emental}
for $p=n$.
In particular, it covers the classical situation $p=n=2$.

\begin{example} \label{ex-emental-p=n}
If $p=n$, let $E$ and $E_\eps$ be as in Example~\ref{ex-emental}
but with $r_k=\de 2^{-2^{k\al}}$ for some $\al>n/(n-1)$.
As in Example~\ref{ex-emental}, $E\subset\R^n$ is a nowhere dense
closed set and
\[
m([0,1]^n\setm E) \le C \de^n \sum_{k=1}^\infty 2^{kn-n2^{k\al}}
\le C\de^n.
\]
That $m(E\setm E_\eps)\le C\eps\to0$, as $\eps\to0$, is shown exactly
as in Example~\ref{ex-emental}.
(This time it is enough to require that $0<\theta<1$.)

To show that $\bigcup_{\eps>0}E_\eps\subset\fineint E$,
let $x\in E_\eps$ be fixed and $r=2^{-j}$, $j=1,2,\ldots$\,.
As in Example~\ref{ex-emental}, the ball $B(q,r_k)$ with $k<j$ intersects 
$B(x,r)$ only if 
\begin{equation*}  
k > \frac{1}{1+\theta} (j+\log_2 \eps)
\ge {(1-\theta)j}
\quad \text{provided that} \
j\ge j_\eps \ge \frac{1}{\theta^2} (-\log_2 \eps),
\end{equation*}
and for each such $k$ there are at most $2^n$ such balls.
By \eqref{cp-Rn-est1} each of these balls has capacity at most
\begin{align*}
C \biggl(\log\frac{2r}{r_k} \biggr)^{1-n} &= C (1-j-\log_2\de+2^{k\al})^{1-n}
\le C 2^{-k\al(n-1)}.
\end{align*}
The total capacity of all such balls with $(1-\theta)j < k <j$
and $B(q,r_k)\cap B(x,r)\ne\emptyset$ is therefore at most
\begin{align}
\sum_{(1-\theta)j < k <j} C 2^{-k\al(n-1)}
&\le C 2^{-j\al(1-\theta)(n-1)}. 
\label{eq-est-cap-small-k-p=n}
\end{align}
Now, for each $k\ge j$, there are at most $(4r/a_k)^n$ balls 
$B(q,r_k)$, $q\in Q_k$,
intersecting $B(x,r)$ and their total capacity is at most
\[
C\biggl(\frac{4r}{a_k}\biggr)^n \biggl(\log\frac{2r}{r_k} \biggr)^{1-n}
\le C 2^{n(k-j)} 2^{-k\al(n-1)}.
\]
Summing up over all $k=j,j+1,\ldots$ and combining this with 
\eqref{eq-est-cap-small-k-p=n} yields for $r=2^{-j}$, $j\ge j_\eps$,
\[
\cp(B(x,r)\setm E,B(x,2r)) 
       \le C 2^{-j\al(1-\theta)(n-1)} 
          + 2^{-jn} \sum_{k=j}^\infty 2^{-k(\al(n-1)-n)}.
\]
As $\al>n/(n-1)$, the last series converges with the sum 
$C2^{jn-j\al(n-1)}$ and we conclude that
\begin{align*}
\cp(B(x,r)\setm E,B(x,2r)) 
&\le C 2^{-j\al(1-\theta)(n-1)}.
\end{align*}
Inserting this and \eqref{cp-Rn-est2} into \eqref{eq-show-thin-at-x-j-eps} shows that
for each $x\in E_\eps$ the sum in \eqref{eq-show-thin-at-x-j-eps}
is majorized by
\begin{align*}   
\sum_{j=j_\eps}^\infty (C2^{-j\al(1-\theta)(n-1)})^{1/(n-1)} 
&= C \sum_{j=j_\eps}^\infty 2^{-j\al(1-\theta)}
<\infty.
\end{align*}
Thus, $X\setm E$ is thin at each $x\in E_\eps$ and 
Proposition~\ref{prop-fineint-locally} shows that
$\bigcup_{\eps>0}E_\eps\subset\fineint E$.
Hence $m(E\setm\fineint E)=0$ and Theorem~\ref{thm-E-E0-probl-same}
implies that the minimal \p-weak upper
gradients with respect to $E$ and $\R^n$
coincide a.e.\ on $E$.
By Theorem~\ref{thm-nontrivial2}, $\Np_0(E)$ is nontrivial
and solutions of obstacle and Dirichlet problems on $E$ are in general
not equal to their boundary data.
By Theorem~\ref{thm-E-E0-probl-same} again, also the solutions of
the $\K_{\psi_1,\psi_2,f}(E)$-
and $\K_{\psi_1,\psi_2,f}(\fineint E)$-obstacle problems coincide
and have the same energies. 
\end{example}

\section{Further examples}
\label{sect-an ex}

Let $X=\R^2$ be equipped with $d\mu=dx+\alp\,dx_1$, where
$dx$ is the $2$-dimensional Lebesgue measure on $\R^2$, 
$dx_1$ is the $1$-dimensional Lebesgue measure on $\R$ (extended
as the zero measure on $\R^2 \setm \R$),
and $\alp$ is a positive real constant.

\begin{prop} \label{prop-R2-R}
Let $u \in \Np(X) $. Then the function
\begin{equation} \label{eq-gu}
     \gt_u=\begin{cases}
       |\nabla u| & \text{in } \R^2 \setm \R, \\
       |\partial_1 u| & \text{in } \R,
       \end{cases}
\end{equation}
is a minimal \p-weak upper gradient of $u$ with respect to $\mu$. 
Here\/ $\nabla u$ is the distributional gradient on\/ $\R^2$
and $\partial_1 u$ is the distributional derivative on\/ $\R$.
\end{prop}

Observe that $u \in \Np(\R^2,dx) \subset W^{1,p}(\R^2)$,
and thus $u$ has a distributional gradient.
Similarly, $u|_\R \in \Np(\R,dx_1)$ is absolutely continuous on $\R$ 
and has a distributional derivative there. 
To prove Proposition~\ref{prop-R2-R} we need the following two auxiliary results
which hold for arbitrary metric spaces $X$.

\begin{lem} \label{lem-loop-erased}
Any\/ \textup{(}rectifiable\/\textup{)} curve $\ga:[0,l_\ga] \to X$
has an associated loop-erased simple curve $\gat:[0,l_{\gat}] \to X$.
\end{lem}

A loop along the curve $\ga$ is a part $\ga|_{[t_0,t_1]}$ such that
$0 \le t_0 < t_1 \le l_\ga$ and $\ga(t_0)=\ga(t_1)$. 
Such a part can be removed 
by redefining $\ga(t)=\ga(t_0)$ for $t_0<t<t_1$.
By doing this iteratively in an appropriate way and then reparameterizing
(see below)
we can obtain a loop-free (i.e.\ simple) curve $\gat\subset\ga$
such that in particular
$\int_{\gat} g \, ds \le \int_{\ga} g \,ds$ for all nonnegative Borel functions
$g$.
Note that a curve may have several different associated loop-erased simple
curves.

\begin{proof}
As the length of $\ga$ is finite there is a longest loop
(it may not be unique), unless $\ga$ is already loop-free. 
Remove it, as described above, and call the resulting curve $\ga_1$.
Repeat the procedure to produce $\ga_2$, $\ga_3$ etc.
This can end after a finite number of steps with $\ga_n$,
which is then (after reparameterization with respect to arc length) 
the desired loop-erased simple curve $\gat$.

Otherwise, we get curves $\ga_j:[0,l_\ga] \to X$, $j=1,2,\ldots$,
which by Ascoli's theorem converge to a curve $\gat$ with the same endpoints.
(Note that here we need a version of Ascoli's theorem valid for metric space
valued equicontinuous functions, see e.g.\ p.\ 169 in 
Royden~\cite{royden}.)
The resulting curve is a $1$-Lipschitz map
which (after reparameterization with respect to arc length) 
is the desired loop-erased simple curve $\gat$.
\end{proof}

\begin{lem}  \label{lem-mu1-mu2}
Let $X$ be equipped with two measures $\mu_1$ and $\mu_2$ such that 
$\mu_1\le\mu_2$.
Then $\Np(X,\mu_2)\subset\Np(X,\mu_1)$ and for every $u\in\Np(X,\mu_2)$,
the minimal \p-weak upper gradients 
with respect to $\mu_1$ and $\mu_2$ satisfy $g_{u,\mu_1}\le g_{u,\mu_2}$
$\mu_1$-a.e.
\end{lem}

\begin{proof}
The inclusion $\Np(X,\mu_2)\subset\Np(X,\mu_1)$ follows directly from the 
fact that upper gradients do not depend on the underlying measure and 
that $\Np(X,\mu_j)$, $j=1,2$, can be defined only using upper gradients.

To compare the minimal \p-weak upper gradients, let 
$u\in\Np(X,\mu_2)\subset\Np(X,\mu_1)$.
It is easily verified that $\Mod_{p,\mu_1}(\Ga)=0$
whenever $\Mod_{p,\mu_2}(\Ga)=0$.
Hence, the minimal \p-weak upper gradient $g_{u,\mu_2}$ of $u$
with respect to $\mu_2$ is a \p-weak upper gradient of $u$
with respect to $\mu_1$ and we  conclude that 
$g_{u,\mu_1}\le g_{u,\mu_2}$ $\mu_1$-a.e.\ in $X$.
\end{proof}

\begin{cor}  \label{cor-PI-sum-mu1-mu2}
Let $\mu_1$ and $\mu_2$ be two measures on $X$ which support 
$(p,p)$-Poincar\'e inequalities for $\Np_0$.
Then so does the measure $\mu=\mu_1+\mu_2$.
\end{cor}

\begin{proof}
Lemma~\ref{lem-mu1-mu2} shows that $g_{u,\mu_j}\le g_{u,\mu}$ $\mu_j$-a.e. in $X$,
$j=1,2$.
Hence
\[
\int_X g_{u,\mu_j}^p \,d\mu_j\le \int_X g_{u,\mu}^p \,d\mu_j 
\le \int_X g_{u,\mu}^p \,d\mu, \quad j=1,2.
\]
The $(p,p)$-Poincar\'e inequalities for $\Np_0$ with respect to
$\mu_1$ and $\mu_2$, together with
\[
\|u\|^p_{L^p(X,\mu)} = \|u\|^p_{L^p(X,\mu_1)} + \|u\|^p_{L^p(X,\mu_2)},
\]
then finish the proof.
\end{proof}

\begin{proof} [Proof of Proposition~\ref{prop-R2-R}]
Since $d\mu\ge dx$ on $\R^2$ and $d\mu\ge dx_1$ on $\R$,
Lemma~\ref{lem-mu1-mu2} implies that the minimal \p-weak upper gradient 
with respect to $\mu$ satisfies $g_u\ge\gt_u$ $\mu$-a.e.
It is therefore enough to show that $\gt_u$ itself is also 
a \p-weak upper gradient of $u$ with respect to $\mu$.
This will be done by showing that it belongs to the $L^p(X)$-closure
of the set of upper gradients of $u$.
Proposition~2.10 in Bj\"orn--Bj\"orn~\cite{BBbook} then shows that
$\gt_u$ is a \p-weak upper gradient of $u$ with respect to $\mu$.

Let $\eps>0$.
As $|\grad u|$ is a minimial \p-weak upper gradient of $u$ with respect to
$dx$, we can find an upper gradient $\gt \in L^p(\R^2,dx)$ of $u$ 
such that $\|\gt-|\grad u|\|_{L^p(\R^2,dx)} < \eps$.
Let 
\[
     g=\begin{cases}
       \gt & \text{in } \R^2 \setm \R, \\
       |\partial_1 u| & \text{in } \R.
       \end{cases}
\]
Then $\|g-g_u\|_{L^p(X)} = \|\gt-g_u\|_{L^p(\R^2,dx)} < \eps$.
We shall show that $g$ is an upper gradient of $u$ in $\R^2$.
We may require $\partial_1 u$ above to be a Borel function on $\R$,
by Proposition~1.2 in Bj\"orn--Bj\"orn~\cite{BBbook}.
Since $\gt$ is a Borel function, so is $g$.

Let $\ga: [0, l_\ga] \to X$ be a curve.
If $\ga \subset \R$, then
\[
        |u(\gamma(0)) - u(\gamma(l_{\gamma}))| 
     \le \int_{\gamma} |\partial_1 u|\,ds
      =  \int_{\gamma} g\,ds.
\]
Similarly, if 
$\{t: \ga(t) \in \R\}$ 
is a finite set,
then 
\[
        |u(\gamma(0)) - u(\gamma(l_{\gamma}))| 
     \le \int_{\gamma} \gt \,ds
      =  \int_{\gamma} g\,ds.
\]
After possibly splitting any other curve into at most four
 parts we may
assume that $\ga(0),\ga(l_\ga) \in \R$, $\ga(0) \ne \ga(l_\ga)$, 
 but $\ga \not\subset \R$.

Let $G=\{t:\ga(t) \in \R^2 \setm \R\}$ which is a nonempty open
subset of $(0,l_\ga)$.
It can thus be written as a pairwise disjoint  union
 $\bigcup_{i=1}^{\infty} I_i$
of open intervals. 
(Here we allow some of the intervals $I_i$ to be empty.)
For a fixed $n$ let $G_n=\bigcup_{i=1}^n I_i$.
Let $\pi(x,y)=(x,0)$ be the orthogonal projection of $\R^2$ onto $\R$,
and 
\[
     \ga_n(t)= \begin{cases}
           \ga(t), & t \in G_n, \\
           \pi \circ \ga(t), & t \in [0,l_\ga] \setm G_n.
       \end{cases}
\]
Then $\ga_n$ is a rectifiable curve. The given parameterization
may not be arc length, but it is a $1$-Lipschitz map.
Let $\gat_n$ be an associated loop-erased simple curve of $\ga_n$,
given by Lemma~\ref{lem-loop-erased}.
Then $\gat_n$ can be split into at most $2n+1$ subcurves such that
each subcurve either is completely in $\R$, or it hits $\R$ only at its
endpoints.
Denote the union of the former by $\gat_n \cap \R$, and the union
of the latter by $\gat_n \setm \R$.
Note that $\gat_n\setm\R\subset\ga|_{G}$.
Using that these
subcurves have already been treated above, we conclude that
\begin{equation}  \label{eq-cu}
        |u(\gamma(0)) - u(\gamma(l_{\gamma}))| 
            \le \int_{\gat_n \setm \R} g \, ds 
                   + \int_{\gat_n \cap \R} g \, ds 
            \le  \int_{\ga|_{G}} g \, ds 
                   + \int_{\gat_n \cap \R} g \, ds. 
\end{equation}
Since $\gat_n$ is a simple curve we obtain that
\begin{align*}
     \liminf_{n \to \infty} \int_{\gat_n \cap \R} g \, ds
     = \liminf_{n \to \infty} \int_{\R} g \chi_{\gat_n \cap \R} \, dx
     \le  \int_{\R} g \chi_{\ga \cap \R} \, dx
     \le  \int_{\ga|_{[0,l_\ga] \setm G}} g \, ds.
\end{align*}
Here we have used dominated convergence in the middle, which
is justified by the fact that the integrands in the second integral
are dominated by $g \chi_{[-a,a]}$ for some $a >0$, and 
$g \in L^p(\R) \subset L^1\loc(\R)$.
(It is for justifying this dominated convergence we need to use
the loop-erased simple curves.)
We have also used the fact that arc length for projections 
is majorized by arc length of the original curve.

Inserting this into \eqref{eq-cu} shows that
\[
        |u(\gamma(0)) - u(\gamma(l_{\gamma}))| 
            \le  \int_{\ga|_{G}} g \, ds 
                   + \int_{\ga|_{[0,l_\ga] \setm G}} g \, ds
                   = \int_\ga g \, ds.
                   \qedhere
\]
\end{proof}

\begin{remark}
The same proof as in Proposition~\ref{prop-R2-R} shows that if $\nu$ 
is any positive Borel measure on $\R$
satisfying $0<\nu(I)<\infty$ for every finite interval $I$, then 
the function
\begin{equation*} 
     g_{u,\mu}=\begin{cases}
       |\nabla u| & \text{in } \R^2 \setm \R, \\
       g_{u,\nu} & \text{in } \R,
       \end{cases}
\end{equation*}
is a minimal \p-weak upper gradient of $u$ with respect to $d\mu=dx+d\nu$. 
Here\/ $\nabla u$ is the distributional gradient on\/ $\R^2$
and $g_{u,\nu}$ is the minimal \p-weak upper gradient of $u$ on\/ $\R$
with respect to $\nu$.
(In this case, $g$ in the proof of Proposition~\ref{prop-R2-R}
consists of $\gt$ and an upper gradient approximating $g_{u,\nu}$
in $L^p(\R,\nu)$.)
See Proposition~\ref{prop-nu-R} below and the comments after it 
for some results 
on one-dimensional minimal \p-weak upper gradients for different measures.

Note also that by Corollary~\ref{cor-PI-sum-mu1-mu2}, the 
$(p,p)$-Poincar\'e inequality for $\Np_0$ holds for $\mu$, provided it 
holds for $\nu$ on $\R$.
Combined with Proposition~\ref{prop-nu-R}, this provides us with many
examples of non-standard measures on $\R^2$ to which 
a large part of our theory applies.
\end{remark}

With a little bit more work we can show
that the measure $d\mu=dx+\alp\,dx_1$ on $\R^2$ supports a
$(q,p)$-Poincar\'e inequality as in Definition~\ref{def-PI}, not only
a $(p,p)$-Poincar\'e inequality for $\Np_0$ as in the above remark.
Here $q=2p/(2-p)$ (for $p<2$) or $q<\infty$ (for $p\ge2$) is the usual Sobolev
exponent on $\R^2$. 
We can clearly assume that $q\ge p$.
Note however that $\mu$ is not doubling and we cannot therefore 
conclude the $(q,p)$-Poincar\'e inequality directly from the 
$(1,1)$-Poincar\'e inequality which would have been somewhat simpler to derive.

Let $u\in\Np(\R^2,\mu)$ and  $Q=I\times I'\subset\R^2$, where $I, I'\subset\R$ 
are finite intervals of length $R$. 
We can assume that $0\in I'$, as otherwise $\mu|_Q$ is just the Lebesgue 
measure on $Q$.
Let also $\uQm$, $\uQx$ and $\uI$ be the integral averages of $u$ over $Q$ with
respect to $d\mu$, $dx$ and $dx_1$, respectively.
Split the left-hand side in the $(q,p)$-Poincar\'e inequality as
\begin{align}   \label{eq-split}
\biggl( \int_Q |u-\uQm|^q\,d\mu \biggl)^{1/q}
&\le \|u-\uQx\|_{L^q(Q,dx)}  + |Q|^{1/q} |\uQx-\uQm| \\
       &\quad + \alp^{1/q}\|u-\uI\|_{L^q(I,dx_1)}  + (\alp|I|)^{1/q} |\uI-\uQm|, 
\nonumber
\end{align}
where $|Q|$ and $|I|$ are the 2- and 1-dimensional Lebesgue
measures of $Q$ and $I$, respectively.
The first and the third term are estimated using the usual 
Sobolev--Poincar\'e  inequalities on $\R^2$ and $\R$, respectively.
For the second term we have (using the fact that $u$ is 
absolutely continuous on a.e.\ line parallel to the $x_2$-axis) that
\begin{align*}
|\uQx-\uQm| &= \biggl| \frac{\mu(Q)-|Q|}{\mu(Q)} \vint_Q u\,dx
   - \frac{\al}{\mu(Q)} \int_I u\,dx_1 \biggr| \\
&\le \frac{\al}{\mu(Q)} \vint_{I'} \int_I |u(x_1,x_2)-u(x_1,0)|\,dx_1\,dx_2 \\
&\le \frac{\al}{\mu(Q)}  \int_I \int_{I'} |\partial_{x_2}u(x_1,t)|\,dt\,dx_1 \\
&\le \frac{\al |Q|^{1-1/p}}{\mu(Q)} \biggl( \int_Q|\grad u|^p\,dx \biggr)^{1/p}.
\end{align*}
Similarly, 
\begin{align*}
|\uI-\uQm| &= \biggl| \frac{\mu(Q)-\al|I|}{\mu(Q)} \vint_I u\,dx_1
   - \frac{1}{\mu(Q)} \int_Q u\,dx \biggr| \\
&\le \frac{|Q|}{\mu(Q)} \vint_{I'} 
                     \vint_I |u(x_1,0)-u(x_1,x_2)|\,dx_1\,dx_2 \\
&\le \frac{|I|\, |Q|^{1-1/p}}{\mu(Q)}
          \biggl( \int_Q|\grad u|^p\,dx \biggr)^{1/p}.
\end{align*}
Since $|\grad u|\le g_u$ a.e.\ on $\R^2$, $|u'|\le g_u$ a.e.\ on $\R$,
$dx\le d\mu$ and $\al\,dx_1\le d\mu$, inserting this into~\eqref{eq-split}
yields
\[
\biggl( \int_Q |u-\uQm|^q\,d\mu \biggl)^{1/q}
   \le C(R) \biggl( \int_Q g_u^p\,d\mu \biggl)^{1/p},
\]
where
\[
C(R)=C R (|Q|^{1/q-1/p} + |I|^{1/q-1/p}) 
             + \frac{\alp|Q|^{1+1/q-1/p}}{\mu(Q)} 
             + \frac{(\alp|I|)^{1/q} |I|\, |Q|^{1-1/p}}{\mu(Q)}.
\]
As $|Q| \le \mu(Q)$, $|I|=R \le \mu(Q)$ and 
$\alp|Q|\le R\mu(Q)$ this
proves the $(q,p)$-Poincar\'e inequality on squares $Q\subset\R^2$.
For balls, using the circumscribed squares gives a weak Poincar\'e
inequality with dilation $\sqrt2$.

Similar arguments can be used in other situations, in particular
on Euclidean spaces.
Here we give a rather general one-dimensional result.

\begin{prop} \label{prop-nu-R}
Let $\mu$ be a positive locally finite Borel measure on\/ $\R$ with the 
Lebesgue--Radon--Nikodym decomposition $d\mu=w\,dx +d\sig$,
where $0\le w\in L^1\loc(\R)$ is locally essentially bounded away from
zero and $\sig\perp dx$.
Then for all $u\in\Np(\R,\mu)$, all $q\ge1$ and all
finite intervals $I\subset\R$,
\[
\biggl( \vint_I |u-u_{I,\mu}|^q\,d\mu \biggr)^{1/q}
   \le 2 |I|^{1-1/p} \biggl(\frac{\mu(I)}{\essinf_I w}\biggr)^{1/p} 
   \biggl( \vint_I  g_{u,\mu}^p\,d\mu \biggr)^{1/p},
\]
where\/ $|I|$ is the Lebesgue measure of $I$.
In particular, $(\R,\mu)$ supports a\/ $(p,p)$-Poincar\'e inequality
for $\Np_0$.

Moreover, for every $u\in\Np(\R,\mu)$, the minimal \p-weak upper gradient
of $u$ with respect to $\mu$ is the function
\begin{equation} \label{eq-gumu}
     \gt_u=\begin{cases}
       |u'| & \text{in } A, \\
       0 & \text{in } \R\setm A,
       \end{cases}
\end{equation}
where $A$ is a maximal null set of the singular part $\sig$ 
of $\mu$ with respect to the Lebesgue measure, i.e.\ $\sig(A)=0$ and
$|\R\setm A|=0$,
and $u'$ is the distributional derivative.
\end{prop}

\begin{remark}
Since $\int_E g_{u,\mu}^p\,d\mu = \int_E g_{u,\mu}^p w\,dx$,
Proposition~\ref{prop-nu-R} shows that
there is no need to consider measures with a singular part when solving
the Dirichlet problem on $\R$, provided that the measure is locally bounded
from below by a positive multiple of the Lebesgue measure.
On the other hand, for obstacle problems it still makes sense to distinguish
between $\mu$ and its absolutely continuous part $w\,dx$, since the 
presence of the singular part $\sig$ may influence the capacity $\Cp$ 
and hence the obstacle
condition $\psi_1\le u\le\psi_2$ q.e.

If $\mu$ is not bounded from below by a positive multiple of the Lebesgue
measure, then Proposition~\ref{prop-nu-R} can fail, as shown by the following
examples.
\end{remark}

\begin{example}
Let $d\mu=|x|^\al\,dx$ with $\al>2p-1$ and $u(x)=|x|^{-\beta}$,
where $1\le\beta<(\al+1)/p-1$.
Then $u\in\Np\loc(\R,\mu)$ and $g_{u,\mu}=\beta |x|^{-\beta-1}$, but $u$ is not 
a distribution, so $g_{u,\mu}$ cannot be its distributional derivative.
\end{example}

\begin{example}
Let $\{q_j\}_{j=1}^\infty$ be a countable dense subset of $\R$ and 
$\{a_j\}_{j=1}^\infty$ be a sequence of positive numbers such that 
$\sum_{j=1}^\infty a_j<\infty$.
Let also $p\ge1$, $\al>p-1$, $0<\eps<1/\al$ and $1\le\beta<(\al+1)/p$.
Then the function 
\[
f(x) = 1+ \sum_{j=1}^\infty a_j|x-q_j|^{-\al\eps}
\] 
belongs to $L^1\loc(\R)$
and is thus finite a.e.
Since also $f\ge1$ on $\R$, it follows that 
$w:=f^{-1/\eps}\in L^1\loc(\R)$ is positive a.e.\ 
and $w(x)<|x-q_j|^\al/a_j^{1/\eps}$ for all $j=1,2,\ldots$\,.

Let $d\mu=w\,dx$ and $u(x)=\sum_{j=1}^\infty a_j^{1+1/p\eps} |x-q_j|^{-\beta}$.
Since
\[
\int_{-R}^R ( a_j^{1/p\eps} |x-q_j|^{-\beta} )^p\,d\mu
\le \int_{-R}^R ( a_j^{1/p\eps} |x|^{-\beta} )^p 
                   \frac{|x|^\al\,dx}{a_j^{1/\eps}} 
=  \int_{-R}^R |x|^{\alp-\beta p}\,dx
    <\infty,
\]
we see that $u\in L^p\loc(\R,\mu)$.
As $\int_a^b u(x)\,dx=\infty$ for
every nonempty interval $(a,b)\subset\R$,
Proposition~1.37\,(c) in~Bj\"orn--Bj\"orn~\cite{BBbook} implies that
the family of all rectifiable curves on $\R$ has zero $\Mod_{p,\mu}$-modulus.
It follows that the zero function is a \p-weak upper gradient with respect 
to $\mu$ of every function and hence $\Np(\R,\mu)=L^p(\R,\mu)$.
\end{example}

\begin{proof}[Proof of Proposition~\ref{prop-nu-R}.]
Lemma~\ref{lem-mu1-mu2} implies that $u\in\Nploc(\R,dx)$ and
\[
g_{u,\mu}\ge g_{u,dx}=|u'|= \gt_u  \quad dx\text{-a.e.\ in } \R,
\] 
and hence $g_{u,\mu}\ge\gt_u$ $\mu$-a.e.\ in $A$.
Since $\gt_u=0$ in $\R\setm A$, we see that  
$g_{u,\mu}\ge\gt_u$ $\mu$-a.e.\ in $\R$.
Conversely, as $u$ is absolutely continuous on $\R$, the fundamental
theorem of calculus and the fact that $\gt_u=|u'|$ $dx$-a.e.\ 
shows that for all $x\le y\in\R$,
\[
|u(x)-u(y)|\le \int_x^y|u'(t)|\,dt = \int_x^y\gt_u\,dt,
\]
i.e.\ $\gt_u$ is an upper gradient of $u$. 
Hence $g_{u,\mu}\le \gt_u$ $\mu$-a.e.\ 
in $\R$.

The fundamental theorem of calculus again, together with 
H\"older's inequality and Fubini's
theorem, now yields (with $I=(a,b)$)
\begin{align*}
\int_I |u(x)-u(a)|^q\,d\mu(x)
    &\le |I|^{q-q/p} \int_I \biggl( \int_I |u'(t)|^p \,dt \biggr)^{q/p} \,d\mu(x)\\
    &\le |I|^{q-q/p} \mu(I) \biggl( \int_I  g_{u,\mu}^p\,dt \biggr)^{q/p}.
\end{align*}
Since $dt\le w^{-1} \, d\mu$, we obtain
\[
\biggl( \vint_I |u(x)-u(a)|^q\,d\mu(x) \biggr)^{1/q}
   \le |I|^{1-1/p} \biggl(\frac{\mu(I)}{\essinf_I w}\biggr)^{1/p} 
   \biggl( \vint_I  g_{u,\mu}^p\,d\mu \biggr)^{1/p},
\]
and the required inequality then follows by a standard 
argument in which the constant $u(a)$ is replaced by the mean value
$u_{I,\mu}$, see e.g.\ Lemma~4.17 in
Bj\"orn--Bj\"orn~\cite{BBbook}.
\end{proof}

We have seen that our theory can be directly applied to
the measure $d\mu=dx+\alp\,dx_1$ on $\R^2$,
or even $d\mu=dx+w(x_1)\,dx_1$ for a  suitable weight $w$,
and we can thus study the minimizers of the corresponding energy.
It may be of interest to see what equation they satisfy.

Let $\Om \subset \R^2$ be a domain.
In $\Om \setm \R$, a minimizer $u$ with respect to $\mu$ is a minimizer 
with respect to  the ordinary $dx$ measure,
and is hence, after redefinition on a set of capacity zero, a
\p-harmonic function and thus locally $C^{1,\alp}$ in $\Om \setm \R$.
As $u|_{\Om \cap \R} \in \Np(\Om \cap \R, dx_1)$, $u|_{\Om \cap \R}$
must be absolutely continuous.
Since all the points in $\Om \cap \R$ are regular boundary points
of $\{(x_1,x_2) \in \Om : \pm x_2 > 0\}$
(for all $p>1$), it follows that $u$ is continuous
across $\R$ and thus (after the redefinition above) $u$ is continuous in $\Om$.

For simplicity let us assume that $p=2$. 
In this case $u$ is harmonic in $\Om \setm \R$ and thus analytic
therein.
It locally minimizes the energy
\[
      \int ((\bdy_1 u)^2 + (\bdy_2 u)^2)\,dx_1\,dx_2 
           + \int (\bdy_1 u)^2\, w\, dx_1.
\]
It must therefore satisfy the corresponding Euler--Lagrange 
equation, which in weak form becomes
\[
      \int_{\Om} \grad u \cdot \grad\phi
\, dx_1\,dx_2
      + \int_{\Om \cap \R} \bdy_1 u \, \bdy_1 \phi \, w\, dx_1 =0
      \quad \text{for all } \phi \in C_0^\infty(\Om).
\]
Consider $\phi(x_1,x_2)
=\phi_1(x_1)\phi_2(\tau x_2) \in C_0^\infty(\Om)$,
where $\tau \ge 1$
and  $\phi_2(0)=1$.
Inserting this into the Euler--Lagrange equation gives
\begin{align}
   &\int_{\R} \biggl(\int_{\R} \bdy_1 u(x_1,x_2) \bdy_1 \phi_1(x_1)\, dx_1\biggr) 
          \phi_2(\tau x_2)\, dx_2 \label{eq-Euler} \\
   & \quad      +  \int_{\R} \biggl(\int_{\R} \tau \bdy_2 u(x_1,x_2) \bdy_2 \phi_2(\tau x_2)\, dx_2 \biggr) 
              \phi_1(x_1)\, dx_1 
      - \int_{\Om \cap \R} \phi_1\bdy_1 (w \bdy_1 u) \, dx_1 =0. \nonumber
\end{align}
After the change of variables $y=\tau x_2$, the inner integral in the 
second term becomes
\begin{align*}
    \int_{\R} \bdy_2 u(x_1,y/\tau) \bdy_2 \phi_2(y)\, dy 
\end{align*}
which tends to
\begin{align*}
        \bdy_2^\limminus u(x_1,0) \int_{-\infty}^0 \bdy_2 \phi_2\,dy 
       + \bdy_2^\limplus u(x_1,0) \int_0^\infty \bdy_2 \phi_2\,dy  
    &   = \bdy_2^\limminus u(x_1,0) - \bdy_2^\limplus u(x_1,0),
\end{align*}
as $\tau \to \infty$,
where $\bdy_2^\limpm u(x_1,0) = \lim_{x_2 \to 0\limpm} \bdy_2 u(x_1,x_2)$. 
As the first term in \eqref{eq-Euler} tends to $0$, as $\tau \to \infty$, 
we obtain that 
\[
      \int_{\R} (\bdy_2^\limminus u(x_1,0) - \bdy_2^\limplus u(x_1,0))  \phi_1(x_1)\, dx_1
      - \int_{\Om \cap \R} \phi_1\bdy_1 (w \bdy_1 u) \, dx_1 =0.
\]
Thus $u$ needs to fulfill
\[
    \bdy_2^\limminus u(x_1,0) - \bdy_2^\limplus u(x_1,0)
       = \bdy_1 (w \bdy_1 u)(x_1,0)
    \quad \text{for } x_1 \in \Om \cap \R,
\]
(in a weak sense) and be harmonic in $\Om \setm \R$. 
For this derivation we have assumed that $u$ is smooth enough.

\appendix

\section{Consequences of Fuglede's and Mazur's lemmas}

In this appendix we prove two convergence results, which have been
used in the earlier sections.
They are generalizations to Dirichlet spaces of results from 
Bj\"orn--Bj\"orn--Parviainen~\cite{BBP} 
(which can also be found in Bj\"orn--Bj\"orn~\cite{BBbook}).
Note that these results hold on arbitrary metric spaces without any
additional assumptions.

\begin{prop} \label{prop-Fuglede-consequence}
Assume that $f_j \in \Dp(X)$ and that
$g_j \in L^p(X)$ is a \p-weak upper gradient of $f_j$, $j=1,2,\ldots$\,.
Assume further that $f_j-f \to 0$ and $g_j \to g$ in $L^p(X)$,
as $j \to \infty$,
and that $g$ is nonnegative.
Then there is a function $\ft=f$ a.e.\ such that $g$ is 
a \p-weak upper gradient of $\ft$, and
thus $\ft \in \Dp(X)$.
There
is also a subsequence\/ $\{f_{j_k}\}_{k=1}^\infty$ such that
$f_{j_k} \to \ft$ q.e., as $k \to \infty$.
\end{prop}

When we say that $f_j-f \to 0$ in $L^p(X)$ we implicitly
require that $f_j-f \in L^p(X)$, which in particular
requires that $f_j$ and $f$ are real-valued a.e.
Note that we do \emph{not} require $f_j \in L^p(X)$
and can therefore \emph{not} use Proposition~3.1 in~\cite{BBP} 
(nor Proposition~2.3 in~\cite{BBbook}).

\begin{proof}
By passing to a subsequence if necessary we may assume
that $f_j \to f$ a.e.,
and (by Fuglede's lemma, see Shan\-mu\-ga\-lin\-gam~\cite{Sh-rev}, 
Lemma~3.4 and Remark~3.5, or Lemma~2.1 in Bj\"orn--Bj\"orn~\cite{BBbook}),
that $\int_\ga g_j\,ds \to \int_\ga g\,ds \in \R$, as $j \to \infty$,
for all curves $\ga \notin \Ga$, where $\Modp(\Ga)=0$.
Let $\ft=\limsup_{j \to \infty} f_j$, and observe that $\ft$ is
defined at every point of $X$ and $\ft=f$ a.e.\ in $X$.
Let $A=\{x \in X : |\ft(x)|=\infty\}$.

By definition,  \p-almost every curve $\ga$
is such that \eqref{ug-cond} holds for all $f_j$ and $g_j$,
$j=1,2,\ldots$, on $\ga$ and all its subcurves, 
and neither $\ga$ nor any of its subcurves belong to
$\Ga$.
Consider such a curve $\ga\colon  [0,l_\ga] \to X$. 
We see that 
either
$\ga(0),\ga(l_\ga) \in A$ or
\[
         |\ft(\ga(l_\ga))- \ft(\ga(0))|  
                      \le \limsup_{j \to \infty} |f_j(\ga(l_\ga))- f_j(\ga(0))|  
                  \le \limsup_{j \to \infty} \int_{\ga} g_j \, ds 
                 =  \int_{\ga} g \, ds.
\]
As $\mu(A)=0$, Proposition~2.5 in 
Bj\"orn--Bj\"orn--Parviainen~\cite{BBP}
(or Corollary~1.51 in Bj\"orn--Bj\"orn~\cite{BBbook})
shows that $g$ is indeed a \p-weak upper gradient of $\ft$,
and thus $\ft \in \Dp(X)$.

Let now $\fh=\liminf_{j \to \infty} f_j$.
Arguing exactly as above we see that $g$ is also a \p-weak upper gradient
of $\fh \in \Dp(X)$ and that
$\fh=f=\ft$ a.e.
Hence  $\fh=\ft$ q.e.,
and thus $f_j \to \ft$ q.e., as $j \to \infty$.
\end{proof}

\begin{lem} \label{lem-mazur-consequence}
Assume that\/ $1<p<\infty$ and that $f \in \Dp(X)$.
Assume further that $g_j$ is a \p-weak upper gradient of $u_j$,
$j=1,2,\ldots$,
and that both sequences\/  $\{u_j-f\}_{j=1}^\infty$ and\/ $\{g_j\}_{j=1}^\infty$
are bounded in $L^p(X)$.
Then there are functions $u$ and $g$ and 
 convex combinations
$v_j=\sum_{i=j}^{N_j} a_{j,i} u_i$
with \p-weak upper gradients $\gb_j=\sum_{i=j}^{N_j} a_{j,i} g_i$,
such that 
\begin{enumerate}
\item \label{mazur-a}
$u-f\in \Np(X)$ and $g\in L^p(X)$\textup{;}
\item \label{mazur-b}
both $v_j-u \to 0$ and $\gb_j \to g$ in $L^p(X)$, 
as $j \to \infty$\textup{;} 
\item 
$v_j \to u$ q.e., as $j \to \infty$\textup{;} 
\item \label{mazur-c}
$g$ is a \p-weak upper gradient of $u$.
\end{enumerate}
\end{lem}

\begin{proof}
Let $w_j=u_j-f$, $j=1,2,\ldots$\,.
Then $g_{w_j}\le g_j+g_f$ and $\{w_j\}_{j=1}^\infty$ is bounded in $\Np(X)$.
Since $L^p(X)$ is reflexive, its unit ball is weakly compact
(by Banach--Alaoglu's theorem) and thus there is a subsequence
of $\{w_j\}_{j=1}^\infty$ which converges weakly in $L^p(X)$.
Taking a subsequence of this subsequence and again using
Banach--Alaoglu's theorem we obtain a subsequence 
(again denoted $\{w_j\}_{j=1}^\infty$)
such that both 
$\{w_j\}_{j=1}^\infty$
and
$\{g_{j}\}_{j=1}^\infty$ converge weakly in $L^p(X)$
say to $w$ and $g$. 
As $g_j$, $j=1,2,\ldots$, are nonnegative we may choose $g$ nonnegative.

Applying Mazur's lemma 
(see, e.g., Yosida~\cite{yosida}, pp.\ 120--121),
repeatedly to the sequences
$\{w_i\}_{i=j}^{\infty}$, $j=1,2,\ldots$,  we find convex combinations
$w'_j=\sum_{i=j}^{N_j'} a'_{i,j} w_i$
such that $\|w'_j-w\|_{L^p(X)} < 1/j$.
Let $v'_j=w'_j+f=\sum_{i=j}^{N_j'} a'_{i,j} u_i$.
Then $g'_j:=\sum_{i=j}^{N_j'} a'_{i,j} g_{i}$
is a \p-weak upper gradient of $v'_j$.
Since moreover $g'_j \to g$ weakly in $L^p(X)$, as $j \to \infty$, 
we can again apply
Mazur's lemma (repeatedly) to obtain convex combinations
$v_j=\sum_{i=j}^{N_j} a_{i,j} u_i$
with \p-weak upper gradients $\gb_j=\sum_{i=j}^{N_j} a_{i,j} g_{i}$
such that $v_j- v \to 0$ and $\gb_j \to g$ in $L^p(X)$, as $j \to \infty$.
By Proposition~\ref{prop-Fuglede-consequence},
there is a function $u=v$ a.e.\ satisfying \ref{mazur-b}--\ref{mazur-c}.

As $g+g_f \in L^p(X)$ is a
\p-weak upper gradient of $u-f \in L^p(X)$, we see that $u-f \in \Np(X)$.
\end{proof}

\section{The variational capacity \texorpdfstring{$\cp$}{} on nonopen sets}
\label{app-capp}

In this appendix we define the variational capacity with respect to 
nonopen sets, which has been used to prove Adams' criterion in 
Section~\ref{sect-Adams}.
We also state those properties of the variational capacity that
we have needed in this paper. 
For proofs of 
Lemma~\ref{lem-Cp<=>cp} and Theorem~\ref{thm-cp}, 
and a considerably more extensive discussion,
we refer to 
Bj\"orn--Bj\"orn~\cite{BBvarcap}.

Let $E \subset X$ be a nonempty bounded set.

\begin{deff}
For an arbitrary set $A \subset E$ we define the \emph{variational capacity}
\begin{equation*} 
  \cp (A,E) =\inf   \int_X g_u^p \, d\mu,
\end{equation*}
where the infimum is taken over all $u\in \Np_0 (E) $ 
(extended by $0$ outside $E$) such that
$u \ge 1$ on $A$.
\end{deff}

The infimum can equivalently be taken over 
all nonnegative $u\in \Np_0 (E) $ such that
$ u = 1$ on $A$.
If $E$ is measurable we may also equivalently integrate over $E$ instead
of $X$.

Note that as $\Np_0(E)\subset\Np(X)$, it is natural to consider the minimal
\p-weak upper gradient $g_u$ with respect to $X$.
On the other hand, by Proposition~\ref{prop-min-grad},
$g_u=g_{u,E}$ in this case (if $E$ is measurable).

The variational capacity $\cp (A,E)$ has been used and studied earlier on 
metric spaces for 
bounded open $E$ in e.g.\ Bj\"orn--MacManus--Shanmugalingam~\cite{BMS}
and J.~Bj\"orn~\cite{BjIll}.
It can also be regarded as the condenser capacity 
$\cp (X\setm E,A,X)$, 
as in Definition~\ref{deff-cond-cap}.

We consider nonopen $E$, which is essential 
for Adams' criterion (Theorem~\ref{thm-Choq-int-est})
in the generality considered here.
The following two results are proved in Bj\"orn--Bj\"orn~\cite{BBvarcap}.

\begin{lem} \label{lem-Cp<=>cp}
Assume that $X$ supports a\/  $(p,p)$-Poincar\'e inequality for $\Np_0$
and that $\Cp(X \setm E)>0$.
Let $A \subset E$. Then
$\Cp(A)=0$ if and only if $\cp(A,E)=0$.
\end{lem}

\begin{thm} \label{thm-cp}
\quad 
\begin{enumerate}
\renewcommand{\theenumi}{\textup{(\roman{enumi})}}%
\item \label{cp-subset-sum}
  If $A_1 \subset A_2 \subset E$,
 then\/ $\cp(A_1,E) \le \cp(A_2,E)$\textup{;}
\item \label{cp-subadd-sum}
$\cp$ is countably subadditive,
i.e.\
if $A_1, A_2,\ldots \subset E$, then

  \[
      \cp\biggl(\bigcup_{i=1}^\infty A_i,E\biggr) 
          \le \sum_{i=1}^\infty \cp(A_i,E)
      \textup{;}
  \]
\item \label{cp-Choq-E-sum}
if\/ $1<p<\infty$ and  $A_1 \subset A_2 \subset \cdots \subset E$,  then
\[
      \cp\biggl(\bigcup_{i=1}^\infty A_i,E\biggr) 
          = \lim_{i \to \infty} \cp(A_i,E).
  \]
\end{enumerate}
\end{thm}

Let us observe the following more or less direct consequence
of Theorem~\ref{thm-fineint}.
We leave the proof to the reader.

\begin{prop}
Assume 
that $X$ is complete and supports a\/ $(1,p)$-Poincar\'e inequality,
that $\mu$ is doubling and that $p>1$.
Let $E \subset X$ be bounded and $A \subset E$.
Then
\[
     \cp(A,E)= \begin{cases}
        \cp(A \cap \fineint E, \fineint E), &
        \text{if } \Cp(A \setm \fineint E)=0, \\
        \infty, & 
        \text{if } \Cp(A \setm \fineint E)>0.
       \end{cases}
\]
\end{prop}

\end{document}